\documentclass[reqno,12pt]{amsart}
\usepackage{amsmath, latexsym, amsfonts, amssymb, amsthm, amscd}
\usepackage{multirow}
\usepackage{hyperref}

\usepackage{mathrsfs}
\usepackage{mathtools}
\usepackage[T1]{fontenc}
\usepackage{comment}
\usepackage{xcolor}

 \usepackage{graphicx}
\usepackage[utf8]{inputenc} 

\usepackage{datetime}

\usepackage{xparse}
\let\realItem\item 
\makeatletter
\NewDocumentCommand\myItem{ o }{%
   \IfNoValueTF{#1}%
      {\realItem}
      {\realItem[#1]\def\@currentlabel{#1}}
}
\makeatother
\usepackage{enumitem}
\setlist[enumerate]{
    before=\let\item\myItem,       
    label=\textnormal{(\arabic*)}, 
    widest=(2')                    
}

\usepackage{cancel}
\usepackage{graphics,epsf,psfrag}
\numberwithin{equation}{section}

\newtheorem{theorem}{Theorem}[section]
\newtheorem{lemma}[theorem]{Lemma}
\newtheorem{example}[theorem]{Example}

\newtheorem{cor}[theorem]{Corollary}
\newtheorem{rem}[theorem]{Remark}
\newtheorem{definition}[theorem]{Definition}

\newtheorem{assumption}[theorem]{Assumption}
\newtheorem{question}[theorem]{Question}
\newtheorem{conjecture}[theorem]{Conjecture}

\newcommand{\R}{\mathbb{R}}

\renewcommand{\tilde}{\widetilde}

\DeclareMathSymbol{\leqslant}{\mathalpha}{AMSa}{"36} 
\DeclareMathSymbol{\geqslant}{\mathalpha}{AMSa}{"3E} 
\DeclareMathSymbol{\eset}{\mathalpha}{AMSb}{"3F}     
\renewcommand{\leq}{\;\leqslant\;}                   
\renewcommand{\geq}{\;\geqslant\;}                   
\newcommand{\dd}{\,\text{\rm d}}             

\newcommand{\Div}{\mathrm{div}}
\newcommand{\mean}[1]{\langle #1 \rangle}

\newcommand{\norm}[1]{\left\lVert#1\right\rVert}
\newcommand{\abs}[1]{\left\lvert#1\right\rvert}
\newcommand{\seminorm}[1]{\left[#1\right]}

\title[On Young regimes for locally monotone SPDEs]{On Young regimes for locally monotone SPDEs}

\author{Florian Bechtold, J\" orn Wichmann}
\address[F. Bechtold]{Fakultät für Mathematik, Universität Bielefeld, 33501 Bielefeld, Germany}
\email{fbechtold@math.uni-bielefeld.de}
\address[J. Wichmann]{School of Mathematics, Monash University, Australia}
\email{joern.wichmann@monash.edu}
\begin{document}
\begin{abstract}
   We consider the following SPDE on a Gelfand-triple $(V, H, V^*)$: 
   \begin{equation} \nonumber
       \begin{split}
            \dd u(t)&=A(t, u(t)) \dd t+\dd I_t(u),\\
            u(0)&=u_0\in H.
       \end{split}
   \end{equation}
  Given certain local monotonicity, continuity, coercivity and growth conditions of the operator $A:[0, T]\times V\to V^*$ and a sufficiently regular operator $I$ 
  we establish global existence of weak solutions.
  
  In analogy to the Young regime for SDEs, no probabilistic structure is required in our analysis, which is based on a careful combination of monotone operator theory and the recently developed Besov rough analysis in \cite{FRIZbesov}. Due to the abstract nature of our approach, it applies to various examples of monotone and locally monotone operators~$A$, such as the $p$-Laplace operator, the porous medium operator, and an operator that arises in the context of shear-thickening fluids; and operators~$I$, including additive Young drivers $I_t(u) = Z_t-Z_0$, abstract Young integrals $I_t(u) = \int_0^t \sigma(u_s)\dd X_s$, and translated integrals $I_t(u) = \int_0^t b(u_s - w_s)\dd s$ that arise in the context of regularization by noise. In each of the latter cases, we identify corresponding noise regimes (i.e. Young regimes) that assure our abstract result to be applicable. In the case of additive drivers, we identify the Brownian setting as borderline, i.e. noises which enjoy slightly more temporal regularity are amenable to our completely pathwise analysis. 
\end{abstract}
\maketitle
\tableofcontents

\section{Introduction}

Going back to the works of Young \cite{young} and Kondurar \cite{kondurar}, it has been a well established fact that differential equations of the form 
\begin{equation}
    Y_t=Y_0+\int_0^t \sigma(Y_s) \dd X_s
    \label{young sde}
\end{equation}
are  well-posed for smooth $\sigma:\R^n\to L(\R^d, \R^n)$ provided $X\in C^\alpha_t\mathbb{R}^d$ for $\alpha>1/2$. It is therefore common to call this regularity class the Young regime\footnote{For a discuss of Young regimes on different regularity scales, see section \ref{besov sde}} associated to~\eqref{young sde}. In modern terms, this can be readily verified by an application of the sewing lemma, see for instance \cite[Section 8.3]{frizhairer}. 

As is also well known, sample paths of $d$-dimensional Brownian motion are almost surely $C^\alpha_t\R^d$ for $\alpha<1/2$, meaning their regularity is just shy of the Young regime. This means that additional probabilistic structure of Brownian motion needs to be exploited in order to obtain a well-posedness theory for \eqref{young sde}. In the case of It\^o calculus for example, one crucially makes use of the martingale property of Brownian motion. Let us stress that from a purely analytical point of view, one could argue that this additional probabilistic structure compensates for the lack of regularity of Brownian sample paths, which fall outside the Young regime for \eqref{young sde} almost surely. 

Recall finally that Itô calculus allows for an infinite dimensional generalization. Indeed, if $H$, $U$ are separable Hilbert spaces, $B$ is an $U$-cylindrical Wiener process, $Z\in L_2(U, H)$ and $W=ZB$, then the variational approach\footnote{refer to section \ref{variational spde lit}} to SPDEs allows for the study of abstract problems of the form 
\begin{equation}
    \dd u_t=A(t, u_t)\dd t+\dd W_t, \qquad u(0)= u_0\in H.
    \label{spde additive intro}
\end{equation}
Here $A:[0, T]\times V\to V^*$ is some (locally) monotone, hemicontinuous, coercive and bounded operator defined on a Gelfand triple $(V, H, V^*)$. Besides semi-linear problems, this abstract framework provides a well-posedness theory for $A$ being the $p$-Laplace or porous medium operator for example. Note that in this setting, the driving noise enjoys regularity $C^\alpha_tH$ for any $\alpha<1/2$ almost surely. 

On the background of all of the above, one could wonder if it is possible to still obtain well-posedness for \eqref{spde additive intro}, provided we replace the infinite-dimensional stochastic process $W$ by some generic path $X$, which posses higher temporal regularity. More precisely, we ask for a Young regime associated to \eqref{spde additive intro} in the following sense:
\begin{question} \label{question:AdditiveScale}
 Given a Gelfand triple $(V, H, V^*)$ and an operator $A:[0, T]\times V\to V^*$ that is (locally) monotone, hemicontinuous, coercive and bounded. 
 
Does there exists a regularity scale, i.e., a choice of parameters $\alpha \in( 1/2,1)$ and $p,q\in [1,\infty]$, such that for any $X\in B^\alpha_{p, q}( (0,T);H)$ the initial value problem 
\[
\dd u_t=A(t, u_t) \dd t+\dd X_t, \qquad u(0)=u_0\in H,
\]
is well-posed?
\end{question}
We give a full answer to Question~\ref{question:AdditiveScale} in Theorem~\ref{thm:AnwerAdditive}. A natural related question that we will address separately, concerns the multiplicative forcing setting, i.e. 

\begin{question} \label{question:Multiplicative-Young-Scale}
Given a Gelfand triple $(V, H, V^*)$, an operator $A:[0, T]\times V\to V^*$ that is locally monotone, hemicontinuous, coercive and bounded, and $\sigma : H\to H$ Lipschitz.

Does there exist a Banach space $E$ and a regularity scale, i.e., a choice of parameters $\alpha \in (1/2,1)$ and $p,q\in [1,\infty]$, such that for any $X\in B^\alpha_{p, q}( (0,T);E)$ the problem 
\[
\dd u_t=A(t, u_t) \dd t+\sigma(u_t) \dd X_t, \qquad u(0)=u_0\in H,
\]
is well-posed?
\end{question}
We will provide partial answers to this question in Theorems~\ref{thm:Existence-abstract-Young} and~\ref{thm:Multiplicative}. Let us give an overview of known results and contextualize the above questions in the related literature landscape.

\subsection{Young regimes for SDEs and semilinear SPDEs}
\label{besov sde}
Beyond the H\"older scale mentioned above, different regularity scales for driving noises $X$ that assure well-posedness of \eqref{young sde} have been investigated in the literature. \cite{Lejay2010} studies the problem on $p$-variation scale. The paper \cite{Prmel2016} can be considered a starting point in investigating \eqref{young sde} on Besov scales, while also addressing the 'rough regime'. In the general case of nonlinear smooth $\sigma$, \cite[Theorem 3.2]{Prmel2016} covers $X\in B^\alpha_{\infty, q}$ for $\alpha>1/2$ and $q\in [1, \infty]$. In \cite[Theorem 3.8]{Liu2021}, the authors are able to cover $X\in W^{\alpha, p}$ with $\alpha>1/2$ and $1/\alpha<p$. Finally, a completion of this program can be considered \cite[Theorem 4.2]{FRIZbesov}, which obtains well-posedness for $X\in B^\alpha_{p, q}$ provided either $\alpha>1/2$, $1/\alpha<p\leq \infty$ and $q\in (0, \infty]$ or $\alpha=1/2$, $2<p\leq \infty$ and $0<q\leq 2$. Since sample paths of Brownian motion are almost surely in $B^{1/2}_{p, \infty}$ for $p\in [1, \infty)$ but not in $B^{1/2}_{p, q}$ for $p, q\in [1, \infty)$ almost surely \cite[Theorem 4.1]{hoelderexponential}, we see that also moving to a finer Besov scale does not eliminate the need of harnessing Brownian motion's martingale structure in order to formulation a well-posedness theory for \eqref{young sde}.

A first step towards non-linear SPDEs is the understanding of semilinear SPDEs. These equations typically take the form
\begin{equation}
    \label{semi lin intro}
    \dd u_t=A u_t \dd t+f(u_t)\dd t+\sigma(u_t)\dd W_t, \qquad u(0)=u_0\in H,
\end{equation}
where $A$ is the generator of an analytic semigroup $S$ on a Hilbert space $H$, $B:H\to H$ is a non-linearity and $W$ is a spatially colored noise as above. The key feature of semilinear problems is that \eqref{semi lin intro} allows for a mild formulation 
\begin{equation}
    \label{mild form intro}
    u_t=S_tu_0+\int_0^t S_{t-s}f(u_s) \dd s+\int_0^tS_{t-s}\sigma(u_s) \dd W_s.
\end{equation}
This means that upon establishing adequate Schauder estimates and maximal inequalities for stochastic convolutions, \eqref{mild form intro} is essentially accessible by fixed point arguments. Some contributions in this direction include \cite{Flandoli1995-va,Gtarek1994,Brzeniak1999,Hofmanov2012Semi,DaPrato2013,Coupek2017,Gerasimovis2019, Bechtold2021}. A classical reference for problems of this type is given by \cite{Da_Prato1992-at}.


In the case of $W$ being replaced by a noise, which behaves as a $H$-fractional Brownian motion in time with $H>1/2$ equally colored in space, \eqref{mild form intro} has been studied in \cite{Maslowski2003} by means of fractional calculus and semigroup techniques. Upon replacing the infinite dimensional stochastic process $W$ with a generic path $X$ of higher temporal regularity, a first investigation of the associated 'Young regime' for \eqref{mild form intro} has been given by Gubinelli, Lejay and Tindel~\cite{Gubinelli2006}. In the case of the $1$-dimensional stochastic heat equation on $[0, 1]$ with Dirichlet boundary conditions, i.e.,
\[
\dd Y_t=\Delta Y_t \dd t+\sigma(Y_t)\dd X_t, \qquad Y_0=y_0,
\]
this yields on $H=L^2_x$ scale the regularity constraint $X\in C^{1/2+}_tL^2_x$ for $\sigma\in C^2_b$ \cite[Theorem 3]{Gubinelli2006}, which could be considered a corresponding Young regime in this setting. A further extension of these considerations to a 'rough regime' has been provided in \cite{gubinelli2010}.  Other contributions in this direction include \cite{Deya2011, Kuehn2020, Hesse2020,  Gerasimovis2021}.

We want to stress that the above results crucially make use of the semi-linear structure and the associated semi-group formulation and fixed point argument -- a strategy unavailable for SPDEs with non-linear leading order differential operator we are concerned with in Questions \ref{question:AdditiveScale} and \ref{question:Multiplicative-Young-Scale}.

\subsection{Monotone operators and evolution equations}
A different method to study evolution equations for a large class of non-linear differential operators consists in the monotone operator approach -- a vastly studied field since its introduction by Browder~\cite{MR156204} and Minty~\cite{MR169064}. Many authors contributed to the development of monotone and pseudo-monotone operator theory, e.g.,~\cite{MR0259693,MR0348562,10.32917/hmj/1206137159,Landes1980}. More references can be found, for example, in the books~\cite{MR1033498,Barbu2010,Showalter2013,Kaltenbach2023}.

We briefly sketch some conceptual ideas: For a separable, reflexive Banach space~$V$ and an operator~$A:V\to V^*$, the theorem of Browder and Minty asserts that the equation
\[
A(u)=b\in V^*,
\]
will always have a solution~$u\in V$, provided that $A:V\to V^*$ is (i) monotone, (ii) hemicontinous and (iii) coercive. Typical examples of monotone operators include the $p$-Laplace and porous medium operators.

Its extension to evolution equations requires working on a Gelfand triple $(V, H, V^*)$ and imposing additionally (iv) a growth condition; in this situation, it can be shown (see for example \cite[p.~319]{MR0259693}) that the associated initial value problem
\[
\partial_t u = A(u) , \qquad u(0)=u_0\in H,
\]
is well-posed in $C([0,T];H)\cap L^\alpha((0,T);V) \cap W^{1,\alpha'}((0,T);V^*)$, where $\alpha>1$ ($\alpha' = \alpha/(\alpha-1)$) quantifies the coercivity of~$A$.

The theory canonically extends to additive perturbations~$Z \in  W^{1,\alpha'}((0,T);V^*)$ and~$Z \in L^\alpha((0,T);V)$ by considering 
\begin{align*}
   \partial_t u= A(u) +\partial_t Z, \hspace{3em} \text{ and } \hspace{3em}  \partial_t v= A(v+Z),
\end{align*}
respectively. In the latter case, the solution~$u$ is reconstructed by the transformation~$u = v+Z$. Thus, additive perturbations are well understood for endpoint regularity assumptions, i.e.,  either the perturbation is temporally regular and spatially rough $Z\in  W^{1,\alpha'}((0,T);V^*)$, or temporally rough and spatially regular $Z\in L^\alpha((0,T);V)$. The latter setting has been investigated extensively in \cite{Gess2011} for example. But an understanding on intermediate scales -- below possessing a time derivative and taking values in $H$ -- is missing. This gap in the literature is addressed by Question \ref{question:AdditiveScale} and its answer in Theorem \ref{thm:AnwerAdditive}.

\subsection{The variational approach to SPDEs}
\label{variational spde lit}
Going back to the seminal works of Bensoussan and Temam~\cite{Bensoussan1972}, Pardoux~\cite{pardoux1975equation}, Krylov and Rozovskii~\cite{Krylov1981}, and G\"{o}ngy~\cite{Gyongy1982}, the variational approach to SPDEs consists in the study of evolutionary problems of the form 
\[
\dd u_t=A(t, u_t) \dd t+\sigma(t, u_t)\dd W_t, \qquad u(0)=u_0\in H,
\]
where $A$, $W$ and $\sigma$ are a monotone, hemicontinuous, coercive and bounded operator, a cylindrical Wiener process and a non-linear noise coefficient, respectively. Since then, it has been a very active field of research with many contributing authors, for example,~\cite{Goldys2009,Gess2012,Gess2014,Marinelli2018,Vallet2018,Ma2019,Scarpa2020,Scarpa2022}.

Starting from the foundational work of Liu and R\"ockner~\cite{Liu2010} (see also their book~\cite{Liu2015}), the monotonicity assumption has been relaxed and the concept of locally monotone operators in the context of SPDEs arose. It attracted much attention and several slightly different versions of local monotonicity and compatibility conditions on the operators~$A$ and~$\sigma$ have been proposed. This led to various generalizations, e.g., 
\begin{itemize}
    \item higher moments of solutions~\cite{Neelima2019,gnann2022higher};
    \item noise coefficients below Lipschitz-regularity~\cite{ZHANG2009,schmitz2024wellposedness};
    \item fully local monotone operators~\cite{röckner2022wellposedness,kumar2023wongzakai};
    \item Levy noise~\cite{Brzeniak2014,Nguyen2021};
    \item the critical variational framework~\cite{Agresti2024}.
\end{itemize}

\subsection{Abstract formulation of the problem}
Let us point out that a common feature of the results discussed in the last subsection is the usage of some martingale structure in $W$, allowing to first give meaning to the stochastic integral $\int_0^t \sigma(u_s) \dd W_s$ and secondly, to derive an Itô-formula for $\norm{u_t}_H^2$. This in turn provides crucial energy estimates similar to the ones encountered in the deterministic setting. Question \ref{question:Multiplicative-Young-Scale} therefore implicitly asks two questions: (i) If we replace the stochastic process $W$ with a generic path $X$, what regularity requirement on $u$ and $X$ do we need to impose in order for $I_t(u)=\int_0^t\sigma(u_s) \dd X_s$ to be still well-defined? (ii) Provided this expression is well-defined, what further requirements on $I$ ensure that we can solve the associated PDE? In order to decouple these two sub-questions, we isolate and abstract (ii) in the following way:
\begin{question} \label{question:Multiplicative-General-Scale}
Given a Gelfand triple $(V, H, V^*)$ and an operator $A:[0, T]\times V\to V^*$ that is locally monotone, hemicontinuous, coercive and bounded. 

Does there exist a regularity scale, i.e., a choice of parameters $\alpha_1 \in (0,1)$, $\alpha_2 \in (1/2,1)$ and $p_1,p_2,q_1,q_2\in [1,\infty]$, and class of operators
\[
\mathcal{I} \subset \{I: B^{\alpha_1}_{p_1,q_1}((0,T);H) \to B^{\alpha_2}_{p_2,q_2}((0,T);H) \}
\]
such that for any $I \in \mathcal{I}$ the problem 
\[
\dd u_t=A(t, u_t) \dd t+\dd I_t(u), \qquad u(0)=u_0\in H,
\]
is well-posed?
\end{question}

In this article, we provide an affirmative answer to the existence of weak solutions raised in Question~\ref{question:Multiplicative-General-Scale} for a specific class of abstract operators~$I$ -- Theorem~\ref{thm:main}. In short, weak solutions exist whenever the abstract operator
\begin{itemize}
    \item can be approximated by more regular integrals;
    \item satisfies a non-standard stability bound;
    \item and is continuous in a rather weak sense.
\end{itemize}

In a separate argument, we address (i) by providing conditions on $u, X, \sigma$, which ensure that $I_t(u)=\int_0^t\sigma(u_s) \dd X_s$ is well-defined. We also show that under these conditions, Theorem ~\ref{thm:main} is applicable thus giving a partial answer to Question \ref{question:Multiplicative-Young-Scale}. 

Thanks to our decoupling approach, Theorem \ref{thm:main} can also be applied in some regularization by noise settings by choosing $I_t(u)=\int_0^tb(u_s-w_s) \dd s$. Here, we assume $w$ to be a path with sufficiently regular local time\footnote{refer to the Appendix \ref{appendix}}. This allows us to extend a recent result~\cite{Bechtold2023} on regularisation by noise for the $p$-Laplace equation, Theorem~\ref{thm:reg-by-noise}.

\subsection{Structure of the paper}
In Section~\ref{sec:Main-results}, we establish the framework and introduce the main results of this article. We close the section by a discussion on possible extensions of the methods.

In Section~\ref{sec:Besov-rough}, we recall and extend the Besov rough path analysis of Friz, Seeger, and Zorin-Kranich~\cite{FRIZbesov}.

In Section~\ref{sec:Young-integration}, we use the Besov sewing lemma to construct the Young integral-- an integral that naturally generalizes Bochner integration and is of utmost importance in our analysis.

In Section~\ref{sec:proof-main}, we proof our main result -- Theorem~\ref{thm:main} -- on the existence of weak solutions to abstract integral equations, answering Question~\ref{question:Multiplicative-General-Scale}.

In Section~\ref{sec:proofs-young-regimes}, we use Theorem~\ref{thm:main} to prove existence and (sometimes) uniqueness of solutions to integral equations with specific integral operators, such as additive drivers and Young integrals.

In Section~\ref{sec:examples}, we apply our main results on concrete examples: the $p$-Laplace equation, the porous medium equation, and shear-thickening fluids. We establish existence and (sometimes) uniqueness. Moreover, we show that our general framework captures the effect of regularization by noise for the $p$-Laplace equation.

In the appendix, Section~\ref{appendix}, we collect useful tools.

\section{Setup and main results} \label{sec:Main-results}
Let $T>0$ be fixed. We write $f \lesssim g$ for two non-negative quantities $f$ and $g$ if $f$ is bounded by $g$ up to a multiplicative constant. Accordingly we define $\gtrsim$ and $\eqsim$. Moreover, we denote by $c$ and $C$ generic constants which can change their value from line to line. For $r\in [1, \infty]$, we denote by $r' = r/(r-1)$ its H\"older conjugate. Minimum and maximum are denoted by $\wedge$ and $\vee$, respectively. 

For a Banach space $\left(X, \norm{\cdot}_X \right)$, let $L^q(0,T;X)$, $q \in (0,\infty]$, be the space of Bochner-measurable functions $u: [0,T] \to X$ satisfying $t \mapsto \norm{u(t)}_X \in L^q(0,T)$. For $\alpha \in (0,1)$ and $q,p \in (0,\infty]$, we denote the Besov space by $B^{\alpha}_{p,q}(0,T;X)$; it is the quasi-Banach space of Bochner-measurable functions with finite quasi-norm given by
\begin{align*}
    \norm{u}_{B^\alpha_{p,q}(0,T;X)} &:= \norm{u}_{L^p(0,T;X)} + \seminorm{u}_{B^\alpha_{p,q}(0,T;X)}, \\
    \seminorm{u}_{B^\alpha_{p,q}(0,T;X)} &:= \left( \int_{0}^T \left( \int_{0}^{t-h} \left( \frac{\norm{u(t+h) - u(t)}_X}{ \abs{h}^\alpha} \right)^p \dd t\right)^{q/p} \frac{\dd h}{h} \right)^{1/q}.
\end{align*}
Moreover, $C([0,T];X)$ is the space of continuous functions with respect to the norm-topology. We also use $C^{\alpha}([0,T];X)$, $\alpha \in (0,1)$, for the space of $\alpha$-H\"older continuous functions. For $u\in C^{\alpha}([0,T];X)$, we denote by $\seminorm{u}_{C^{\alpha}([0,T]; X)}:=\sup_{s\neq t\in [0,T]}\frac{\norm{u_t-u_s}_X}{|t-s|^\alpha}$ and $\norm{u}_{C^{\alpha}([0,T];X)}=\sup_{t\in [0,T]}\norm{u_t}_x+\seminorm{u}_{C^{\alpha}([0,T]; X)}$ the corresponding semi-norm and norm, respectively. If there is no ambiguity, we abbreviate $ L^q_t X := L^q(0,T;X) $, $B^\alpha_{p,q}X = B^\alpha_{p,q}(0,T;X)$ and $C_t X = C([0,T];X)$. We moreover write $B^{\alpha+}_{p, q}X$ if $u\in B^{\alpha+\epsilon}_{p, q}X$ for some $\epsilon>0$ and $B^{\alpha-}_{p, q}X$ if $u\in B^{\alpha-\epsilon}_{p, q}X$ for any $\epsilon>0$.

If a Banach space $\left(X, \norm{\cdot}_X \right)$ embeds continuously into another Banach space $\left(Y, \norm{\cdot}_Y \right)$, we write $X\hookrightarrow Y$. 
If a sequence $(u_n)_n\subset X$ converges to $u\in X$ weakly, respectively, weakly star in a Banach space $\left(X, \norm{\cdot}_X \right)$, we write $u\rightharpoonup u$, respectively, $u\overset{*}{\rightharpoonup} u$.

\subsection{Assumptions}
Let $H$ be a separable Hilbert space with inner product $\left(\cdot, \cdot \right)_H$. Let $V$ be a reflexive Banach space such that $V \subset H$ continuously and densely. We write $\langle \cdot, \cdot \rangle_{V^*,V}$ for the duality pairing of $V^*$ and $V$. The triple $(V,H,V^*)$ is called \textit{Gelfand triple} and forms the basis of the variational approach.
\begin{assumption}[Compact Gelfand triple] \label{ass:compact-Gelfand}
We assume that $V \hookrightarrow H$ is compact.
\end{assumption}
\begin{assumption}[Monotone operator] \label{ass:monotone-operator}
We assume that $ A: [0,T] \times V \to V^*$ is $\mathcal{B}(0,T) \otimes \mathcal{B}(V)$-measurable and satisfies the following conditions:
\begin{enumerate}
    \item[(H1)] \label{it:H1} (Hemicontinuity) For all $u,v,w \in V$ and $t \in [0,T]$ the map
    \begin{align*}
        \mathbb{R} \ni \lambda \mapsto \langle A(t,u+\lambda v), w \rangle_{V^*,V}
    \end{align*}
    is continuous.
    \item[(H2)] \label{it:H2} (Local monotonicity) For all $u,v \in V$ and $t \in [0,T]$ it holds
    \[
    2\langle A(t, u)-A(t, v), u-v\rangle_V\leq (h_t+\eta(u))\norm{u-v}_H^2,
    \]
    where $h \in L^1(0,T)$ is non-negative, and $\eta:V\to [0, \infty)$ is measurable and locally bounded.
    \item[(H3)] \label{it:H3} (Coercivity) There exist $\alpha\in (1, \infty)$, $c_1\in (0, \infty)$, $c_2\in \mathbb{R}$ and $f\in L^1(0, T)$ such that for all $v \in V$ and $t \in [0,T]$
    \begin{align*}
        \langle A(t, u), u\rangle_{V^*,V}\leq -c_1\norm{u}_V^\alpha+c_2\norm{u}_H^2+f_t.
    \end{align*}
    \item[(H4)] \label{it:H4} (Boundedness) There exist $c_3 \in [0, \infty)$ and $g\in L^\frac{\alpha}{\alpha-1}(0, T)$ such that for all $v \in V$ and $t \in [0,T]$
    \begin{align*}
        \norm{A(t, v)}_{V^*}\leq g(t)+c_3\norm{v}^{\alpha-1}_V,
    \end{align*}
    where $\alpha$ is as in~\ref{it:H3}.
\end{enumerate}
\end{assumption}
Assumption~\ref{ass:monotone-operator} is standard in the variational framework for SPDEs, see e.g.~\cite{Liu2010,Liu2015}. It covers various operators, such as the $p$-Laplace operator, the porous medium operator, and an operator that arises in the context of shear-thickening fluids. We present these examples in Section~\ref{sec:examples}. 

\begin{assumption} \label{ass:integral-operator}
Let $q\in (2, \infty]$, $\gamma>1/2+1/q$ and 
\[
I: L^\infty(0,T; H)\cap B^{1/2}_{2, \infty}(0,T;H)\to \left\{ I \in B^{\gamma}_{q, \infty}(0,T;H): I_0 = 0 \right\}.
\]
We assume that $I$ can be approximated by a sequence of operators $\{I^n\}_{n \in \mathbb{N}}$ in the following sense: let $b^n : [0,T] \times H \to H$ be $\mathcal{B}(0,T) \otimes \mathcal{B}(H)$-measurable and there exist $C_n \in [0,\infty)$ such that for all $u,v \in H$ and $t \in [0,T]$
\[
    \norm{b^n(t, v)}_H\leq C_n(1+\norm{v}_H), \qquad \norm{b^n(t,u) - b^n(t,v)}_H \leq C_n \norm{u-v}_H.
    \]
We set $I^n_t(v) = \int_0^t b^n(s,v(s)) \dd s$ and further assume the following conditions to hold uniformly in $n \in \mathbb{N}$:
\begin{enumerate}
    \item[(H5)] \label{it:H5} (Time-local boundedness on $B^{\gamma}_{q,\infty}(0,T;H)$) There exist $\lambda : [0,\infty) \to [0,\infty)$ with $\lambda(0) = 0$ and continuous in $0$, and $c_4 \in [0,\infty)$ such that for all $t > s \in [0,T]$ and $u\in C([s,t];H) \cap B^{1/2}_{2, \infty}(s,t;H)$,
        \begin{align*}
            \seminorm{I^n(u)}_{B^{\gamma}_{q,\infty}(s,t;H)}^2 \leq c_4\left(1+\lambda(\abs{t-s})\left( \seminorm{u}_{B^{1/2}_{2, \infty}(s,t;H)}^2+\norm{u}_{L^\infty(s,t; H)}^2 \right) \right).
        \end{align*}
    \item[(H6)] \label{it:H6} (Continuity) There exists $\overline{\gamma} \in (1/2, \gamma]$ such that for all $v_n, v \in L^\infty(0,T;H)\cap B^{1/2}_{2, \infty}(0,T;H)$ with
    \begin{align*}
        \sup_{n \in \mathbb{N}} \norm{v_n}_{L^\infty(0,T;H)} +  \sup_{n \in \mathbb{N}} \norm{v_n}_{B^{1/2}_{2,\infty}(0,T;H)} < \infty
    \end{align*}
    and $v_n \rightarrow v \in L^2(0,T; H)$, implies $I^n(v_n) \rightarrow I(v) \in B^{\overline{\gamma}}_{2,\infty}(0,T;H)$.
\end{enumerate}
\end{assumption}
Before stating the main results, we present some examples of operators~$I$ satisfying Assumption~\ref{ass:integral-operator}. 
\begin{rem}
The operator~$I$ is a proxy for a generalized integral. By definition it is an extension of Bochner integration. Various choices are possible, e.g.:
\begin{itemize}
    \item (Young drivers) $I_t(u) \equiv Z_t - Z_0$;
    \item (Young integrals) $I_t(u) = \int_0^t \sigma(u_s) \dd X_s$;
    \item (Integrals of distributions) $I_t(u) = \int_0^t b(u_s - w_s) \dd s$.  
\end{itemize}
For a verification of the above choices in concrete examples we refer to Sections~\ref{sec:proofs-young-regimes} and~\ref{sec:examples}.

\end{rem}

To increase our familiarity with Assumption~\ref{ass:integral-operator}, let us discuss the regular situation at this stage:
\begin{example} \label{ex:regular-integral-operator}
 Let $I_t(u) = \int_0^t b(s,u_s) \dd s$, where $b$ satisfies the conditions of $b^n$ in Assumption~\ref{ass:integral-operator}. In this situation, defining an approximate integral is trivial: set~$b^n = b$ and, thus,~$I^n = I$. Fix $q \in (2,\infty]$ and $\gamma \in (1/2+1/q,1)$.

Ad~\ref{it:H5}: The fundamental theorem and the growth condition on $b$ imply
\begin{align*}
    \seminorm{I^n(u)}_{B^{\gamma}_{q,\infty}(s,t;H)} &\leq \abs{t-s}^{2(1 - (\gamma - 1/q))} \sup_{r \in [s,t]} \norm{b(r,u_r)}_H^2 \\
    & \leq  2 \abs{t-s}^{2(1 - (\gamma - 1/q))} C^2 (1+\norm{u}_{L^\infty(s,t; H)}^2).
\end{align*}
Thus,~\ref{it:H5} holds with $\lambda(r) = r^{2(1 - (\gamma - 1/q))}$ and $c_4 = 2 C^2 T^{2(1 - (\gamma - 1/q))}$.

Ad~\ref{it:H6}: Since $I^n = I$, the operator convergence $I^n \to I$ is trivial. It remains to verify that strong convergence of $v_n \to v \in L^2(0,T;H)$ with the additional information $v_n, v \in L^\infty(0,T;H)\cap B^{1/2}_{2, \infty}(0,T;H)$, implies strong convergence of $I(v_n) \to I(v) \in B^{\overline{\gamma}}_{2,\infty}(0,T;H)$ for some $\overline{\gamma} \in (1/2,\gamma]$. Fix $\ell \in (2,\infty)$ and $\overline{\gamma} \in (1/2, 1 - 1/\ell + 1/q)$.

The embedding $W^{1,\ell} \hookrightarrow B^{\overline{\gamma}}_{q,\infty}$ and the uniform Lipschitz-continuity of $b$ show
\begin{align*}
    \seminorm{I(v_n) - I(v)}_{B^{\overline{\gamma}}_{q,\infty}(0,T;H)} &\leq T^{1 - 1/\ell + 1/q - \overline{\gamma}} \seminorm{I(v_n) - I(v)}_{W^{1,\ell}(0,T;H)} \\
    &= T^{1 - 1/\ell + 1/q - \overline{\gamma}} \left( \int_0^T \norm{b(s,v_n(s)) - b(s,v(s))}_H^\ell \dd s \right)^{1/\ell} \\
    &\leq  T^{1 - 1/\ell + 1/q - \overline{\gamma}} C \left( \int_0^T  \norm{v_n(s) - v(s)}_H^\ell \dd s\right)^{1/\ell} \\
    &\leq T^{1 - 1/\ell + 1/q - \overline{\gamma}} C  \left( \norm{v_n}_{L^\infty_t H}+ \norm{v}_{L^\infty_t H} \right)^{1-2/\ell} \norm{v_n - v}_{L^2_t H}^{2/\ell}.
\end{align*}
This verifies~\ref{it:H6}.

\end{example}

\begin{rem}
We want to emphasize that regular integral operators~$I$, as presented in Example~\ref{ex:regular-integral-operator}, don't need to incorporate additional information on the Besov scale~$B^{1/2}_{2,\infty}(0,T;H)$. This changes for more rough integrals such as Young integration.
\end{rem}

\subsection{Main result: Existence of weak solutions}
In this section we define our notion of solution and present our main result.
\begin{definition}\label{def:weak-solution}
A function~$u \in C([0,T];H) \cap  B^{1/2}_{2, \infty}(0,T;H)\cap L^\alpha(0,T;V)$ is called weak solution to 
\begin{equation}
\begin{cases}
     \dd u_t+A(t, u_t) \dd t&=\dd I_t(u),\\
      \hfill u(0)&=u_0,
      \end{cases}
      \label{problem def}
\end{equation}
if for all $t \in [0,T]$ and $v \in V$ the following evolution equation is satisfied:
\begin{align} \label{eq:weak-solution}
    \left(u_t - u_0, v \right)_H = \int_0^t \langle A(s,u_s), v \rangle_{V^*,V} \dd s + \left( I_t(u), v \right)_H.
\end{align}
\end{definition}

The following result addresses existence of weak solutions in an abstract framework. Due to its generality, it is a powerful tool for the analysis of various choices of operators.
\begin{theorem}[Existence of weak solutions] \label{thm:main}
Let $T> 0$ and $u_0 \in H$. Moreover, let  
\begin{itemize}
    \item Assumption~\ref{ass:compact-Gelfand} be satisfied;
    \item $A$ satisfy Assumption~\ref{ass:monotone-operator};
    \item and $I$ satisfy Assumption~\ref{ass:integral-operator}.
\end{itemize}

Then there exists a weak solution to \ref{problem def} in the sense of Definition~\ref{def:weak-solution}. Moreover, this solution satisfies 
\begin{align} \label{eq:main-result-estimate}
\begin{aligned}
   \norm{u}_{C([0,T]; H)}^2 + \norm{u}_{B^{1/2}_{2,\infty}(0,T; H)}^2 &+ \norm{u}_{L^{\alpha}(0,T;V)}^\alpha \\
   &\leq C \left( \norm{u_0}_H^2 + \norm{f}_{L^1(0,T)} +\norm{g}_{L^{\alpha'}(0,T)}^{\alpha'} +1\right),
\end{aligned}
\end{align}
for a constant $C = C(c_1,c_2,c_3,c_4,\lambda,\alpha,q,\gamma,T) >0$. 
\end{theorem}
We postpone the proof of Theorem~\ref{thm:main} to Section~\ref{sec:proof-main}. However, we discuss the main steps at this point. This clarifies and motivates the new concepts needed for the existence theory of generic integral equations of the form~\eqref{eq:weak-solution}.

\subsubsection{Proof strategy and main ideas} \label{sec:proof-strategy}
The proof consists of $4$ steps:
\begin{itemize}
    \item construction of approximate solutions by choosing more regular operators~$I^n$;
    \item derivation of uniform a priori estimates on $B^{1/2}_{2,\infty} H \cap C_t H \cap L^\alpha V$;
    \item extraction of limits using compactness;
    \item identification of the limiting equation.
\end{itemize}

Replacing the abstract operator~$I$ by the more feasible and explicit integral~$I^n$ allows us to use standard monotone operator theory for the construction of an approximate solution $u^n$. A vital step towards a limit using compactness are uniform bounds on the approximate solution. In our situation, two regularity scales are of major importance: $C_t H \cap L^{\alpha}_t V$ and $B^{1/2}_{2,\infty} H$. Let us give the heuristic main ideas that go into two a priori bounds on these scales. In Assumption \ref{ass:monotone-operator}, we will restrict ourselves here to $h=\eta=c_2=f=g=0$ and $c_1=c_3=1$ for the sake of readability.

The first regularity scale requires us to find an expansion of the squared $H$-norm, i.e. a classical energy inequality. Formally, provided $I(u)$ is sufficiently regular -- the precise formula is stated in Theorem~\ref{chain rule norm} -- any solution~$u$ to~\eqref{eq:weak-solution} satisfies
\begin{equation}
      \dd \norm{u}_H^2 = 2 \langle A(t,u), u \rangle_{V^*,V} \dd t + 2\left( u, \dd I_t(u) \right)_H.
      \label{energy intro}
\end{equation}
Crucially, we need to make sense of the last term. If we assumed $I(u)\in C^1_tH$, this term would obviously make sense as 
\[
\int_0^t\left( u_s, \dd I_s(u) \dd s\right)_H:=\int_0^t \left( u_s, (\partial_t I(u))_s \right)_H \dd s
\]
However, we are interested in $I(u)$ possessing only some weaker time regularity on $H$-scale.  At this stage, the sewing lemma on Besov spaces distributes regularity requirements from the integrator $\dd I(u)$ to the integrand~$u$; in other words, 
\begin{align*}
     \int_0^t \left( u_s, \dd I_s(u) \right)_H = \mathscr{S}_t( u, I(u))
\end{align*}
can be constructed as a Young integral, whenever $I(u) \in C^{\gamma}H$ for some $\gamma>1/2$ provided we assume $u \in B^{1/2}_{2,\infty} H$, cf. Theorem~\ref{thm:young-integral}. Accordingly however, the Young integral $\mathscr{S}$ has to be controlled in these norms, meaning that the energy estimate formally becomes 
\begin{equation}
\norm{u}^2_{L^\infty_tH}+\norm{u}^\alpha_{L^\alpha_tV}\lesssim \norm{u_0}_H^2+\norm{u}_{B^{1/2}_{2, \infty}H}\norm{I(u)}_{B^\gamma_{q, \infty}H}.
    \label{energy intro 2}
\end{equation}
The corresponding rigorous statement can be found in Lemma \ref{lem:second-apriori-bound}. In order to control $I(u)$, \ref{it:H5} appears as a natural condition. Under this condition, we are left with the task of controlling $u$ on $B^{1/2}_{2,\infty} H$-scale. 

The Besov scale measures integrability of weighted time-differences. Locally, time-differences of the solution satisfy
\begin{align*}
    \norm{\dd u}_H^2 = \langle A(t,u), \dd u \rangle_{V^*,V} \dd t + \left( \dd u, \dd I_t(u) \right)_H.
\end{align*}

Let us for the moment assume that $A = 0$. In this situation, Cauchy--Schwarz's inequality implies $\norm{\dd u}_H \leq \norm{\dd I_t(u)}_H$, which shows that $u$ inherits any time-decay of $\dd I(u)$. Closing the estimate requires the generalized integral operator~$I$ to be sufficiently regularizing, e.g., for some $\alpha > 0$
\begin{align*}
    \norm{\dd I_t(u)}_H \leq  \abs{\dd t}^\alpha (\norm{\dd u}_H + 1).
\end{align*}
The precise regularization condition is formulated in~\ref{it:H5}. On short time-intervals, this implies the a priori bound
\begin{align*}
    \norm{\dd u}_H^2 \leq \frac{\abs{\dd t }^{2\alpha} }{1 - \abs{\dd t}^{2\alpha}}.
\end{align*}
Therefore, $\alpha$-regularization of $I$ transfers to $u$.

Going back to the general case, we need to cope with the additional difficulties imposed by the operator~$A$. Cauchy--Schwarz's and Young's inequalities imply
\begin{align*}
    \norm{\dd u}_H^2 \leq 2\langle A(t,u), \dd u \rangle_{V^*,V} \dd t + \norm{\dd I_t(u)}_H^2.
\end{align*}
Since we can't expect any time decay of $\dd u \in V$, the first term is of size~$\dd t$. Thus, even though $I$ might be $\alpha$-regularizing for some $\alpha > 1/2$, 
$u$ will no longer inherit $\alpha$-regularity but is restricted to $\alpha \wedge 1/2$ instead. Since we will always require $\alpha>1/2$ in order to make sense of the chain rule in \eqref{energy intro}, the optimal regularity we might expect of $u$ on $H$ scale is $1/2$. More precisely, the bound we obtain in Lemma \ref{lem:first-apriori-bound} reads in this formal setting 
\begin{align} \label{eq:B1/2H-estimate intro}
    \seminorm{u}_{B^{1/2}_{2, \infty} H}^2\lesssim \norm{u}_{L^\alpha_t V}^\alpha+  \seminorm{I(u)}_{B^{1/2}_{2, \infty}H}^2.
\end{align}

A crucial step in the proof of Theorem \ref{thm:main} is the combination of estimates for $\norm{\dd u}_H^2$ and $\dd \norm{u}_H^2$, i.e. \eqref{energy intro 2} and \eqref{eq:B1/2H-estimate intro} to derive stable a priori bounds, which couldn't be closed individually. The precise statement is given in Lemma \ref{lem:uniform-estimate}. 

\begin{rem}
We want to stress that the classical deterministic and stochastic situation, where $I(u) = \int b(u) \dd s$ and $I(u) = \int b(u) \dd W(s)$ for some Wiener process~$W$, respectively, require the estimate on~$\dd \norm{u}_H^2$ only. In this situations, increased time-regularity on Besov scales can be verified a posteriori and is not part of the construction. 
\end{rem}

Obviously, the above considerations in deriving \eqref{energy intro 2} and \eqref{eq:B1/2H-estimate intro} are only formal and need to be made rigorous through an approximation procedure. Once a priori bounds on $u^n$, a solution to an approximate problem driven by $I^n(u^n)$ are derived uniformly in $n$, one proceeds with a compactness argument. In order to be able to pass to the limit in $I$, we naturally require some form of continuity expressed in \ref{it:H6}.   The limit identification for the remaining parts is more delicate; we need to re-weight the evolution of the squared $H$-norm in order to use the local monotonicity condition~\ref{it:H2}. This requires an application of the newly developed chain rule for Young integral equations, Theorem~\ref{thm:Chain-rule-after}, to find the evolution of (for some fixed $\phi$)
\begin{align*}
 t \mapsto F(t,\norm{u_t}_H^2) =   \exp\left(\int_0^t h_s + \eta(\phi_s) \dd s \right) \norm{u_t}_H^2.
\end{align*}
Most importantly, $\partial_y F$ is a sufficiently regular multiplier, so that $\mathscr{S}( \partial_y F u^n, \dd I^n(u^n))$ and $\mathscr{S}( \partial_y F u, \dd I(u)) $ are well-defined Young integrals.

\begin{rem}
    The identification of the limit is considerably simpler if the local monotonicity is replaced by the stronger assumption:
    \begin{enumerate}
        \item[(H2')] \label{it:H2-alter} (Monotonicity) For all $t \in [0,T]$ and $u,v \in V$ it holds
        \begin{align*}
    \langle A(t, u)-A(t, v), u-v\rangle_{V^*,V}\leq 0.
\end{align*}
    \end{enumerate}
In this situation, $F(t,y) = y$ and the chain rule~\eqref{eq:Chain-rule} is trivial.
\end{rem}

\subsubsection{Sharpness of a chain rule}
As mentioned above in the proof strategy for Theorem \ref{thm:main}, one crucial component consists in the establishment of a chain-rule, allowing to deduce an expression for the squared $H$-norm of any solution \eqref{energy intro}. Let us give the precise statement thereof and discuss in what sense this result is sharp when compared to the setting of cylindrical Brownian motion.

\begin{theorem}
\label{chain rule norm}
Let $\gamma>1/2$, and $\alpha \in (1,\infty)$. 
    Suppose $X_0\in H$, $Y\in L^{\alpha'}_t V^*$, and $I\in C^\gamma_tH$ with $I_0 = 0$. Define the  $V^*$-continuous element
\begin{align} \label{eq:X-definition} 
    X_t:=X_0+\int_0^t Y_s \dd s+I_t.
\end{align}
    Assume that $X\in B^{1/2}_{2, \infty}H\cap L^\infty_tH\cap L^\alpha_t V$. Then $X$ is a continuous $H$-valued element for which we have 
\begin{align} \label{chain rule norm formula}
    \norm{X_t}_H^2=\norm{X_0}_H^2+2\int_0^t \langle Y_s, X_s\rangle_{V^*, V} \dd s+2\mathscr{S}_t(X, \dd I),
\end{align}
    where $\mathscr{S}_t(X, \dd I)=\int_0^t ( X_s, \dd I_s)_H$ denotes the Young integral constructed in Theorem~\ref{thm:young-integral}.
\end{theorem}
The proof is presented in Section~\ref{subsec:expanding-the-square}.

\begin{rem}
\label{chain rule sharp}
    Theorem \ref{chain rule norm} is sharp in the following sense: Consider an $U$-cylindrical Brownian motion~$B$ on $(\Omega, \mathcal{F}, \mathbb{P})$ and $Z\in L_2(U, H)$. By Kolmogorov's continuity theorem, we have $W=ZB\in C^\gamma_tH$ almost surely for any $\gamma<1/2$; in particular, $W$ doesn't fullfill the conditions stated Theorem~\ref{chain rule norm}. In this case for $X_0\in H$ and $Y\in L^{\alpha'}([0, T]\times \Omega, V^*)$ progressively measurable, we have for
    \[
    X_t=X_0+\int_0^t Y_s \dd s+W_t
    \]
    by It\^o's formula (see for example \cite[Theorem 4.2.5]{Liu2015}
    \begin{equation}
      \norm{X_t}_H^2=\norm{X_0}_H^2+2\int_0^t \langle Y_s, X_s\rangle_{V^*, V} \dd s+2\int_0^t (X_s, \dd W_s)_H+t\norm{Z}_{L_2(U, H)}^2
        \label{ito norm}
    \end{equation}
    where the integral $\int_0^t(X_s, \dd W_s)_H$ has to be understood as an infinite-dimensional It\^o integral. Note in particular the second order Itô correction term not present in \eqref{chain rule norm formula}.
\end{rem}

\subsection{Main result: Integral operators and Young regimes} \label{sec:main-integral-operator}
Theorem~\ref{thm:main} is a powerful tool towards answering Questions~\ref{question:AdditiveScale},~\ref{question:Multiplicative-Young-Scale} and~\ref{question:Multiplicative-General-Scale}. It identifies an abstract set of operators, i.e., operators that satisfy Assumption~\ref{ass:integral-operator}, such that existence of solutions is guaranteed, partially answering Question~\ref{question:Multiplicative-General-Scale}. In this section, we present various choices of integral operators that fall in our abstract framework:
\begin{itemize}
    \item additive Young drivers;
    \item abstract Young integrals;
    \item linear multiplicative spatially homogeneous noise;
    \item averages over oscillating paths (regularization by noise).
\end{itemize}
We identify situations that give rise to Young regimes; that is, an identification of the data regularity such that both, existence and uniqueness of solutions is valid. Existence boils down to the verification of Assumption~\ref{ass:integral-operator}. Uniqueness is more delicate and requires a case-by-case study. All proofs are postponed to Section~\ref{sec:proofs-young-regimes}.


\subsubsection{Additive Young regime} Additive drivers have one major advantage: differences of solutions to the same equation are no longer driven by the driver. This leads to the following result that answers Question~\ref{question:AdditiveScale}:
\begin{theorem} \label{thm:AnwerAdditive}
Let Assumption~\ref{ass:monotone-operator} be satisfied, $\gamma > 1/2$ and $Z \in C^\gamma_tH$. Define $I_t(u) \equiv Z_t - Z_0 $. 

Then there exists a unique weak solution to~\eqref{problem def} in the sense of Definition~\ref{def:weak-solution}. Moreover, any two weak solutions $u$ and $v$ of~\eqref{problem def} started in $u_0$ and $v_0$, respectively, satisfy for all $t\in [0,T]$
    \[
    \norm{u_t-v_t}_H^2\leq\norm{u_0-v_0}_H^2\exp{\left(\int_0^t h_s+\eta(u_s) \dd s \right)},
    \]
    where $h$ and $\eta$ are given by \ref{it:H2}. 
\end{theorem}

\begin{rem}
If the additive driver~$Z$ is a stochastic process with sample paths almost surely in $C^\gamma_tH$, the above path-wise stability result establishes path-by-path uniqueness. An example is space-colored cylindrical $h$-fractional Brownian motion~$W^h$ with Hurst parameter $h > 1/2$, i.e. 
\[
W^h_t=\sum_k \lambda_k e_k \beta^{k, h}_t
\]
where $(\lambda_k)_k\in \ell^2$, $(e_k)_k$ forms an orthonormal basis in $H$ and $(\beta^{k, h})_k$ is a sequence of independent $h$-fractional Brownian motions. Since $W^h \in C^{h-}_t H$, Theorem~\ref{thm:AnwerAdditive} is applicable. Note that for $h=1/2$, our Theorem is not applicable. Indeed, in this case, the variational approach to SPDEs provides a well-posedness theory, see for example \cite[Theorem 4.2.4]{Liu2015}. Let us stress again however, that our result does not rely on any probabilistic structure of the driver $Z$. 
\end{rem}

\subsubsection{Abstract Young integrals}
An answer of Question~\ref{question:Multiplicative-Young-Scale} involves the rigorous construction of an abstract Young integral:
\begin{align} \label{eq:abstract-young-integral-intro}
  I_t(u) =\int_0^t \sigma(u_s) \dd X_s.
\end{align}
Both, time and space regularity of integrand and integrator need to be compatible for the construction of~\eqref{eq:abstract-young-integral-intro}. We capture the spatial compatibility of the integrator by restricting to well-defined and stable pointwise multiplication: 
\begin{definition} \label{multiplicative ideal}
A Banach space $E\subset H$ is called multiplier in $H$ if the bilinear multiplication
    \begin{align*}
    \cdot :H\times E&\to H, \hspace{3em}        (h, e)\mapsto h\cdot e,
    \end{align*}
    is well-defined and there exists a constant $C>0$ such that for all $h \in H$ and $e \in E$ it holds $\norm{h\cdot e}_{H}\leq C \norm{h}_H\norm{e}_E$.

 A multiplier~$E$ in $H$ is called multiplicative ideal in $H$ if $C=1$.
\end{definition}
\begin{example}
An obvious example is given by $H = E=(\mathbb{R},\abs{\cdot})$. Other than that, the most typical example to have in mind is $H=L^2(\Lambda)$ and $E=L^\infty(\Lambda)$, where $\Lambda$ is a bounded domain. Further examples can be given by means of Theorem \ref{multiplication theorem}, for example, $H=W^{-1, 2}(\Lambda)$ and $E=C^{1+}(\Lambda)$.
\end{example}

The Banach space $E$ quantifies the spatial compatibility of abstract Young integration. This, together with appropriate temporal compatibility allows us to rigorously define~\eqref{eq:abstract-young-integral-intro} as done in Lemma~\ref{abstract young lemma}. Again, it is a consequence of the remarkable sewing lemma on Besov spaces. 

Towards an answer of Question~\ref{question:Multiplicative-Young-Scale}, we need further restrictions on the non-linearity~$\sigma$ and the path regularity of $X$:
\begin{assumption} \label{ass:Young-integral-condition}
Let $E$ be a multiplier in $H$, $\gamma > 3/4$ and $X \in C_t^\gamma E$. Moreover, let $\sigma: H\to H$ satisfy the following:
\begin{enumerate}
    \item (Linear growth) There exists a constant $C > 0$ such that for all $u \in H$
    \begin{align*}
        \norm{\sigma(u)}_H\leq C(1+\norm{u}_H);
    \end{align*}
    \item (Lipschitz continuity) There exists a constant $L > 0$ such that for all $u,v \in H$
    \begin{align*}
        \norm{\sigma(u)-\sigma(v)}_H\leq L \norm{u-v}_H.
    \end{align*}
\end{enumerate}
\end{assumption}
Assumption~\ref{ass:Young-integral-condition} guarantees that the Young integral is a valid integral operator in our abstract framework, which leads to a partial answer of Question~\ref{question:Multiplicative-Young-Scale}:
\begin{theorem} \label{thm:Existence-abstract-Young}
   Suppose Assumptions \ref{ass:compact-Gelfand}, \ref{ass:monotone-operator}  and~\ref{ass:Young-integral-condition} are satisfied.  Let $I_t(u) = \int_0^t \sigma(u_s) \dd X_s$ be defined by Lemma~\ref{abstract young lemma}.

    Then there exists a weak solution to~\eqref{problem def} in the sense of Definition~\ref{def:weak-solution}.
\end{theorem}
It remains open whether Assumption~\ref{ass:Young-integral-condition} is also sufficient for uniqueness of weak solutions.

\subsubsection{Linear multiplicative Young regime}
Restricting to linear multiplicative and spatially homogeneous noise ensures uniqueness:
\begin{theorem} \label{thm:Multiplicative}
 Suppose Assumptions \ref{ass:compact-Gelfand} and \ref{ass:monotone-operator} are met and let $A$ satisfy additionally~\ref{it:H2-alter}. Moreover, let $\gamma > 3/4$, $\beta\in C^\gamma_t \mathbb{R}$ and $I_t(u) = \int_0^t u_s \dd \beta_s$ be defined by Lemma~\ref{abstract young lemma}. 
 
 Then there exists a unique weak solution to~\eqref{problem def} in the sense of Definition~\ref{def:weak-solution}. Moreover, any two weak solutions $u$ and $v$ of~\eqref{problem def} started in $u_0$ and $v_0$, respectively, satisfy for all $t\in [0,T]$
    \[
    \norm{u_t-v_t}_H^2\leq\norm{u_0-v_0}_H^2\exp{(\beta_t-\beta_0)}.
    \]
\end{theorem}

\begin{rem}
    As in the previous case of additive noise, if $\beta$ is a stochastic process with sample paths $\beta\in C^\gamma_t\mathbb{R}$  for some $\gamma>3/4$ almost surely, then the above result implies path-by-path uniqueness. A particular example is given by an $\mathbb{R}$-valued fractional Brownian motion with Hurst parameter $h\in (3/4, 1)$.
\end{rem}
\subsubsection{Regularization by noise}
In our previous work \cite[Theorem 1.1]{Bechtold2023}, we investigated regularization by noise properties for the evolutionary $p$-Laplace problem, i.e.~\eqref{problem def} for $A$ being the $p$-Laplace operator (refer to Section~\ref{sec: p-laplace}) and $I_t(u)=\int_0^tb(u_s-w_s) \dd s$. Here $w:[0, T]\to \mathbb{R}$ is a continuous path that admits a sufficiently regular local time (refer to the Appendix \ref{appendix} and the corresponding subsection in \ref{sec: p-laplace} for an overview of the general regularization by noise philosophy in this setting). By means of our abstract approach in Theorem \ref{thm:main}, we are able to generalize these results in the following way.

\begin{assumption} \label{ass:Reg-by-noise}
Let $b \in \mathscr{S}'(\mathbb{R})$ be a Schwartz distribution, $w \in C([0,T];\mathbb{R})$ and $L^w$ be the local time of $w$. We assume the following to be true:
\begin{enumerate}
    \item There exists $\gamma > 3/4$ such that $b * L^w \in C^\gamma(0,T;C^{0,1}(\mathbb{R}))$;
    \item For all $n \in \mathbb{N}$ there exist $b^n: \mathbb{R} \to \mathbb{R}$ and constants $C_n >0$ such that:
\begin{itemize}
    \item For all $x \in \mathbb{R}$ it holds $ \abs{b^n(x)}\leq C_n(1+\abs{x})$;
    \item For all $x,y \in \mathbb{R}$ it holds $\abs{b^n(x)-b^n(y)}\leq C_n \abs{x-y}$;
    \item $\norm{b^n*L^w- b*L^w}_{C^{\gamma}(0,T; C^{0,1}(\mathbb{R}))} \leq 1/n$.
\end{itemize}
\end{enumerate}
\end{assumption}

\begin{theorem} \label{thm:reg-by-noise}
Let $\Lambda\subset \R^d$ be a bounded Lipschitz domain,  $T > 0$, $u_0 \in L^2(\Lambda)$ and let Assumption~\ref{ass:Reg-by-noise} be satisfied. Assume $p>\frac{2d}{d+2}$. 

Then there exists a weak solution
    \begin{align*}
        u \in C([0,T]; L^2(\Lambda)) \cap B^{1/2}_{2,\infty}(0,T; L^2(\Lambda)) \cap L^p(0,T; W^{1,p}_{0}(\Lambda))
    \end{align*}
    to
    \begin{align*}
        \partial_t u - \Delta_p u  =b(u-w), \qquad u(0) = u_0\in L^2(\Lambda),
    \end{align*}
    where we understand the right-hand side in the sense of Lemma \ref{verify nonlinear H5} as
    \[
    \int_0^tb(u_r-w_r)= \mathscr{I}_t \big( b * L^w (\mean{u}) \big).
    \]
\end{theorem}
\begin{rem}
    Compared with \cite[Theorem 1.1]{Bechtold2023} we do not require a condition on $b^2*L$ but only on $b*L$ in the above assumptions. This is because in our previous work, we employed a strong formulation approach, allowing us to work on a regularity scale $C^{1/2}_tL^2_x$. This however required us to test the equation by itself on differential level, making appear the $b^2$. In particular, Theorem \ref{thm:reg-by-noise} now allows us to treat even distributional $b$ (since $b^2$ no longer needs to be well-defined), provided the regularization effect of $L^w$ is sufficiently strong (see Example \ref{reg by noise example}). Moreover, we are able to weaken the assumption on the initial condition to $u_0\in L^2(\Lambda)$ compared with $u_0\in W^{1, p}_0(\Lambda)$ in \cite[Theorem 1.1]{Bechtold2023}. Finally, upon inspecting the proof, it can be seen that the statement generalizes to locally monotone operators $A$ satisfying Assumptions \ref{ass:compact-Gelfand} and  \ref{ass:monotone-operator}, provided $H=L^2(\Lambda)$. 
\end{rem}

\begin{example}
\label{reg by noise example}
  Let $b=\delta_0\in H^{-1/2-}$. Let $w$ be an $\mathbb{R}$-valued fractional Brownian motion with Hurst parameter $H<\frac{1}{6}$. By \cite[Theorem 3.4]{harang2020cinfinity}, almost any realization of $w$ admits a local time $L^w\in C^{3/4+}_tH^{3/2+}_x$ and thus $b*L\in C^{3/4+}_tC^{0, 1}_x$. Upon mollification of $\delta_0$, we see that Assumption \ref{ass:Reg-by-noise} is satisfied, i.e. we may conclude by Theorem \ref{thm:reg-by-noise} that the problem 
  \[
  \partial_tu-\Delta_pu=\delta_0(u-w), \qquad u(0)=u_0\in L^2(\Lambda),
  \]
  admits a solution. 
\end{example}

\subsection{Outlook}
A full identification of the Young regime for abstract Young integral equations involves establishing uniqueness of solutions. So far, we obtained uniqueness only in special cases: additive drivers and linear multiplicative spatial homogeneous noise. It remains open whether uniqueness of solutions holds in the general framework of Theorem~\ref{thm:main}.

Theorem~\ref{thm:Existence-abstract-Young} shows that locally monotone, abstract Young integral equations are solvable if the driving path satisfies $X \in C^{3/4+}_t E$. In comparison, the Young regime for SDEs is $X \in C^{1/2+}$. We believe that this gap is due to technical restrictions of our method and can eventually be removed. In a similar fashion, we believe that $Z \in B^{1/2+}_{2,\infty} H$ is sufficient for the well-posedness of additively driven, locally monotone equations, cf. Theorem~\ref{thm:AnwerAdditive}. More generally, we conjecture:
\begin{conjecture}
Let Assumptions~\ref{ass:compact-Gelfand} and~\ref{ass:monotone-operator} be satisfied. Moreover, let $\theta \in (0,1)$ and $Z \in B^{r_\theta+}_{\alpha_\theta,\infty} V_\theta$ with
\begin{align*}
    \frac{1}{\alpha_\theta} := \frac{1-\theta}{\alpha'} + \frac{\theta}{\alpha}, \qquad r_\theta := (1-\theta)\cdot 1 + \theta \cdot 0,
\end{align*}
where $\alpha$ is given in~\ref{it:H3} and $V_\theta$ is an interpolation space of $V$ and $V^*$.

Then 
\begin{align*}
    \dd u = A(t,u) \dd t + \dd Z, \qquad u(0) = u_0\in H,
\end{align*}
is well-posed.
\end{conjecture}

As already discussed in the introduction, classical monotone operator theory covers the endpoint cases $\theta \in \{0,1\}$, i.e., $Z\in W^{1,\alpha'}_t V^*$ and $Z\in L^\alpha_t V$. Our results contribute towards the choice $\theta = 1/2$. 
The general case $\theta \in (0,1)\backslash \{1/2\}$ requires an extension of the Besov sewing lemma to time and space; most importantly, $(u, \dd Z)$ needs to be well-defined and a new a priori bound on positive time regularity scale needs to be devised.   
\newline
Another possible research direction would be going beyond the Young, and into the 'rough' regime: Recall that in the SDE setting
\[
Y_t=Y_0+\int_0^t \sigma(Y_s) \dd X_s
\]
a pathwise solution theory can be established provided $X\in C^\alpha_t$ with $\alpha>1/3$ admits a so called rough path lift $\mathbf{X}=(X, \mathbb{X})$. Provided this lift is accordingly chosen, one is able to recover in this setting the classical Itô-theory of stochastic integration. A main advantage of this 'factorization procedure' (lifting the path in a first probabilistic step to a rough path and performing then a purely pathwise analysis) lies in the fact that stability estimates and thus path-by-path uniqueness, the Wong-Zakai theorem and large deviations principles essentially come for free. 
In the setting of locally monotone SPDEs, a natural question to address would be: 
\begin{question}
    Let Assumptions \ref{ass:compact-Gelfand} and \ref{ass:monotone-operator} be satisfied. Let $Z$ be a stochastic process such that almost surely $Z\in C^\gamma_tH$ with $\gamma\in (1/3, 1/2]$. Does there exists a rough path lift $\mathbf{Z}=(Z, \mathbb{Z})$ with $\mathbb{Z}\in C^{2\gamma}_t (H\otimes H)$, for which the problem 
    \[
    du=A(t, u)dt+\dd \mathbf{Z}, \qquad u(0)=u_0\in H,
    \]
    is well-posed?
\end{question}
As in the previous conjecture, this question might be also asked for intermediate spatial regularity scales, i.e. replacing various $H$'s with some interpolation space $V_\theta$. 

\section{Elements of Besov rough analysis and auxiliary results} \label{sec:Besov-rough}
For the convenience of the reader, let us recall the Sewing Lemma in a Besov space setting due to \cite{FRIZbesov} (see also \cite{gubi} \cite[Lemma 4.2]{frizhairer} for the more classical H\"older setting). We will restrict ourselves to the case of Nikolskii spaces sufficient for our purposes.

\subsection{The sewing lemma}
Let $E$ be a Banach space, $[0,T]$ a given interval. Let $\Delta_n$ denote the $n$-th simplex of $[0,T]$, i.e. $\Delta_n:\{(t_1, \dots, t_n)| 0\leq t_1\dots\leq t_n\leq T \} $. For measurable functions~$f:[0,T] \to E$ and $A:\Delta_2\to E$ we define $\delta f: \Delta_2 \to E$ and $\delta A: \Delta_3\to E$ by
\begin{align*}
    (\delta f)_{s,t}&:= f_t - f_s \hspace{2em} \text{ and } \hspace{2em}
    (\delta A)_{s,u,t}:=A_{s,t}-A_{s,u}-A_{u,t},
\end{align*}
respectively. 

Given $A$ and $\delta A$ as above, we define for $p\in (0, \infty]$ and $\tau\in (0, T]$
\[
\Omega_p(A, \tau):=\sup_{0\leq h\leq \tau} \left(\int_0^{T-h}\norm{A_{r, r+h}}_E^p \dd r\right)^{1/p},
\]
as well as 
\begin{equation*}
    \begin{split}
        \bar{\Omega}_p(\delta A, \tau):&=\sup_{\theta\in [0, 1]}\sup_{0\leq h\leq \tau}\left(\int_0^{T-h}\norm{(\delta A)_{r, r+\theta h, r+h}}_E^p \dd r\right)^{1/p}. 
    \end{split}
\end{equation*}
We further define for $\alpha > 0$
\begin{align*}
    \norm{A}_{\mathbb{B}^\alpha_{p, \infty}([0, T], E)} &:= \sup_{0\leq \tau\leq T} \tau^{-\alpha}\Omega_p(A, \tau), \\
    \norm{\delta A}_{\bar{\mathbb{B}}^\alpha_{p, \infty}([0, T], E)} &:=\sup_{0\leq \tau\leq T}\tau^{-\alpha}\bar{\Omega}_p(\delta A, \tau).
\end{align*}
Notice that
\begin{align} \label{eq:A-dom-deltaA}
\norm{\delta A}_{\bar{\mathbb{B}}^\alpha_{p,
\infty}} \lesssim \norm{A}_{\mathbb{B}^\alpha_{p,\infty}},
\end{align}
since $\bar{\Omega}_p(\delta A, \tau) \leq 3^{1/p \vee 1} \Omega_p(A, \tau)$.

Let $\mathcal{P} := \{0 = \tau_0 < \tau_1 < \ldots < \tau_N = 1 \}$ be a partition of $[0,1]$. For $(s,t) \in \Delta_2$ we set
\begin{align} \label{eq:sewing-germs}
    \mathscr{I}_\mathcal{P}A_{s,t}:= \sum_{i=1}^N A_{s + \tau_{i-1}(t-s), s+\tau_i(t-s) } \quad \text{ and } \quad \mathscr{R}_\mathcal{P}A := \mathscr{I}_\mathcal{P}A - A.
\end{align}

\begin{lemma}[Besov Sewing, \cite{FRIZbesov} Theorem 3.1]
\label{sewing}
Assume that $p_1, p_2 \in (0,\infty]$, $\alpha \in (0,1)$, $\gamma > 1 \vee 1/p_2$, and $A:\Delta_2\to E$ be measurable such that
\[
\norm{A}_{\mathbb{B}^\alpha_{p_1, \infty}([0, T], E)}+\norm{\delta A}_{\bar{\mathbb{B}}^\gamma_{p_2, \infty}([0, T], E)} < \infty.
\]
Then there exist $\mathscr{I}A \in B^\alpha_{p_1 \wedge p_2,\infty}([0,T];E)$ and $\mathscr{R}A \in \mathbb{B}^\gamma_{p_2,\infty}([0,T];E)$ such that
\begin{align}\label{eq:sewing-convergence}
    \lim_{\norm{\mathcal{P}} \to 0} \norm{\mathscr{I}_{\mathcal{P}}A - \delta \mathscr{I}A}_{\mathbb{B}^\gamma_{p_2,\infty}} =  \lim_{\norm{\mathcal{P}} \to 0} \norm{ \mathscr{R}_\mathcal{P}A -  \mathscr{R} A}_{\mathbb{B}^\gamma_{p_2,\infty}} = 0.
\end{align}
Moreover, we have the bounds
\begin{align}
    \norm{ \mathscr{R} A}_{\mathbb{B}^\gamma_{p_2,\infty}([0, T], E)} \lesssim_{p_2,\gamma} \norm{\delta A}_{\bar{\mathbb{B}}^\gamma_{p_2, \infty}([0, T], E)},
\end{align}
and
\begin{align} \label{est:sewed-seminorm}
    \seminorm{\mathscr{I}A}_{B^\alpha_{p_1 \wedge p_2,\infty}([0,T];E)} \lesssim_{T,p_1,p_2,\gamma} \norm{A}_{\mathbb{B}^\alpha_{p_1, \infty}([0, T], E)} + T^{\gamma - \alpha}\norm{\delta A}_{\bar{\mathbb{B}}^\gamma_{p_2, \infty}([0, T], E)}.
\end{align}
\end{lemma}

\subsection{Dominated convergence}
In order to perform our analysis, we need a result ensuring that the convergence $A^n\to A$ implies (under certain additional assumptions) the convergence $\mathscr{I}A^n\to \mathscr{I}A$. The sewing Lemma implies that such an additional assumption could consist in the convergence $\delta A^n\to \delta A$. However, we can relax this condition to $(\delta A^n)$ being uniformly bounded in certain spaces (Lemma~\ref{sewing convergence}). Towards this end, we need to first establish the following Lemma.

\begin{lemma}
    Suppose the setting of the previous Lemma~\ref{sewing}. Then, we have 
    \begin{equation}
           \Omega_{p_1\wedge p_2}(\delta \mathscr{I}A, h)\lesssim h^\alpha \norm{A}_{\mathbb{B}^\alpha_{p_1, \infty}([0, T], E)}+h^\gamma \norm{\delta A}_{\bar{\mathbb{B}}^\gamma_{p_2, \infty}([0, T], E)}.
           \label{finer sewing estimate}
    \end{equation}
    Moreover, we have for any $p\in (0, \infty]$ and $m\in \mathbb{N}$,
  \begin{equation}
          \Omega_{p}(\delta \mathscr{I}A, h)\leq m^{1 \vee 1/p} \, \Omega_{p}\left(\delta \mathscr{I}A, \frac{h}{m}\right).
          \label{scale estimate}
  \end{equation}
  \label{sewing convergence hilfe}
\end{lemma}
\begin{proof}
    Estimate \eqref{finer sewing estimate} essentially follows from the proof of the Besov sewing Lemma, see \cite[p.190]{FRIZbesov}. Concerning \eqref{scale estimate}, we have that 
    \begin{equation*}
        \begin{split}
               \Omega_{p}(\delta \mathscr{I}A, h)&=\sup_{0\leq \mu\leq h}\left(\int_0^{T-\mu} \norm{(\delta \mathscr{I}A)_{s, s+\mu}}_E^{p} \dd s\right)^{1/p}\\
               &=\sup_{0\leq \tau\leq \frac{h}{m}}\left(\int_0^{T-m\tau } \norm{(\mathscr{I}A)_{s, s+m\tau}}_E^{p}\dd s\right)^{1/p}
        \end{split}
    \end{equation*}
 Now remark that since $(\mathscr{I}A)_{s, s+m\tau}=\mathscr{I}A_{s+m\tau }-\mathscr{I}A_{s}$, i.e. the sewing is an increment, we have 
 \begin{equation*}
     \begin{split}
         \int_0^{T-m\tau } \norm{(\delta \mathscr{I}A)_{s, s+m\tau}}_E^{p}\dd s&=\int_0^{T-m\tau} \norm{\sum_{k=0}^{m-1}(\mathscr{I}A)_{s+(k+1)\tau}-(\mathscr{I}A)_{s+k\tau}}^{p}_E\dd s\\
         &\leq m^{(p-1)\vee 0}\sum_{k=0}^{m-1}\int_0^{T-m\tau }\norm{(\mathscr{I}A)_{r+(s+1)\tau}-(\mathscr{I}A)_{r+s\tau}}^{p}_E\dd s\\
         &=m^{(p-1)\vee 0}\sum_{k=0}^{m-1}\int_{k\tau}^{T-(m-k)\tau }\norm{(\mathscr{I}A)_{s+\tau}-(\mathscr{I}A)_{s}}^{p}_E\dd s\\
         &\leq m^{p \vee 1}\int_0^{T-\tau}\norm{(\mathscr{I}A)_{s+\tau}-(\mathscr{I}A)_{s}}^{p}_E\dd s
     \end{split}
 \end{equation*}
 and thus 
 \begin{equation*}
     \begin{split}
         \Omega_{p}(\delta \mathscr{I}A, h)&\leq m^{1 \vee 1/p} \sup_{0\leq \tau\leq \frac{h}{m}}\left(\int_0^{T-\tau}\norm{(\mathscr{I}A)_{s+\tau}-(\mathscr{I}A)_{s}}^{p}_E\dd s\right)^{1/p}\\
         &=m^{1 \vee 1/p} \, \Omega_{p}\left(\delta \mathscr{I}A, \frac{h}{m}\right).
     \end{split}
 \end{equation*}
\end{proof}

Using Lemma \ref{sewing convergence hilfe}, we can now proceed to the following extension of Lemma A.2 in \cite{Galeati2021} to the Besov setting we are concerned with. 

\begin{lemma}
\label{sewing convergence}
Let $p_1, p_2 \in (0,\infty]$, $\alpha \in (0,1)$, $\gamma > 1\vee 1/p_2$, and $A, A^n: \Delta_2 \to E$ be measurable and satisfy for some $R > 0$
\begin{align*}
    \sup_{n \in \mathbb{N}}\norm{A^n}_{\mathbb{B}^\alpha_{p_1, \infty}E}  + \norm{A}_{\mathbb{B}^\alpha_{p_1, \infty}E}+  \sup_{n\in \mathbb{N}}\norm{\delta A^n}_{\bar{\mathbb{B}}^\gamma_{p_2, \infty}} + \norm{\delta A}_{\bar{\mathbb{B}}^\gamma_{p_2, \infty}}\leq R.
\end{align*}
Then $\norm{A^n-A}_{\mathbb{B}^\alpha_{p_1, \infty}}\to 0$ implies $\seminorm{\mathscr{I}A-\mathscr{I}A^n}_{\mathbb{B}^\alpha_{p_1\wedge p_2, \infty}}\to 0$ as $n\to \infty$. Moreover, for $n\in \mathbb{N}$ sufficiently large, we have 
\[
\seminorm{\mathscr{I}A-\mathscr{I}A^n}_{\mathbb{B}^\alpha_{p_1\wedge p_2, \infty}}\lesssim (1+R)\norm{A^n-A}_{\mathbb{B}^\alpha_{p_1, \infty}}^{\frac{\gamma-1}{\gamma-\alpha}}.
\]
\end{lemma}
\begin{proof}
   Since the sewing map $\mathscr{I}$ is linear, assume without loss of generality that $A=0$. We will verify the assertion for $p = p_1 \wedge p_2 \in [1,\infty]$. The case $p \in (0,1)$ follows analogously. Combining the two estimates in Lemma \ref{sewing convergence hilfe}, we obtain
   \[
   \Omega_{p}(\delta \mathscr{I}A^n, h)\lesssim m^{1-\alpha}h^\alpha \norm{A^n}_{\mathbb{B}^\alpha_{p_1, \infty} }+m^{1-\gamma}h^\gamma \norm{\delta A^n}_{\bar{\mathbb{B}}^\gamma_{p_2, \infty}}.
   \]
   As we have by assumption 
   \[
   \frac{\norm{\delta A^n}_{\bar{\mathbb{B}}^\gamma_{p_2, \infty}}}{\norm{A^n}_{\mathbb{B}^\alpha_{p_1, \infty} }}\to \infty,
   \]
   we (for $n$ sufficiently big) choose $m \in \mathbb{N}$ such that $m^{1-\alpha}\simeq (\norm{\delta A^n}_{\bar{\mathbb{B}}^\gamma_{p_2, \infty}}/\norm{A^n}_{\mathbb{B}^\alpha_{p_1, \infty} })^\theta$. By assumption, this yields the estimate
   \[
      \Omega_{p}(\delta \mathscr{I}A^n, h)\lesssim (1+R)h^\alpha \norm{A^n}_{\mathbb{B}^\alpha_{p_1, \infty} }^{1-\theta}+(1+R)h^\gamma \norm{A^n}_{\mathbb{B}^\alpha_{p_1, \infty} }^{\theta \frac{\gamma-1}{1-\alpha}}.
   \]
   Choosing $\theta=(1-\alpha)/(\gamma-\alpha)$, deviding by $h^\alpha$ and taking the supremum over $h\in [0, T]$ then yields the claim. 
\end{proof}
In the classical H\"older setting, it is well known that if two germs $A, A'$ give rise to two integrals $\mathscr{I}A$, $\mathscr{I}A'$ via the sewing lemma and satisfy $|A_{s,t}-A'_{s,t}|\lesssim |t-s|^\beta$ for $\beta>1$, then $\mathscr{I}A=\mathscr{I}A'$. The following is a generalization of this to the Nikolskii setting. The lemma justifies the use of different local approximations that upon 'sewing' all yield the same integral. 
\begin{lemma}
\label{local approx doesnt matter}
Let $p_1, p_2 \in (0,\infty]$, $\alpha \in (0,1)$, $\gamma > 1\vee 1/p_2$, and $A, A': \Delta_2 \to E$ be measurable and satisfy 
\begin{align*}
   \norm{A'}_{\mathbb{B}^\alpha_{p_1, \infty}E}  + \norm{A}_{\mathbb{B}^\alpha_{p_1, \infty}E}+ \norm{\delta A'}_{\bar{\mathbb{B}}^\gamma_{p_2, \infty}} + \norm{\delta A}_{\bar{\mathbb{B}}^\gamma_{p_2, \infty}} < \infty.
\end{align*}
 Suppose furthermore that 
    \[
    \norm{A-A'}_{\mathbb{B}^\beta_{p_1, \infty}([0, T], E)}<\infty
    \]
    for some $\beta>1$. Then $\mathscr{I}A=\mathscr{I}A'$.
\end{lemma}
\begin{proof}
    By \eqref{finer sewing estimate} of Lemma \ref{sewing convergence hilfe} and the linearity of the sewing, we have that 
    \begin{align*}
    \Omega_{p_1\wedge p_2}(\delta \mathscr{I}(A-A'), h)&\lesssim h^\beta \norm{A-A'}_{\mathbb{B}^\beta_{p_1, \infty}([0, T], E)}\\&+h^\gamma (\norm{\delta A}_{\bar{\mathbb{B}}^\gamma_{p_2, \infty}([0, T], E)}+ \norm{\delta A'}_{\bar{\mathbb{B}}^\gamma_{p_2, \infty}([0, T], E)}).
    \end{align*}
   This readily implies
    \[
    \left(\int_0^{T-h}\norm{\mathscr{I}(A-A')_{r+h}-\mathscr{I}(A-A')_r}_E^{p_1\wedge p_2}dr\right)^{\frac{1}{p_1\wedge p_2}}\lesssim h^\beta+h^\gamma,
    \]
    from which we conclude that the weak derivative $\partial_t (\mathscr{I}(A-A'))\in L^{p_1\wedge p_2}([0, T], E)$ exists and is equal to zero. As $\mathscr{I}(A-A')_0=0$, this implies the claim. 
\end{proof}

\section{Young integration} \label{sec:Young-integration}
The sewing lemma is a powerful tool for the construction of a generalized integration theory. In this section, we define the Young integral $\mathscr{S}$ that can be written heuristically as
\begin{align} \label{eq:Young-integral}
  (u,I) \mapsto  \mathscr{S}_t(u,\dd I) = \int_0^t \left(  u_s, \dd I_s \right)_H. 
\end{align}
It is a non-symmetric bi-linear operator that intertwines the integrand~$u$ and the integrator~$I$ by sewing function values of $u$ and infinitesimal increments of $I$. We show that the Young integral generalises classical Bochner integration for sufficiently regular integrators and naturally occurs in the evolution equation for the squared $H$-norm, cf. Theorem~\ref{chain rule norm}. Finally, we derive a chain rule for sufficiently smooth transformations of integral equations that involve the Young integral.

\subsection{Young integrals}
The next theorem rigorously constructs~\eqref{eq:Young-integral} leveraging the sewing lemma on Besov spaces.

\begin{theorem}[Young integral] \label{thm:young-integral}
Let $\mu \in (0,\infty]$, and $p,q \in [\mu,\infty]$ such that $1/\mu = 1/p + 1/q$.
Moreover, let $\alpha \in (0,1)$ and $\beta > 0$ such that $\alpha + \beta > 1 \vee 1/\mu$. 

There exists a bi-linear operator $ \mathscr{S} :  B^{\alpha}_{p,\infty}H \times   B^{\beta}_{q,\infty}H \to B^{\beta}_{\mu,\infty}$ called Young integral such that
\begin{align}\label{eq:stability-Young-integral}
    \seminorm{\mathscr{S}(u,\dd I)}_{B^{\beta}_{\mu,\infty}} \lesssim_{T,\mu,\alpha,\beta} \norm{u}_{B^{\alpha}_{p,\infty}H} \seminorm{I}_{B^{\beta}_{q, \infty}H}. 
\end{align}
\end{theorem}
\begin{proof}
Let $u \in B^{\alpha}_{p,\infty}H$ and $I \in B^{\beta}_{q,\infty}H$. We construct~$\mathscr{S}$ as the sewing of the germ
\begin{align} \label{eq:germ}
 \Delta_2 \ni (s,t) \mapsto  (u,\dd I)_{s,t} := \left( u_s, I_t - I_s \right)_H .
\end{align}
Since we want to apply the Besov sewing lemma (Lemma~\ref{sewing}), we need to check that $(u,\dd I)$ is sew-able. We will show that
\begin{align} \label{eq:nec-bound-young-int}
   \norm{(u,\dd I)}_{\mathbb{B}^\beta_{\mu, \infty}}+\norm{\delta (u,\dd I)}_{\bar{\mathbb{B}}^{\alpha + \beta}_{\mu, \infty}} < \infty.
\end{align}

By H\"older's inequality
   \begin{align*}
    \left(\int_0^{T-h}\abs{(u,\dd I)_{r,r+h}}^\mu \dd r\right)^{1/\mu} &\leq \norm{u}_{L^p_t H}\left(\int_0^{T-h}\norm{I_{r+h}-I_r}_H^q \dd r\right)^{1/q}\\
            &\leq  \norm{u}_{L^p_t H}\seminorm{I}_{B^\beta_{q, \infty}H}h^{\beta}.
    \end{align*}
This implies $ \norm{(u,\dd I)}_{\mathbb{B}^\beta_{\mu, \infty}} \leq  \norm{u}_{L^p_t H}\seminorm{I}_{B^\beta_{q, \infty}H}$. 

Similarly, using H\"older's inequality and $\delta (u,\dd I)_{s, r, t} =( u_s-u_r,I_t-I_r)_H$,
\begin{align*}
    &\left(\int_0^{T-h}| \delta (u,\dd I)_{r, r+\theta h, r+h}|^{\mu} \dd r\right)^{1/\mu}\\
    &\hspace{2em} \leq \left(\int_0^{T-h} \norm{u_r - u_{r+\theta h}}_H^p \dd r\right)^{1/p}\left(\int_0^{T-h}\norm{I_{r+h}-I_{r+\theta h}}_H^q \dd r \right)^{1/q}\\
    &\hspace{2em}  \leq h^{\alpha + \beta}\seminorm{u}_{B^{\alpha}_{p, \infty}H}\seminorm{I}_{B^{\beta}_{q, \infty}H}.
\end{align*}
It follows $\norm{\delta (u,\dd I)}_{\bar{\mathbb{B}}^{\alpha + \beta}_{\mu, \infty}} \leq \seminorm{u}_{B^{\alpha}_{p, \infty}H}\seminorm{I}_{B^{\beta}_{q, \infty}H}$.

We have verified~\eqref{eq:nec-bound-young-int} and, thus, $\mathscr{S}(u, \dd I) := \mathscr{I}[ (u,\dd I)] \in B^\gamma_{\mu,\infty}$ is well-defined by Lemma~\ref{sewing}. Estimate~\eqref{eq:stability-Young-integral} follows from the stability of the sewing~\eqref{est:sewed-seminorm}. Indeed,
\begin{align*}
        \seminorm{\mathscr{I}[(u,\dd I)]}_{B^\alpha_{\mu,\infty}} &\lesssim_{T,\mu,\alpha,\beta} \norm{(u,\dd I)}_{\mathbb{B}^\alpha_{\mu, \infty}} + T^{\beta}\norm{\delta (u,\dd I)}_{\bar{\mathbb{B}}^{\alpha +\beta}_{\mu, \infty}} \\
        &\lesssim_{T,\mu,\alpha,\beta} \norm{u}_{B^{\alpha}_{p,\infty}H} \seminorm{I}_{B^{\beta}_{q, \infty}H}.
\end{align*}

The bi-linearity of $\mathscr{S}$ is inherited from linearity of the sewing~$\mathscr{I}$, and bi-linearity of the germ~$(u,I) \to (u,\dd I)$. 
\end{proof}


Next, we identify the Young integral as a classical Bochner integral provided the integrator~$I$ is sufficiently regular.

\begin{lemma} \label{lem:Identification-Young}
Let $u \in L^\infty_t H \cap B^{1/2}_{2,\infty} H$ and $I \in W^{1,\infty}_t H$. Then for all $t \in [0,T]$
\begin{align} \label{eq:Identification-Young}
    \mathscr{S}_t(u,\dd I) = \int_0^t \left(  u_r,\partial_t I_r \right)_H \dd r.
\end{align}
\end{lemma}
\begin{proof}
Let $\mu \in (1,2]$ and $q \in (2,\infty]$ such that $1/\mu = 1/2 + 1/q$.

Since $I \in W^{1,\infty}_t H \hookrightarrow B^{\gamma}_{r,\infty} H$ for all $\gamma \in (0,1]$ and $r \in [1,\infty]$, $u$ and $I$ satisfy the assumptions of Theorem~\ref{thm:young-integral} with $\alpha = 1/2$, $p = 2$, $\beta > 1/\mu$, $q = q$. Therefore, the left-hand side of~\eqref{eq:Identification-Young} is well-defined in $B^{\beta}_{\mu,\infty}$. Similarly, $r\mapsto \left( u_r,\partial_t I_r\right)_H \in L^\infty_t $, which makes the right-hand side of~\eqref{eq:Identification-Young} well-defined. It remains to establish the identity~\eqref{eq:Identification-Young}.

Let $(s,t) \in \Delta_2$ and $\mathcal{P} := \{0 = \tau_0 < \tau_1 < \ldots < \tau_N = 1 \}$ be a partition of $[0,1]$. The fundamental theorem implies, for $i =1,\ldots, N$,
\begin{align*}
    \frac{I_{s+\tau_{i}(t-s)} - I_{s+\tau_{i-1}(t-s)}}{(\tau_i - \tau_{i-1})(t-s)} = \int_0^1 \partial_t I_{s+\tau_{i}(t-s) + \theta((\tau_i - \tau_{i-1})(t-s))} \dd \theta=: \partial_t I_{\mathcal{P}}^i(s,t).
\end{align*}
This,~\eqref{eq:germ} and~\eqref{eq:sewing-germs} show
\begin{align*}
    &\mathscr{I}_{\mathcal{P}}[(u,\dd I)]_{s,t} - \int_s^t \left(  u_r,\partial_t I_r \right)_H \dd r  \\
    &= \sum_{i=1}^N \int_{s + \tau_{i-1}(t-s)}^{s+\tau_i(t-s)} \left( \left(u_{s+\tau_{i-1}(t-s)}, \frac{I_{s+\tau_{i}(t-s)} - I_{s+\tau_{i-1}(t-s)}}{(\tau_i - \tau_{i-1})(t-s)} \right)_H - (u_r, \partial_t I_r )_H \right) \dd r \\
    &= \int_{s}^{t} \left(u_{\mathcal{P}}^-(r;s,t), \partial_t I_{\mathcal{P}}^{\mathrm{av}}(r;s,t)   - \partial_t I_r \right)_H  + \left(u_{\mathcal{P}}^-(r;s,t) - u_r,  \partial_t I_r\right)_H  \dd r,
\end{align*}
where 
\begin{align*}
    u_{\mathcal{P}}^-(r;s,t) &:= \sum_{i=1}^N 1_{(s + \tau_{i-1}(t-s), s + \tau_{i}(t-s)]}(r)  u_{s + \tau_{i-1}(t-s)}, \\
    \partial_t I_{\mathcal{P}}^{\mathrm{av}}(r;s,t) &:= \sum_{i=1}^N 1_{(s + \tau_{i-1}(t-s), s + \tau_{i}(t-s)]}(r) \partial_t I_{\mathcal{P}}^i(s,t),
\end{align*}
 denote the piece-wise constant interpolation of nodal values and averaged nodal values,  respectively. H\"older's inequality allows us to derive
 \begin{align*}
    &\abs{\mathscr{I}_{\mathcal{P}_l}[(u,\dd I)]_{s,t} - \int_s^t \left(  u_r,\partial_t I_r \right)_H \dd r } \\
    &\quad \leq \norm{u}_{L^\infty_t H} \norm{\partial_t I_{\mathcal{P}_l}^{\mathrm{av}}(\cdot;s,t) - \partial_t I}_{L^1([s,t];H)}  + \norm{u_{\mathcal{P}_l}^-(\cdot;s,t) - u}_{L^1([s,t];H)} \norm{\partial_t I}_{L^\infty H}.
\end{align*}
 Similarly to~\cite[Lemma~4.2.6]{Liu2015} there exists a sequence of partitions~$\mathcal{P}_l$ such that $\norm{\mathcal{P}_l} \to 0$ and 
$\norm{u_{\mathcal{P}_l}^-(\cdot;s,t) - u}_{L^1([s,t];H)} + \norm{\partial_t I_{\mathcal{P}_l}^{\mathrm{av}}(\cdot;s,t) - \partial_t I}_{L^1([s,t];H)} \to 0$. This establishes the claim.
\end{proof}
Lemma~\ref{lem:Identification-Young} shows that Young integration is a consistent generalisation of classical Bochner integration. It relaxes the regularity conditions on the integrator, while strengthening the conditions on the integrand. Additionally, due to the identification~\eqref{eq:Identification-Young} it is possible to derive new stability results for classical Bochner integrals by exploiting the stability of the Young integral.

\begin{cor} \label{cor:weighted-time-derivative}
    Let $u\in L^\infty_t H \cap B^{1/2}_{2,\infty} H$, $I \in W^{1,\infty} H$ with $I_0 = 0$, and $\lambda \geq 0$. Moreover, let $q \in (2,\infty]$, $1/\mu = 1/2 + 1/q$ and $\gamma \in (1/\mu, 1)$.
    Then 
    \begin{align} \label{eq:weighted-time-derivative}
    \sup_{t \in [0,T]} \int_0^t e^{\lambda(t-s)} \left( \partial_s I_s, u_s \right)_H \dd s \lesssim_{\lambda,T,q,\gamma} \left( \norm{u}_{L^\infty_t H} + \seminorm{u}_{B^{1/2}_{2,\infty }H} \right)\seminorm{I}_{B^{\gamma}_{q,\infty} H}.
    \end{align}
\end{cor}
\begin{proof}
Let $(s,t) \in \Delta_2$. First, notice that 
    \begin{align*}
        \int_0^t e^{\lambda(t-s)} \left( \partial_s I_s, u_s \right)_H \dd s = e^{\lambda t} \left( \int_0^t  \left( \partial_s J_s, u_s \right)_H \dd s + \lambda \int_0^t  \left( J_s, u_s \right)_H \dd s \right),
    \end{align*}
    where $J_s := e^{-\lambda s} I_s $. 

Next, we show various bounds of $J$ in terms of $I$; clearly, $ \norm{J}_{L^\infty_t H} \leq \norm{I}_{L^\infty_t H}$. Moreover, $J_t - J_s = e^{-\lambda s} \left( (e^{-\lambda(t-s)} - 1) I_t + I_t - I_s \right)$, which implies 
    \begin{align*}
    \left( \int_0^{T-h} \norm{J_{t+h}-J_t}_H^q \dd t \right)^{1/q}      &\leq \left( \int_0^{T-h} \norm{ (e^{-\lambda h} - 1) I_{t+h} + I_{t+h} - I_t }_H^q \dd t \right)^{1/q} \\
    &\leq \abs{e^{-\lambda h} - 1} \norm{I}_{L^q_t H} + \left( \int_0^{T-h} \norm{I_{t+h} - I_t }_H^q \dd t \right)^{1/q} \\
    &\leq \lambda h \norm{I}_{L^q_t H} + h^{\gamma} \seminorm{I}_{B^\gamma_{q,\infty}H}.
    \end{align*}
Therefore, using $I_0 = 0$ and Lemma~\ref{lem:Shift-bound}, 
\begin{align} \label{eq:scaled-bound}
    \seminorm{J}_{B^{\gamma}_{q,\infty} H} \leq \left( \lambda T  C_{\gamma,q} +1 \right) \seminorm{I}_{B^{\gamma}_{q,\infty}H}.
\end{align}
Finally, $\partial_s J_s = -\lambda J_s + e^{-\lambda s} \partial_s I_s$, which shows $\norm{\partial_s J}_{L^\infty H} \leq \lambda \norm{I}_{L^\infty H} + \norm{\partial_t I}_{L^\infty_t H}$.

We are ready to show~\eqref{eq:weighted-time-derivative}; since $u \in L^\infty_t H \cap B^{1/2}_{2,\infty}H$ and $ J \in W^{1,\infty}_t H$, it is possible to apply Lemma~\ref{lem:Identification-Young}. Notice that $\mathscr{S}_0(u,J)= 0 = J_0$. Moreover, recall that $B^{\gamma}_{q,\infty} H\hookrightarrow B^{\gamma}_{\mu,\infty}H \hookrightarrow L^\infty_t H$, cf.~\cite[Theorem~16]{MR1108473}.  This and the stability of the Young integral~\eqref{eq:stability-Young-integral} imply
    \begin{align*}
        &e^{\lambda t} \left( \int_0^t  \left( \partial_s J_s, u_s \right)_H \dd s + \lambda \int_0^t  \left( J_s, u_s \right)_H \dd s \right) \\
        &\quad = e^{\lambda t} \left( \mathscr{S}_t(u,\dd J) + \lambda \int_0^t  \left( J_s, u_s \right)_H \dd s \right) \\
        &\quad \leq e^{\lambda T} \left( \norm{ \mathscr{S}(u,\dd J)}_{L^\infty_t} + \lambda T \norm{J}_{L^\infty_t H} \norm{u}_{L^\infty_t H} \right) \\
        &\quad \lesssim   \seminorm{ \mathscr{S}(u,\dd J)}_{B^{\gamma}_{\mu,\infty}} +  \seminorm{J}_{B^{\gamma}_{q,\infty} H} \norm{u}_{L^\infty_t H}  \\
        &\quad \lesssim   \left(\norm{u}_{L^\infty_t H} + \seminorm{u}_{B^{1/2}_{2,\infty }H} \right) \seminorm{J}_{B^{\gamma}_{q,\infty}H}.
    \end{align*}
Assertion~\eqref{eq:weighted-time-derivative} follows using~\eqref{eq:scaled-bound}, which completes the proof.
\end{proof}

\subsection{Proof of an import formula} \label{subsec:expanding-the-square}
In this section we prove the formula for the squared $H$-norm of solutions to generic integral equations.
\begin{proof}[Proof of Theorem~\ref{chain rule norm}]
The proof consists of three steps:
\begin{enumerate}
    \item \label{it:01} $X$ is uniformly continuous in time with respect to the weak topology in $H$;
    \item \label{it:02} $X$ satisfies the equation~\eqref{chain rule norm formula};
    \item \label{it:03} $X$ is continuous in time with respect to the strong topology in $H$.
\end{enumerate}

\underline{Step~\ref{it:01}:} Let $\varepsilon >0$ and $h \in H$. We need to show that there exists $\delta > 0$ such that $\abs{t-s} \leq \delta$ implies $\abs{\left( X_t - X_s,h \right)_H} \leq \varepsilon.$

Let us start with an estimate of $(X_t,v)$ for $v \in V$. Since $V \hookrightarrow H$ continuously and densely, and~\eqref{eq:X-definition}
\begin{align*}
    \left( X_t - X_s, v \right)_H &= \langle X_t - X_s, v \rangle_{V^*,V} \\
    &= \int_s^t \langle Y_r, v \rangle_{V^*,V} \dd r + \left( I_t - I_s, v \right)_H \\
    &\leq \left( \int_s^t \norm{Y_r}_{V^*} \dd r + \norm{I_t - I_s}_H \right) C\norm{v}_V \\
    &\leq \abs{t-s}^\beta C\norm{v}_V,
\end{align*}
where $\beta = 1/\alpha \wedge \gamma$.

Using the density of $V \subset H$ there exists $h_n \in V$ such that $\norm{h_n}_V \leq n$ and $\norm{h - h_n}_H \leq n^{-1} \norm{X}_{L^\infty H}^{-1}$. This allows us to conclude
\begin{align*}
    \abs{\left( X_t - X_s, h \right)_H} &\leq \abs{\left( X_t - X_s, h - h_n \right)_H }+ \abs{ \left( X_t - X_s, h_n \right)_H} \\
    &\leq 2/n + \abs{t-s}^\beta C n.
\end{align*}
Now, we first choose $n \in \mathbb{N}$ such that $n \geq 4/\varepsilon$; secondly, we choose $\delta > 0$ such that $\delta^\beta Cn \leq \varepsilon/2 $. This implies the uniform continuity with respect to the weak topology in $H$. 

\underline{Step~\ref{it:02}:}
Let $\mathcal{J} = \{ t\in[0,T]: \, X_t \in V\}$ be the set of regular times of $X$. Since $X \in L^\alpha_t V$, the complement $\mathcal{N} = [0,T] \backslash \mathcal{J}$ is a null set. Let $t,s \in \mathcal{J}$, consider increments of~\eqref{eq:X-definition}, and choose $v=X_t+X_s \in V$, yielding 
\begin{align*} \nonumber
    \norm{X_t}^2_H-\norm{X_s}_H^2&=(X_t-X_s, X_t+X_s)_H\\
    &=\underbrace{\int_s^t \langle Y_r, X_t + X_s\rangle_{V^*,V} \dd r}_{=: \mathrm{I}_{s,t}}+\underbrace{(I_t-I_s, X_t+X_s)_H}_{=: \mathrm{II}_{s,t}}.
\end{align*}
Let $\mathcal{P} := \{0 = \tau_0 < \tau_1 < \ldots < \tau_N = 1 \}$ be a partition of $[0,1]$.

We start with $\mathrm{I}$: clearly,
\begin{align*}
    \mathrm{I}_{s,t} - 2\int_s^t \langle Y_r, X_r \rangle_{V^*,V} \dd r =  \underbrace{\int_s^t \langle Y_r, X_t - X_r\rangle_{V^*,V} \dd r}_{\mathrm{R}^1_{s,t}} +  \underbrace{\int_s^t \langle Y_r, X_s - X_r \rangle_{V^*,V} \dd r}_{\mathrm{R}^2_{s,t}}.
\end{align*}
An application of the sewing~\eqref{eq:sewing-germs} implies
\begin{align*}
    \mathscr{I}_{\mathcal{P}} \mathrm{R}^1_{s,t}  &= \int_s^t \langle Y_r,X_{\mathcal{P}}^+(r;s,t) - X_r\rangle_{V^*,V} \dd r, \\
    \mathscr{I}_{\mathcal{P}} \mathrm{R}^2_{s,t}  &= \int_s^t \langle Y_r,X_{\mathcal{P}}^-(r;s,t) - X_r\rangle_{V^*,V} \dd r,
\end{align*}
where $X_{\mathcal{P}}^+(r;s,t) := \sum_{i=1}^N 1_{(s + \tau_{i-1}(t-s), s + \tau_{i}(t-s)]}(r)  X_{s + \tau_i(t-s)}$, respectively, $X_{\mathcal{P}}^-(r;s,t) := \sum_{i=1}^N 1_{(s + \tau_{i-1}(t-s), s + \tau_{i}(t-s)]}(r)  X_{s + \tau_{i-1}(t-s)}$ denote the piece-wise constant interpolation of the nodal values to the left, respectively, right interval. Similarly to~\cite[Lemma~4.2.6]{Liu2015} there exists a sequence of partitions~$\mathcal{P}_l$ such that $\norm{\mathcal{P}_l} \to 0$ and 
$\norm{X_{\mathcal{P}_l}^+(\cdot;s,t) - X}_{L^\alpha([s,t];V)} + \norm{X_{\mathcal{P}_l}^-(\cdot;s,t) - X}_{L^\alpha([s,t];V)} \to 0$. Therefore,
\begin{align*}
    &\abs{(\mathscr{I}_{\mathcal{P}_l} \mathrm{I})_{s,t} -  2\int_s^t \langle Y_r, X_r \rangle_{V^*,V} \dd r } \\
    &\hspace{2em} \leq \norm{Y}_{L^{\alpha'}_t V^*} \left( \norm{X_{\mathcal{P}_l}^+(\cdot;s,t) - X}_{L^\alpha([s,t];V)} + \norm{X_{\mathcal{P}_l}^-(\cdot;s,t) - X}_{L^\alpha([s,t];V)}\right),
\end{align*}
which implies $\lim_{\norm{\mathcal{P}_l} \to 0}(\mathscr{I}_{\mathcal{P}_l} \mathrm{I})_{s,t}$ exists and equals $ 2\int_s^t \langle Y_r, X_r \rangle_{V^*,V} \dd r$.

Next, we investigate~$\mathrm{II}$: a trivial expansion shows
\begin{align*}
    \mathrm{II}_{s,t} - 2(I_t-I_s, X_s)_H = (I_t-I_s, X_t-X_s)_H =:\mathrm{K}_{s,t}.
\end{align*}
Due to the regularity assumption $I \in C^\gamma_t H$ and $X \in B^{1/2}_{2,\infty}H \cap L^\infty_t H$, we can apply Theorem~\ref{thm:young-integral}. This ensures the existence of $\lim_{\norm{\mathcal{P}}\to 0} \mathscr{I}_{\mathcal{P}}[(X,\dd I)] = \mathscr{S}(X,\dd I)$. It remains to show that~$\mathrm{K}$ is sew-able and decays sufficiently fast for $\abs{t-s} \ll 1$, which implies that its sewing vanishes. 

Let $h \in (0,T)$. H\"older's inequality implies
\begin{align*}
    \left( \int_0^{T-h} \abs{\mathrm{K}_{t,t+h}}^2 \dd t \right)^{1/2} &\leq  h^{\gamma + 1/2} \seminorm{I}_{C^\gamma_t H} \seminorm{X}_{B^{1/2}_{2,\infty}H}.
\end{align*}
This together with~\eqref{eq:A-dom-deltaA} ensure
\begin{align*}
    \norm{ \delta \mathrm{K}}_{\bar{\mathbb{B}}^{\gamma + 1/2}_{2,\infty}} + \norm{\mathrm{K}}_{\mathbb{B}^{\gamma + 1/2}_{2,\infty}} \lesssim \seminorm{I}_{C^\gamma_t H} \seminorm{X}_{B^{1/2}_{2,\infty}H}.
\end{align*}
Therefore, Lemma~\ref{local approx doesnt matter} implies $\lim_{\norm{\mathcal{P}} \to 0} \mathscr{I}_{\mathcal{P}} \mathrm{K} = 0$.

In total, we have verified that~\eqref{chain rule norm formula} holds for all $t,s \in \mathcal{J}$. Next, we argue that~\eqref{chain rule norm formula} holds for all $[t,s] \in [0,T]$. Notice that
\begin{align*}
    t\mapsto \int_0^t \langle Y_r, X_r \rangle _{V^*,V} \dd r \in W^{1,1}_t \hookrightarrow C_t,
\end{align*}
and
\begin{align*}
    t \mapsto \mathscr{S}_{t}(X, \dd I) \in B^{\gamma}_{2,\infty} \hookrightarrow C^{\gamma - 1/2}_t.
\end{align*}
Thus, the right-hand side of~\eqref{chain rule norm formula} is continuous in time, so must be its left-hand side.

\underline{Step~\ref{it:03}:} It remains to show the continuity of $X$ in the strong topology of $H$. Let $t \in [0,T]$ and $\varepsilon > 0$. Due to~\eqref{chain rule norm formula}, the map $t \mapsto \norm{X_t}_H^2$ is continuous, i.e., there exists $\delta > 0$ such that $\abs{t-s} \leq \delta$ implies $\abs{\norm{X_t}_H^2 - \norm{X_s}_H^2} \leq \varepsilon$. Additionally, by Step~\ref{it:01}, $X$ is uniformly continuous with respect to the weak topology, i.e., for all $h \in H$ (in particular $h = X_t$) there exists $\tilde{\delta} > 0$ such that $\abs{t-s} \leq \tilde{\delta}$ implies $\abs{(X_t - X_s, h)_H} \leq \varepsilon$. Choosing $\overline{\delta} = \delta \wedge \tilde{\delta}$ and $s \in[0,T]$ such that $\abs{s-t} \leq \overline{\delta}$ guarantee
\begin{align*}
    \norm{X_t - X_s}_H^2 =  \norm{X_s}_H^2- \norm{X_t}_H^2  +  2\left( X_t -  X_s, X_t \right)_H \leq 2 \varepsilon.
\end{align*}
Thus, we have verified~\ref{it:03}.
\end{proof}

\subsection{Chain rule for Young integral equations}
In order to cope with locally monotone operators, it is crucial to have access to a chain rule for Young integral equations. 

\begin{theorem} \label{thm:Chain-rule-after}
Let $q \in (2,\infty]$, $\gamma>1/2+1/q$, and $1/\mu = 1/2 + 1/q$. Moreover, let $b \in L^1_t$, $u \in   L^\infty_t H \cap B^{1/2}_{2,\infty}H$, $ I \in B^{\gamma}_{q,\infty}H$, and
    \begin{align} \label{eq:y-expansion}
   \forall t \in [0,T]: \quad     y_t = y_0 + \int_0^t b_s \dd s + \mathscr{S}_t( u,\dd I).
    \end{align}
Let $F \in W^{1,1}\big([0,T] \times \mathbb{R}\big)$ satisfy the following: for all $R>0$ there exists $C_R >0$ such that
\begin{subequations} \label{eq:localBounded}
\begin{align} \label{eq:localBounded-l1}
\int_0^T \sup_{z \in B_R} \abs{\partial_t F(t,z)} \dd t \leq C_R, \\
\label{eq:localBounded-linf}
   \sup_{t \in [0,T]} \sup_{z \in B_R}  \abs{\partial_y F(t,z)}\leq C_R, \\ \label{eq:localBounded-lip-space}
 \sup_{t \in [0,T]} \sup_{z_1,z_2 \in B_R} \frac{\abs{\partial_y F(t,z_1) - \partial_yF(t,z_2)}}{\abs{z_1 - z_2}}\leq C_R, \\ \label{eq:localBounded-lip-time}
 \sup_{h \in [0,T]} h^{-1/2}\left(  \int_0^{T-h}   \sup_{z \in B_R} \abs{\partial_y F(t+h, z) - \partial_y F(t, z)}^2 \dd t \right)^{1/2} \leq C_R.
\end{align}
\end{subequations}
Then the following chain rule is satisfied
    \begin{align} \label{eq:Chain-rule}
    \begin{aligned}
    \forall t \in [0,T]: \quad    &F(t,y_t) = F(0,y_0) \\
    &+ \int_0^t \partial_t F(s,y_s) \dd s +  \int_0^t \partial_y F(s,y_s) b_s \dd s + \mathscr{S}_t( \partial_y F(y) u, \dd I),
    \end{aligned}
    \end{align}
    where $t \mapsto \mathscr{S}_t( \partial_y F(y) u, \dd I) \in B^{\gamma}_{\mu,\infty}$ denotes the Young integral constructed in Theorem~\ref{thm:young-integral}.
\end{theorem}
\begin{proof}
The proof consists of three steps:
\begin{enumerate}
    \item \label{it:chain-first} the convex hull of $y([0,T])$ is contained in a ball, which allows us to use the local assumption~\eqref{eq:localBounded} on $F$;
    \item \label{it:chain-second} increments of $t \mapsto F(t,y_t)$ poses an alternative representation using a finite partition~$\mathcal{P}$;
    \item \label{it:chain-third} terms in the new representation are identified for partitions with vanishing mesh-size.
\end{enumerate}

\underline{Step~\ref{it:chain-first}:} First, we notice that $t\mapsto y_t$ is continuous as the right-hand side of~\eqref{eq:y-expansion} is continuous. Moreover, using $B^{\gamma}_{\mu,\infty} \hookrightarrow C^{\gamma - 1/\mu}$ and~\eqref{eq:stability-Young-integral},
  \begin{align*}
      \sup_{t \in [0,T]} \abs{y_t} &\leq \abs{y_0} + \norm{b}_{L^1_t} + \sup_{t\in [0,T]} \abs{\mathscr{S}_{t}( u, \dd I)} \\
      &\lesssim \abs{y_0} + \norm{b}_{L^1_t} + \left( \norm{u}_{L^\infty_t H} +\norm{u}_{B^{1/2}_{2,\infty}H} \right) \seminorm{I}_{B^{\gamma}_{q,\infty}H}< \infty.
  \end{align*}
Therefore, it exists $R > 0$ such that $\mathrm{conv}( y([0,T]) ) \subset B_R$, where $\mathrm{conv}( y([0,T]) )$ denotes the convex hull of the $y$-image of~$[0,T]$. 

\underline{Step~\ref{it:chain-second}:} Due to the fundamental theorem, for all $(s,t) \in \Delta_2$,
  \begin{align*}
      &F(t,y_t) - F(s,y_s) = F(t,y_t) - F(s,y_t) + F(s,y_t) -  F(s,y_s) \\
      &= \int_0^1 \partial_t F( s + \theta(t-s),y_t )\dd \theta \,  (t-s)  + \int_0^1 \partial_y F(s, y_s + \theta(y_t - y_s) ) \dd \theta \,(y_t - y_s).
  \end{align*}
Recall that, using~\eqref{eq:y-expansion}, 
  \begin{align*}
      y_t - y_s = \int_s^t b_r \dd r + \mathscr{S}_t( u, \dd I) - \mathscr{S}_s( u, \dd I).
  \end{align*}
Therefore, introducing the shorten notation $\mathscr{S}_{\cdot} = \mathscr{S}_{\cdot}( u, \dd I)$,
\begin{align}
    \begin{aligned} \label{eq:pre-sewing}
     F(t,y_t) - F(s,y_s)  &= \underbrace{\int_0^1 \partial_y F(s, y_s + \theta(y_t - y_s) ) \dd \theta \,\int_s^t b_r \dd r }_{=: \mathrm{I}_{s,t}}  \\
     &\quad + \underbrace{\int_0^1 \partial_y F(s, y_s + \theta(y_t - y_s) ) \dd \theta \, \left( \mathscr{S}_t - \mathscr{S}_s\right)}_{=: \mathrm{II}_{s,t}} \\
     &\quad + \underbrace{\int_0^1 \partial_t F( s + \theta(t-s), y_t )\dd \theta \,  (t-s)}_{=: \mathrm{III}_{s,t}}.
     \end{aligned}
\end{align}
Let $\mathcal{P} := \{0 = \tau_0 < \tau_1 < \ldots < \tau_N = 1 \}$ be a partition of $[0,1]$. Sewing~\eqref{eq:pre-sewing} along the finite partition~$\mathcal{P}$ -- in other words, we apply the linear operator $\mathscr{I}_\mathcal{P}$ defined in~\eqref{eq:sewing-germs} to~\eqref{eq:pre-sewing} -- yields
\begin{align*}
    F(t,y_t) - F(s,y_s) = (\mathscr{I}_{\mathcal{P}} \mathrm{I} )_{s,t} + (\mathscr{I}_\mathcal{P} \mathrm{II})_{s,t} +  (\mathscr{I}_{\mathcal{P}} \mathrm{III} )_{s,t}.
\end{align*}

\underline{Step~\ref{it:chain-third}:} 
The chain rule~\eqref{eq:Chain-rule} follows provided we show, for all $(s,t) \in \Delta_2$,
\begin{subequations}\label{eq:conv-both}
\begin{align} \label{eq:conv-classic}
  \lim_{\norm{\mathcal{P}} \to 0} (\mathscr{I}_{\mathcal{P}} \mathrm{I} )_{s,t}  &= \int_s^t \partial_y F(y_r) b_r \dd r, \\ \label{eq:conv-sewing}
   \lim_{\norm{\mathcal{P}} \to 0}  (\mathscr{I}_{\mathcal{P}} \mathrm{II} )_{s,t} &= \mathscr{S}_t(\partial_y F(y)u, \dd I) - \mathscr{S}_s(\partial_y F(y)u, \dd I), \\ \label{eq:conv-time}
    \lim_{\norm{\mathcal{P}} \to 0}  (\mathscr{I}_{\mathcal{P}} \mathrm{III} )_{s,t} &= \int_s^t \partial_t F(r,y_r) \dd r,
\end{align}
\end{subequations}
 We will discuss each case separately. 

\underline{\eqref{eq:conv-classic}:} By definition of the sewing along a finite partition, cf.~\eqref{eq:sewing-germs},
\begin{align*}
    &(\mathscr{I}_{\mathcal{P}} \mathrm{I} )_{s,t} - \int_s^t \partial_y F(r,y_r) b_r \dd r =  \int_{s}^{t} \mathscr{H}_{\mathcal{P}}(r;s,t) b_r \dd r,
\end{align*}
where
\begin{align*}
    &\mathscr{H}_{\mathcal{P}}(r;s,t) \\
    &= \int_0^1 \Big( \partial_y F( s_{\mathcal{P}}^-(r;s,t) , y_{\mathcal{P}}^-(r;s,t) + \theta(y_{\mathcal{P}}^+(r;s,t) - y_{\mathcal{P}}^-(r;s,t)) ) - \partial_y F(r,y_r) \Big) \dd \theta,
\end{align*}
and
\begin{subequations} \label{eq:interpolations}
\begin{align}    
    s_{\mathcal{P}}^-(r;s,t) &= \sum_{i=1}^N 1_{[s + \tau_{i-1}(t-s),s + \tau_{i}(t-s))}(r) \left(s + \tau_{i-1}(t-s) \right), \\
    y_{\mathcal{P}}^-(r;s,t) &=  \sum_{i=1}^N 1_{[s + \tau_{i-1}(t-s),s + \tau_{i}(t-s))}(r) y_{s + \tau_{i-1}(t-s)}, \\
    y_{\mathcal{P}}^+(r;s,t) &=  \sum_{i=1}^N 1_{[s + \tau_{i-1}(t-s),s + \tau_{i}(t-s))}(r) y_{s + \tau_{i}(t-s)}.
\end{align}
\end{subequations}
Continuity implies $y_{\mathcal{P}}^\pm(r;s,t) \rightarrow y_r$ and $s_{\mathcal{P}}^-(r;s,t) \rightarrow r$ as $\norm{\mathcal{P}} \to 0$. It follows that $\mathscr{H}_{\mathcal{P}}(r;s,t) \rightarrow 0$ as $\norm{\mathcal{P}} \to 0$. Moreover,~\eqref{eq:localBounded-linf} shows 
\begin{align*}
   \mathscr{H}_{\mathcal{P}}(r;s,t) \leq 2 \sup_{t \in [0,T]} \sup_{z \in B_R} \abs{\partial_y F(t,z)} < \infty.
\end{align*}
This and $b \in L^1_t$ allow us to use dominated convergence to conclude
\begin{align*}
   \lim_{\norm{\mathcal{P}}\to 0} \abs{(\mathscr{I}_{\mathcal{P}} \mathrm{I} )_{s,t} - \int_s^t \partial_y F(r,y_r) b_r \dd r} \leq \lim_{\norm{\mathcal{P}}\to 0} \int_s^t \abs{\mathscr{H}_{\mathcal{P}}(r;s,t)} \abs{b_r} \dd r = 0.
\end{align*}

\underline{\eqref{eq:conv-sewing}:} 
Recall that Young integration -- as a sewing -- is locally well-approximated by its germ~\eqref{eq:germ}. Indeed, using~\eqref{eq:sewing-germs}, 
\begin{align*}
    \mathscr{S}_t &- \mathscr{S}_s  - (u,\dd I)_{s,t} \\
    &= \left( \mathscr{S}_t - \mathscr{S}_s - \mathscr{I}_{\mathcal{P}}[(u,\dd I)]_{s,t} \right) + \left( \mathscr{R}_{\mathcal{P}}[(u,\dd I)]_{s,t} - \mathscr{R}[(u,\dd I)]_{s,t} \right) + \mathscr{R}[(u,\dd I)]_{s,t},
\end{align*}
where the right-hand side vanishes if sewed along partitions with vanishing mesh-size, cf.~\eqref{eq:sewing-convergence}. This enables the alternative representation of~$ \mathscr{I}_{\mathcal{P}} \mathrm{II}$,
\begin{align*}
    \mathscr{I}_{\mathcal{P}} \mathrm{II}_{s,t} - \mathscr{I}_{\mathcal{P}}\left[ \partial_y F(y) (u,\dd I)\right]_{s,t} &= \sum_{i=1}^4  \mathscr{I}_{\mathcal{P}} \left[ \mathrm{R}^i \right]_{s,t},
\end{align*}
where $\partial_y F(y) (u,\dd I)_{s,t} := \partial_y F(s,y_s) (u_s,I_t - I_s)_H$ and
\begin{align*}
    \mathrm{R}^1_{s,t} &=  \int_0^1 \left( \partial_y F(s, y_s + \theta(y_t - y_s) ) - \partial_y F(s,y_s) \right) \dd \theta \, \left( \mathscr{S}_t - \mathscr{S}_s \right), \\
    \mathrm{R}^2_{s,t} &= \partial_y F(s,y_s)\left( \mathscr{S}_t - \mathscr{S}_s - \mathscr{I}_{\mathcal{P}}[(u,\dd I)]_{s,t} \right), \\
     \mathrm{R}^3_{s,t} &= \partial_y F(s,y_s)\left( \mathscr{R}_{\mathcal{P}}[(u,\dd I)]_{s,t} - \mathscr{R}[(u,\dd I)]_{s,t} \right), \\
      \mathrm{R}^4_{s,t} &= \partial_y F(s,y_s) \mathscr{R}[(u,\dd I)]_{s,t}.
\end{align*}
The assertion~\eqref{eq:conv-sewing} follows provided that 
\begin{enumerate}[label=(\Alph*)]
    \item \label{it:01-chain} $\partial_y F(y) (u,\dd I)$ defines a sew-able germ,
    \item \label{it:02-chain} and the sewing of $\mathrm{R}^i$ vanishes for all $i=1,\ldots, 4$.
\end{enumerate}

\underline{\ref{it:01-chain}:} 
We start by showing the sew-ability of $\partial_y F(y) (u,\dd I)$; that is, we need to verify the conditions of Lemma~\ref{sewing}. Let $h \in (0,T)$. Clearly,
\begin{align*}
  \left(  \int_0^{T-h} \abs{\partial_y F(y) (u,\dd I)_{t,t+h}}^{\mu} \dd t \right)^{1/\mu} \leq  \sup_{t \in [0,T]} \sup_{z \in B_R} \abs{\partial_y F(t,z)} \norm{u}_{L^\infty_t H} \seminorm{I}_{B^{\gamma}_{\mu,\infty}H} h^\gamma,
\end{align*}
which implies $\norm{\partial_y F(y) (u,\dd I)}_{\mathbb{B}^\gamma_{\mu,\infty}} \lesssim \norm{u}_{L^\infty_t H} \seminorm{I}_{B^{\gamma}_{q,\infty}H}$, since $B^{\gamma}_{\mu,\infty} \hookrightarrow B^{\gamma}_{q,\infty}$. Next, we decompose the $\delta$-operator of the germ and estimate each term separately,
\begin{align*}
     &\left(  \int_0^{T-h} \abs{(\delta \partial_y F(y) (u,\dd I) )_{t,t+\theta h, t+h}}^{\mu} \dd t \right)^{1/\mu} \\
     &= \left(  \int_0^{T-h} \abs{ \left( \partial_y F(t+\theta h, y_{t+\theta h})u_{t+\theta h} - \partial_y F(t, y_{t})u_{t},I_{t+h} - I_t\right)_H}^{\mu} \dd t \right)^{1/\mu} \\
     &\lesssim \left(  \int_0^{T-h} \abs{ \left( \partial_y F(t+\theta h, y_{t+\theta h}) - \partial_y F(t+\theta h, y_{t}) \right) \left(  u_{t+\theta h} ,I_{t+h} - I_t\right)_H}^{\mu} \dd t \right)^{1/\mu} \\
     &\quad + \left(  \int_0^{T-h} \abs{ \left( \partial_y F(t+\theta h, y_{t}) - \partial_y F(t, y_{t}) \right) \left(  u_{t+\theta h} ,I_{t+h} - I_t\right)_H}^{\mu} \dd t \right)^{1/\mu} \\
     &\quad + \left(  \int_0^{T-h} \abs{  \partial_y F(t, y_{t})\left( u_{t+\theta h} - u_t ,I_{t+h} - I_t\right)_H}^{\mu} \dd t \right)^{1/\mu}\\
     &=: \mathrm{J}_1 + \mathrm{J}_2 + \mathrm{J}_3.
\end{align*}

$\mathrm{J}_1$ is controlled using the uniform local Lipschitz-continuity of $\partial_y F$ in its second argument~\eqref{eq:localBounded-lip-space}, and~\eqref{eq:y-expansion}
\begin{align*}
    \mathrm{J}_1 &\lesssim \norm{u}_{L^\infty_t H} \left(  \int_0^{T-h} \left( \abs{y_{t+h} - y_t} \norm{ I_{t+h} - I_t}_H \right)^{\mu} \dd t \right)^{1/\mu} \\
    &\lesssim \norm{u}_{L^\infty_t H} \left(  \int_0^{T-h} \left( \int_t^{t+h} \abs{b_r} \dd r \norm{ I_{t+h} - I_t}_H \right)^{\mu} \dd t \right)^{1/\mu}\\
    &\quad + \norm{u}_{L^\infty_t H} \left(  \int_0^{T-h} \left(\abs{\mathscr{S}_{t+h} -\mathscr{S}_{t} }  \norm{ I_{t+h} - I_t}_H \right)^{\mu} \dd t \right)^{1/\mu}\\
    &=: \mathrm{J}_1^a + \mathrm{J}_1^b.
\end{align*}
Estimating $\mathrm{J}_1^a$: Using H\"older's inequality with $a = 2/\mu \in (1,\infty]$ and $a' = 2/(2-\mu)$, noticing that $a' \mu = q$, and Fubini's theorem
\begin{align*}
    \mathrm{J}_1^a &\lesssim  \norm{u}_{L^\infty_t H} \norm{b}_{L^1_t}^{1/2} \left(  \int_0^{T-h} \left( \int_t^{t+h} \abs{b_r} \dd r  \right)^{\mu/2} \norm{ I_{t+h} - I_t}_H^{\mu} \dd t \right)^{1/\mu} \\
    &\lesssim  \norm{u}_{L^\infty_t H} \norm{b}_{L^1_t} \norm{ I}_{B^{\gamma}_{q,\infty}H} h^{1/2 + \gamma}.
\end{align*}
Clearly, $1/2 + \gamma > 1 + 1/q > 1$.

Estimating $\mathrm{J}_1^b$: Let $\theta = q/(\mu +q) \in (0,1]$. Using H\"older's inequality with $a = 1/\theta$ and $a' = 1/(1-\theta)$, observing that $a' \theta \mu =q$, and employing the embedding $B^{\gamma}_{q,\infty} \hookrightarrow B^{\gamma}_{\mu,\infty} \hookrightarrow L^\infty$,
\begin{align*}
    \mathrm{J}_1^b &\lesssim  \norm{u}_{L^\infty_t H} \norm{\mathscr{S}}_{L^\infty_t}^{1-\theta} \norm{I}_{L^\infty_t H}^{1-\theta} \left(  \int_0^{T-h} \left(\abs{\mathscr{S}_{t+h} -\mathscr{S}_{t} }  \norm{ I_{t+h} - I_t}_H \right)^{\theta \mu} \dd t \right)^{1/\mu} \\
    &\leq \norm{u}_{L^\infty_t H} \norm{\mathscr{S}}_{L^\infty_t}^{1-\theta} \norm{I}_{L^\infty_t H}^{1-\theta} \left(  \int_0^{T-h} \abs{\mathscr{S}_{t+h} -\mathscr{S}_{t} }^{\mu} \dd t \right)^{\theta /\mu} \left(  \int_0^{T-h}  \norm{ I_{t+h} - I_t}_H^{a' \theta \mu} \dd t \right)^{\theta/q}  \\
    &\lesssim  \norm{u}_{L^\infty_t H} \norm{\mathscr{S}}_{B^{\gamma}_{\mu,\infty}} \norm{I}_{B^{\gamma}_{q,\infty} H} h^{\theta 2 \gamma}.
\end{align*}
Notice that $\theta 2 \gamma > 1$ if and only if $\gamma > 1/2 + 1/(2+q)$. The latter inequality is satisfied since $\gamma > 1/2 + 1/q$.

Overall, we have established, for $\overline{\gamma} = \theta 2 \gamma \wedge (1/2 + \gamma) > 1$,
\begin{align} \label{est:J1}
    \mathrm{J}_1 \lesssim   \norm{u}_{L^\infty_t H} \norm{I}_{B^{\gamma}_{q,\infty}H} \left( \norm{\mathscr{S}}_{B^{\gamma}_{\mu,\infty}}  + \norm{b}_{L^1_t}\right) \, h^{\overline{\gamma}}.
\end{align}

Next, we investigate $\mathrm{J}_2$; H\"older's inequality with $a = 2/\mu \in (1,\infty]$ and $a' = 2/(2-\mu)$, and~\eqref{eq:localBounded-lip-time} show
\begin{align} \nonumber
    \mathrm{J}_2 &\leq \norm{u}_{L^\infty_t H}  \left(  \int_0^{T-h} \hspace{-0.5em}  \abs{\partial_y F(t+\theta h, y_{t}) - \partial_y F(t, y_{t})}^2 \dd t \right)^{1/2} \left(  \int_0^{T-h}  \hspace{-0.5em} \norm{I_{t+h} - I_t }_H^{q} \dd t \right)^{1/q} \\ \label{est:J2}
    &\lesssim \norm{u}_{L^\infty_t H} \seminorm{I}_{B^{\gamma}_{q,\infty} H} \,h^{1/2+\gamma}. 
\end{align}

Lastly, we check $\mathrm{J}_3$; using~\eqref{eq:localBounded-linf} and H\"older's inequality with $a = 2/\mu \in (1,\infty]$ and $a' = 2/(2-\mu)$,
\begin{align} \label{est:J3}
   \mathrm{J}_3 &\lesssim  \left(  \int_0^{T-h} \norm{ u_{t+\theta h} - u_t}_H^\mu \norm{I_{t+h} - I_t}_H^{\mu} \dd t \right)^{1/\mu} \hspace{-1em}\leq  \seminorm{u}_{B^{1/2}_{2,\infty} H} \seminorm{I}_{B^{\gamma}_{q,\infty}H }\, h^{1/2 + \gamma}.
\end{align}

Combining~\eqref{est:J1},~\eqref{est:J2} and~\eqref{est:J3} provide
\begin{align*}
    \norm{ \delta \partial_y F(y) (u,\dd I)}_{\bar{\mathbb{B}}^{\overline{\gamma}}_{\mu,\infty}} \lesssim  \left( \norm{u}_{L^\infty_t H} + \seminorm{u}_{B^{1/2}_{2,\infty}H} \right) \norm{I}_{B^{\gamma}_{q,\infty} H} \left( \norm{\mathscr{S}}_{B^{\gamma}_{\mu,\infty}}  + \norm{b}_{L^1_t}\right) .
\end{align*}

Therefore, we have shown that Lemma~\ref{sewing} is applicable with $\gamma = \overline{\gamma} > 1$, $\alpha = \gamma$ and $p_1 = p_2 = \mu \in (1,2]$, which establishes the existence of $\mathscr{I}[\partial_y F(y) (u,\dd I)] \in B^{\gamma}_{\mu,\infty}.$

\underline{\ref{it:02-chain}:} It is sufficient to show that each $\mathrm{R}^i$ decays sufficiently fast towards the diagonal $\abs{t-s} \ll 1$. Then the assertion follows by Lemma~\ref{local approx doesnt matter}.

First, we investigate $\mathrm{R}^1$; using~\eqref{eq:localBounded-lip-space} and~\eqref{eq:y-expansion}
\begin{align*}
    &\left( \int_0^{T-h} \abs{\mathrm{R}^1_{t,t+h}}^{\mu/2} \dd t \right)^{2/\mu} \\
    &\quad \leq \sup_{t \in [0,T]} \sup_{z_1,z_2 \in B_R} \frac{\abs{\partial_y F(t,z_1) - \partial_y F(t,z_2)}}{\abs{z_1 - z_2}} \left( \int_0^{T-h} (\abs{y_{t+h} - y_t}\abs{\mathscr{S}_{t+h} - \mathscr{S}_t})^{\mu/2} \dd t \right)^{2/\mu} \\
    &\quad \lesssim  \left( \int_0^{T-h} \left(\int_t^{t+h} \abs{b_r} \dd r \abs{\mathscr{S}_{t+h} - \mathscr{S}_t} \right)^{\mu/2} \dd t \right)^{2/\mu} +  \left( \int_0^{T-h} \abs{\mathscr{S}_{t+h} - \mathscr{S}_t}^{\mu} \dd t \right)^{2/\mu}\\
    &\quad =: \mathrm{K}_1 + \mathrm{K}_2.
\end{align*}
Let $\theta = 2/(\mu+1) \in [2/3,1)$, since $\mu \in ( 1,2]$. 
H\"older's inequality with $a = (\mu+1)/\mu$, $a' = \mu+1$, and Fubini's theorem show
\begin{align*}
    \mathrm{K}_1 &\lesssim  \norm{b}_{L^1_t}^{1-\theta} \norm{\mathscr{S}}_{L^\infty}^{1-\theta} \left( \int_0^{T-h} \left(\int_t^{t+h} \abs{b_r} \dd r \right)^{\theta \mu/2} \abs{\mathscr{S}_{t+h} - \mathscr{S}_t}^{\theta \mu/2}  \dd t \right)^{2/\mu} \\
    &\leq  \norm{b}_{L^1_t}^{1-\theta} \norm{\mathscr{S}}_{L^\infty}^{1-\theta} \left( \int_0^{T-h}\int_t^{t+h} \abs{b_r} \dd r \dd t  \right)^{2/(\mu+1)} \left(\int_0^{T-h} \abs{\mathscr{S}_{t+h} - \mathscr{S}_t}^{\mu} \dd t \right)^{2/(\mu(\mu+1))} \\
    &\leq h^{2(1 + \gamma)/(\mu+1)} \norm{b}_{L^1_t} \norm{\mathscr{S}}_{L^\infty}^{1-\theta} \seminorm{\mathscr{S}}_{B^{\gamma}_{\mu,\infty}}^{2/(\mu+1)}.
\end{align*}
Similarly by H\"older's inequality
\begin{align*}
    \mathrm{K}_2 &\leq T^{(\mu-1)/\mu} \left( \int_0^{T-h} \abs{\mathscr{S}_{t+h} - \mathscr{S}_t}^{\mu} \dd t \right)^{2/\mu} \leq T^{(\mu-1)/\mu} h^{2\gamma} \seminorm{\mathscr{S}}_{B^{\gamma}_{\mu,\infty}}^2.
\end{align*}
Define $\tilde{\gamma} =2(1 + \gamma)/(\mu+1) \wedge 2 \gamma$. Therefore, using~\eqref{eq:A-dom-deltaA} and $B^{\gamma}_{\mu,\infty} \hookrightarrow L^\infty$, we find
\begin{align*}
    \norm{\delta \mathrm{R}^1}_{\bar{\mathbb{B}}^{\tilde{\gamma}}_{\mu/2,\infty}} + \norm{\mathrm{R}^1}_{\mathbb{B}^{\tilde{\gamma}}_{\mu/2,\infty}} \lesssim \norm{b}_{L^1_t} \norm{\mathscr{S}}_{B^{\gamma}_{\mu,\infty}} + \seminorm{\mathscr{S}}_{B^{\gamma}_{\mu,\infty}}^2.
\end{align*}
We apply Lemma~\ref{local approx doesnt matter} with $\alpha \in (0,1)$ arbitrary, $\gamma = \tilde{\gamma}$, $p_1 = p_2 = \mu/2$, $A= \mathrm{R}^1$ and $A' = 0$, which implies $\lim_{\norm{\mathcal{P}}\to 0} 
 \mathscr{I}_{\mathcal{P}}\mathrm{R}^1 = \mathscr{I}\mathrm{R}^1 = \mathscr{I} 0 = 0$.

Next, we investigate $\mathrm{R}^2$; using~\eqref{eq:localBounded-linf}
\begin{align*}
    \left( \int_0^{T-h} \abs{\mathrm{R}^2_{t,t+h}}^\mu \dd t \right)^{1/\mu} \leq \sup_{t \in [0,T]} \sup_{z \in B_R} \abs{\partial_y F(t,z)} h^{\gamma +1/2}\norm{\delta \mathscr{S} - \mathscr{I}_\mathcal{P}}_{\mathbb{B}^{\gamma +1/2}_{\mu,\infty}}.
\end{align*}
This implies 
\begin{align*}
    \norm{\mathrm{R}^2}_{\mathbb{B}^{\gamma +1/2}_{\mu,\infty}} + \norm{\delta \mathrm{R}^2}_{\bar{\mathbb{B}}^{\gamma +1/2}_{\mu,\infty}} \lesssim \sup_{t \in [0,T]} \sup_{z \in B_R} \abs{\partial_y F(t,z)} \norm{\delta \mathscr{S} - \mathscr{I}_\mathcal{P}}_{\mathbb{B}^{\gamma +1/2}_{\mu,\infty}},
\end{align*}
from which we conclude by Lemma~\ref{local approx doesnt matter} that $\lim_{\norm{\mathcal{P}}\to 0} \mathscr{I}_{\mathcal{P}}\mathrm{R}^2  = 0$. Similarly, one shows $\lim_{\norm{\mathcal{P}}\to 0} \mathscr{I}_{\mathcal{P}}\mathrm{R}^3= \lim_{\norm{\mathcal{P}}\to 0} \mathscr{I}_{\mathcal{P}}\mathrm{R}^4= 0$.

\underline{\eqref{eq:conv-time}:} 
Artificially introducing the expected asymptotic of the sewing yields
\begin{align*}
  &\mathscr{I}_{\mathcal{P}} \mathrm{III}_{s,t}  - \int_s^t\partial_t F(r,y_r) \dd r  =  \int_{s}^{t}  \overline{\mathscr{H}}_{\mathcal{P}}(r;s,t) \dd r,
\end{align*}
where 
\begin{align*}
    &\overline{\mathscr{H}}_{\mathcal{P}}(r;s,t) \\
    &= \int_0^1 \Big( \partial_t F(s_{\mathcal{P}}^-(r;s,t) + \theta (s_{\mathcal{P}}^+(r;s,t) - s_{\mathcal{P}}^-(r;s,t) ), y_{\mathcal{P}}^+(r;s,t)- \partial_t F(r,y_r) \Big) \dd \theta,
\end{align*}
and $s_{\mathcal{P}}^\pm(r;s,t)$ and $y_{\mathcal{P}}^+(r;s,t)$ are defined analogously to~\eqref{eq:interpolations}. The claim follows in the same way as~\eqref{eq:conv-classic}.
\end{proof}

\section{Proof: Existence of weak solutions} \label{sec:proof-main}
In this section we prove Theorem~\ref{thm:main}. For a summary of the proof we refer to Section~\ref{sec:proof-strategy}. From now on we assume that the assumptions of Theorem~\ref{thm:main} are satisfied.

\subsection{The approximate model}
We start by considering the approximate model 
\begin{align} \label{eq:approximate-model}
  \forall t \in[0,T], v \in V:\quad   \left(u_t - u_0, v \right)_H = \int_0^t \langle A(s,u_s), v \rangle_{V^*,V} \dd s + \left( I_t^n(u), v \right)_H.
\end{align}
Due to the structural assumption on $I^n$ we can differentiate~\eqref{eq:approximate-model} in time, which leads us to the equivalent differential equation 
\begin{align} \label{eq:approximate-model-dif}
  \forall t \in (0,T)\backslash \mathcal{N}, v \in V:\quad   \langle \partial_t u_t - A(t,u_t), v \rangle_{V^*,V} = \left(b^n(t,u_t), v \right)_H,
\end{align}
where $\mathcal{N} \subset (0,T)$ is a nullset. 

The existence of weak solutions to monotone differential equations is a standard result and can be found e.g. in~\cite{MR1033498}.
\begin{lemma} \label{lem:existence-approximate-model}
There exists a unique weak solution $u^n \in W^{1,\alpha'}_t V^* \cap C_t H \cap L^\alpha_t V$ to~\eqref{eq:approximate-model-dif}. Moreover, $u^n$ solves~\eqref{eq:approximate-model-dif} if and only if $u^n$ solves~\eqref{eq:approximate-model}.
\end{lemma}

\subsection{New a priori bounds}
Next, we will derive new a priori estimates leveraging the mapping properties of the approximate integral operators. Most importantly, since we don't expect an asymptotically stable time derivative in $V^*$, we avoid using estimates that invoke the time derivative. Instead, we use estimates on the fractional Besov space~$B^{1/2}_{2,\infty} H$. This space can be thought of as a $1/2$-interpolation space in between $W^{1,\alpha'}_t V^*$ and $L^\alpha_t V$.  

We start by deriving two estimates for approximate weak solutions: the first addresses fractional time-regularity on $H$, and the second provides an estimate on the natural energy space $C_t H \cap L^\alpha_t V$. 
\begin{lemma} \label{lem:first-apriori-bound}
There exists a constant $C_{c_3,\alpha} >0$ such that
\begin{align} \label{eq:B1/2H-estimate}
    \seminorm{u^n}_{B^{1/2}_{2, \infty} H}^2\leq C_{c_3, \alpha} \left(\norm{u^n}_{L^\alpha_t V}^\alpha+\norm{g}_{L_t^{\alpha'}}^{\alpha'} \right)+  \seminorm{I^n(u^n)}_{B^{1/2}_{2, \infty}H}^2.
\end{align}
\end{lemma}
\begin{proof}
Let $\mathcal{J} = \{ t \in [0,T]: \, u^n(t) \in V\}$. Fix $h\in (0,T) \cap \mathcal{J} $ and $t \in (0,T-h)\cap \mathcal{J} $, and subtract~\eqref{eq:approximate-model} for $t' \in \{t+h,t\}$,
\begin{align*}
    \left(u^n_{t+h} - u^n_t, v \right)_H = \int_t^{t+h} \langle A(s,u^n_s), v \rangle_{V^*,V} \dd s + \left( I_{t+h}^n(u^n) - I_{t}^n(u^n), v \right)_H.
\end{align*}
Choosing $v = u^n_{t+h} - u^n_t \in V$, and using Cauchy-Schwarz and Young's inequalities
\begin{align*}
    &\frac{1}{2} \norm{u^n_{t+h} - u^n_t}_H^2 \leq \frac{1}{2} \norm{I_{t+h}^n(u^n) - I_{t}^n(u^n)}_{H}^2 \\
    &\hspace{2em} +\frac{1}{\alpha'}\int_t^{t+h}\norm{A(s, u^n_s)}_{V^*}^{\frac{\alpha}{\alpha-1}} \dd s+\frac{h}{\alpha}\norm{u^n_{t+h}-u^n_t}_{V}^\alpha .
\end{align*}
Next, we integrate over $t \in (0,T-h) \cap \mathcal{J}$ and use Fubini's theorem
\begin{align*}
    &\int_0^{T-h} \norm{u^n_{t+h} - u^n_t}_H^2 \dd t \leq  \int_0^{T-h}  \norm{I_{t+h}^n(u^n) - I_{t}^n(u^n)}_{H}^2 \dd t  \\
    &\hspace{2em} +h \left( \frac{2}{\alpha'} \int_0^{T}\norm{A(t, u^n_t)}_{V^*}^{\alpha'} \dd t+\frac{2}{\alpha} \int_0^{T-h} \norm{u^n_{t+h}-u^n_t}_{V}^\alpha \dd t \right).
\end{align*}
Re-scaling by $h^{-1}$, taking the supremum over $h \in (0,T)$, and using~\ref{it:H4} show
\begin{align*}
    \seminorm{u^n}_{B^{1/2}_{2,\infty}H}^2 &\leq \seminorm{I^n(u^n)}_{B^{1/2}_{2,\infty}H}^2 + \left(\frac{(2c_3)^{\alpha'}}{\alpha'}  +\frac{2^\alpha}{\alpha}  \right) \int_0^T \norm{u^n_t}^{\alpha}_V \dd t  + \frac{2^{\alpha'}}{\alpha'} \int_0^T \abs{g_t}^{\alpha'} \dd t.
\end{align*}
This establishes the estimate~\eqref{eq:B1/2H-estimate}.
\end{proof}

\begin{lemma} \label{lem:second-apriori-bound}
There exists a constant $C_{c_1,c_2,T,q,\gamma} > 0$ such that
\begin{align}
\begin{aligned} \label{eq:energy-estimate}
    \norm{u^n}_{L^\infty_t H}^2 &+ \norm{u^n}_{L^\alpha_t V}^\alpha \leq C_{c_1,c_2,T,q,\gamma} \left( \norm{u_0}_H^2 + 2\norm{f}_{L^1_t} \right) \\
    &+ C_{c_1,c_2,T,q,\gamma} \left( \norm{u^n}_{L^\infty_t H} + \seminorm{u^n}_{B^{1/2}_{2,\infty }H} \right)\seminorm{I^n (u^n)}_{B^{\gamma}_{q,\infty}H}.
    \end{aligned}
\end{align}
\end{lemma}
\begin{proof}
    Let $t\in [0,T] \backslash \mathcal{N} \cap \mathcal{J}$, where $\mathcal{J} = \{ t \in [0,T]: \, u^n(t) \in V\}$ and $\mathcal{N}$ is defined in~\eqref{eq:approximate-model-dif}. Choose $v = u^n_t \in V$ in~\eqref{eq:approximate-model-dif} to find
    \begin{align} \label{eq:Test-dif}
 \langle \partial_t u^n_t, u^n_t \rangle_{V^*,V} - \langle A(t,u^n_t), u^n_t \rangle_{V^*,V} = \left(b^n(t,u^n_t), u^n_t \right)_H.
    \end{align}
Since $V \hookrightarrow H$ continuously and densely, it holds
\begin{align*}
     \langle \partial_t u^n_t, u^n_t \rangle_{V^*,V} = \partial_t \left( \frac{1}{2} \norm{u^n_t}^2_H \right). 
\end{align*}
Invoking~\ref{it:H3},
\begin{align*}
    - \langle A(t,u^n_t), u^n_t \rangle_{V^*,V} \geq c_1\norm{u^n_t}_V^\alpha-c_2\norm{u^n_t}_H^2-f_t.
\end{align*}
Therefore,
\begin{align*}
    \partial_t \norm{u^n_t}^2_H \leq 2 c_2\norm{u^n_t}_H^2  + 2\left( \left(b^n(t,u^n_t), u^n_t \right)_H + f_t\right),
\end{align*}
and Gronwall's inequality implies for all $t' \in [0,T]$
\begin{align} \label{eq:Gronwall}
    \norm{u^n(t')}_{H}^2 \leq e^{2c_2 t'} \norm{u_0}_H^2 +  2\int_0^{t'} e^{2c_2(t'-t)} \left( \left(b^n(t,u^n_t), u^n_t \right)_H + f_t\right) \dd t.
\end{align}
Moreover, integrating~\eqref{eq:Test-dif} over $t \in [0,t'] \backslash \mathcal{N} \cap \mathcal{J}$ and using~\eqref{eq:Gronwall}, allow us to show
\begin{align*}
    \norm{u^n(t')}_{H}^2 &+ 2  c_1 \int_0^{t'} \norm{u^n_t}_V^\alpha \dd t \\
    &\leq \norm{u_0}_{H}^2+ 2\int_0^{t'} \left(b^n(t,u^n_t), u^n_t \right)_H + c_2\norm{u^n_t}_H^2 + f_t\dd t \\
    &\leq \norm{u_0}_{H}^2+ 2\int_0^{t'}  \left(b^n(t,u^n_t), u^n_t \right)_H + f_t\dd t \\
    &\hspace{-1em} + \int_0^{t'}  2c_2\left\{ e^{2c_2 t} \norm{u_0}_H^2 +  2\int_0^{t} e^{2c_2(t-s)} \left( \left(b^n(s,u^n_s), u^n_s \right)_H + f_s \right) \dd s \right\}  \dd t \\
    &= e^{2c_2 t'} \norm{u_0}_{H}^2 + 2 \int_0^{t'} e^{2c_2(t'-t)} \left( \left(b^n(t,u^n_t), u^n_t \right)_H + f_t\right) \dd t,
\end{align*}
where we used that
\begin{align*}
    \int_0^{t'}  2c_2 e^{2c_2 t} \norm{u_0}_H^2  \dd t = (e^{2c_2 t'} - 1)\norm{u_0}_H^2,
\end{align*}
and
\begin{align*}
    &\int_0^{t'}  2c_2 \int_0^{t} e^{2c_2(t-s)} \left( \left(b^n(s,u^n_s), u^n_s \right)_H + f_s \right) \dd s \dd t \\
    &\hspace{2em} = \int_0^{t'} \left( e^{2c_2(t'-t)} - 1\right) \left( \left(b^n(t,u^n_t), u^n_t \right)_H + f_t\right) \dd t.
\end{align*}

Recall that $b^n(s,u^n_s) = \partial_s I^n_s( u^n) \in L^{\infty}_t H$ and $u^n \in W^{1,\alpha'}_t V^* \cap L^{\alpha}_t V \hookrightarrow B^{1/2}_{2,\infty} H$. Therefore, we can apply Corollary~\ref{cor:weighted-time-derivative} with $\lambda = 2c_2$ to find a constant $C_{c_2,T,q,\gamma}> 0$ such that
\begin{align*}
  &\sup_{t' \in [0,T]}  \int_0^{t'} e^{2c_2(t'-t)} \left(b^n(t,u^n_t), u^n_t \right)_H \hspace{-2pt} \dd t \\
  &\hspace{3em} \leq C_{c_2,T,q,\gamma} \left( \norm{u^n}_{L^\infty_t H} + \seminorm{u^n}_{B^{1/2}_{2,\infty }H} \right)\seminorm{I^n(u^n)}_{B^{\gamma}_{q,\infty}H}.
\end{align*}
We arrive at
\begin{align*}
    \norm{u^n(t')}_{H}^2 &+ 2  c_1 \int_0^{t'} \norm{u^n_t}_V^\alpha \dd t \leq e^{2c_2 T} \left( \norm{u_0}_H^2 + 2\norm{f}_{L^1_t} \right) \\
    &\hspace{2em} + 2  C_{c_2,T,q,\gamma} \left( \norm{u^n}_{L^\infty_t H} + \seminorm{u^n}_{B^{1/2}_{2,\infty }H} \right)\seminorm{I^n (u^n)}_{B^{\gamma}_{q,\infty}H}.
\end{align*}
The assertion~\eqref{eq:energy-estimate} follows by taking the supremum over $t' \in [0,T]$.
\end{proof}

So far, neither~\eqref{eq:B1/2H-estimate} nor~\eqref{eq:energy-estimate} are closed estimates for $u^n$ as both sides of the inequalities depend on the approximate weak solution. Next, we will use the quantified mapping properties of the approximate integral operators~$I^n$, cf.~\ref{it:H5}, to close the argument.

\begin{lemma} \label{lem:uniform-estimate}
There exists a constant $C = C(c_1,c_2,c_3,c_4,\lambda,\alpha,q,\gamma,T) > 0$ such that 
\begin{align} \label{eq:uniform-estimate}
   \norm{u^n}_{L^\infty_t H}^2 + \norm{u^n}_{L^\alpha_t V}^\alpha + \norm{u^n}_{B^{1/2}_{2,\infty} H}^2 \leq C \left( \norm{u_0}_H^2 + \norm{f}_{L^1_t} +\norm{g}_{L_t^{\alpha'}}^{\alpha'} +1\right).
\end{align}
\end{lemma}

\begin{proof}
The proof consists of six steps:
\begin{enumerate}
    \item \label{it:closing-01} a preliminary estimate;
    \item \label{it:closing-02} setup of induction;
    \item \label{it:closing-03} start of induction;
    \item \label{it:closing-04} induction hypothesis;
    \item \label{it:closing-05} induction step;
    \item \label{it:closing-06} gluing of local norms to global ones.
\end{enumerate}

\underline{Step~\ref{it:closing-01}:} First we use Lemma~\ref{lem:first-apriori-bound} and~\ref{lem:second-apriori-bound} to estimate
\begin{align*}
     &\norm{u^n}_{L^\infty_t H}^2 + \norm{u^n}_{L^\alpha_t V}^\alpha + \seminorm{u^n}_{B^{1/2}_{2,\infty} H}^2 \\
     &\leq \norm{u^n}_{L^\infty_t H}^2 + (1+C_{c_3, \alpha}) \norm{u^n}_{L^\alpha_t V}^\alpha + C_{c_3, \alpha}\norm{g}_{L_t^{\alpha'}}^{\alpha'} +  \seminorm{I^n(u^n)}_{B^{1/2}_{2, \infty}H}^2 \\
     &\leq  (1 + C_{c_3, \alpha})\left( \norm{u^n}_{L^\infty_t H}^2 +\norm{u^n}_{L^\alpha_t V}^\alpha \right)+ C_{c_3, \alpha}\norm{g}_{L_t^{\alpha'}}^{\alpha'} +  \seminorm{I^n(u^n)}_{B^{1/2}_{2, \infty}H}^2  \\
     &\leq  (1 + C_{c_3, \alpha}) C_{c_1,c_2,T,q,\gamma} \left( \norm{u_0}_H^2 + 2\norm{f}_{L^1_t} \right) +  C_{c_3, \alpha}\norm{g}_{L_t^{\alpha'}}^{\alpha'}  + \seminorm{I^n(u^n)}_{B^{1/2}_{2, \infty}H}^2 \\
     &\quad + (1 + C_{c_3, \alpha}) C_{c_1,c_2,T,q,\gamma} \left( \norm{u^n}_{L^\infty_t H} + \seminorm{u^n}_{B^{1/2}_{2,\infty }H} \right)\seminorm{I^n (u^n)}_{B^{\gamma}_{q,\infty}H}.
\end{align*}
Young's inequality implies
\begin{align} \label{eq:estimate-pre-H5}
\begin{aligned}
\norm{u^n}_{L^\infty_t H}^2 &+ \norm{u^n}_{L^\alpha_t V}^\alpha + \seminorm{u^n}_{B^{1/2}_{2,\infty} H}^2 \\
    &\leq 2(1 + C_{c_3, \alpha}) C_{c_1,c_2,T,q,\gamma} \left( \norm{u_0}_H^2 + 2\norm{f}_{L^1_t} \right) + 2 C_{c_3, \alpha}\norm{g}_{L_t^{\alpha'}}^{\alpha'}  \\
    &\quad +  (1 + C_{c_3, \alpha})^2 (1+ C_{c_1,c_2,T,q,\gamma})^2 \seminorm{I^n (u^n)}_{B^{\gamma}_{q,\infty}H}^2.
\end{aligned}
\end{align}
Finally, using~\ref{it:H5} 
\begin{align}\label{eq:Iteration-Bound}
\begin{aligned}
\norm{u^n}_{L^\infty_t H}^2 &+ \norm{u^n}_{L^\alpha_t V}^\alpha + \seminorm{u^n}_{B^{1/2}_{2,\infty} H}^2    \\
&\leq \mathrm{K}_1 \norm{u_0}_H^2 + \mathrm{K}_2 + \mathrm{K}_3 + \mathrm{K}_3 \lambda(T)\left\{ \seminorm{u^n}_{B^{1/2}_{2, \infty}H}^2+\norm{u^n}_{L^\infty_t H}^2 \right\},
\end{aligned}
\end{align}
where
\begin{align*}
    \mathrm{K}_1 &:=  2(1 + C_{c_3, \alpha}) C_{c_1,c_2,T,q,\gamma}, \\
    \mathrm{K}_2 &:=  4(1 + C_{c_3, \alpha}) C_{c_1,c_2,T,q,\gamma}\norm{f}_{L^1_t} + 2C_{c_3, \alpha}\norm{g}_{L_t^{\alpha'}}^{\alpha'}, \\
    \mathrm{K}_3 &:=  (1 + C_{c_3, \alpha})^2 (1+ C_{c_1,c_2,T,q,\gamma})^2  c_4.
\end{align*}


\underline{Step~\ref{it:closing-02} -- Induction setup:}  Define
\begin{align} \label{def:short-time}
    T^* = \inf\{ t \in [0,T]: \lambda(t) \geq (2 \mathrm{K}_3)^{-1} \},
\end{align}
where we set $T^* = T$ if $\lambda(t) < (2 \mathrm{K}_3)^{-1}$ for all $t\in[0,T]$. Since $t\mapsto \lambda(t)$ is continuous in $0$ and $\lambda(0) = 0$, it holds $T^* > 0$. We define the ratio between $T$ and $T^*$ as $r = \lceil T/T^* \rceil$, the $k$-th interval as~$J_k := [(k-1)T^*, k T^*)\cap [0,T]$, and the partition of $[0,T)$ with mesh-size $T^*$ by $\mathcal{P} = \{ J_k:  k \in \{1,\ldots, r\} \}$.

\underline{Step~\ref{it:closing-03} -- Induction start:} : Clearly, $u^n$ solves~\eqref{eq:approximate-model} on~$[0,T^*)$. Thus,~\eqref{eq:Iteration-Bound} holds for $T$ replaced by $T^*$. In particular, $\mathrm{K}_3 \lambda(T^*) \leq 1/2$ by definition of $T^*$. We infer
\begin{align*}
\norm{u^n}_{L^\infty(J_1;H)}^2 &+ \norm{u^n}_{L^\alpha(J_1;V)}^\alpha + \seminorm{u^n}_{B^{1/2}_{2,\infty}(J_1;H )}^2    \\
&\leq \mathrm{K}_1 \norm{u_0}_H^2 + \mathrm{K}_2 + \mathrm{K}_3+ 1/2 \left\{ \seminorm{u^n}_{B^{1/2}_{2, \infty}(J_1;H)}^2+\norm{u^n}_{L^\infty(J_1;H)}^2 \right\},
\end{align*}
which shows 
\begin{align}
    \norm{u^n}_{L^\infty(J_1;H)}^2 + \norm{u^n}_{L^\alpha(J_1;V)}^\alpha + \seminorm{u^n}_{B^{1/2}_{2,\infty}(J_1;H )}^2 \leq 2 \mathrm{K}_1 \norm{u_0}_H^2 + 2\mathrm{K}_2 + 2\mathrm{K}_3.
\end{align}

\underline{Step~\ref{it:closing-04} -- Induction hypothesis:} 
We assume that for all $\tilde{k} \leq k-1$ it holds
\begin{align} \label{eq:induction-hypothesis}
\begin{aligned}
 \norm{u^n}_{L^\infty(J_{\tilde{k}};H)}^2 &+ \norm{u^n}_{L^\alpha(J_{\tilde{k}};V)}^\alpha + \seminorm{u^n}_{B^{1/2}_{2,\infty}(J_{\tilde{k}};H )}^2    \\
&\leq (2 \mathrm{K}_1)^{\tilde{k}} \norm{u_0}_H^2 + \left( 2\mathrm{K}_2 + 2\mathrm{K}_3 \right) \sum_{l=0}^{\tilde{k}-1} (2  \mathrm{K}_1)^l .
\end{aligned}
\end{align}

\underline{Step~\ref{it:closing-05} -- Induction step:} 
Notice that $u^n$ solves for $t \in J_k$
\begin{align*}
     \left(u_t^n - u_{(k-1)T^*}^n, v \right)_H = \int_{(k-1)T^*}^t \langle A(s,u_s^n), v \rangle_{V^*,V} \dd s + \left( \overline{I}_t^{k-1,n}(u^n) , v \right)_H,
\end{align*}
where $ \overline{I}_t^{k-1,n}(u^n) := I_t^n(u^n) - I_{(k-1)T^*}^n(u^n)$. Similarly to Step~\ref{it:closing-01}, we derive
\begin{align} \label{eq:Estimate-k}
\begin{aligned}
    \norm{u^n}_{L^\infty(J_k;H)}^2 &+ \norm{u^n}_{L^\alpha(J_k;V)}^\alpha + \seminorm{u^n}_{B^{1/2}_{2,\infty}(J_k; H)}^2 \\
    &\leq \mathrm{K}_1 \norm{u_{(k-1)T^*}^n}_{H}^2 + \mathrm{K}_2 \\
    &\quad +  (1 + C_{c_3, \alpha})^2 (1+ C_{c_1,c_2,T^*,q,\gamma})^2 \seminorm{\overline{I}^{k-1,n}(u^n)}_{B^{\gamma}_{q,\infty}(J_k;H)}^2.
\end{aligned}
\end{align}
Observe that for fixed $h \leq T^*$ 
\begin{align*}
   \int_{(k-1)T^*}^{kT^*-h} \norm{\overline{I}_{t+h}^{k-1,n}(u^n) -\overline{I}_{t}^{k-1,n}(u^n) }_H^2 \dd t = \int_{(k-1)T^*}^{kT^*-h} \norm{I_{t+h}^{n}(u^n) -I_{t}^{n}(u^n) }_H^2 \dd t,
\end{align*}
which implies $\seminorm{\overline{I}^{k-1,n}(u^n)}_{B^{\gamma}_{q,\infty}(J_k;H)} = \seminorm{I^n(u^n)}_{B^{\gamma}_{q,\infty}(J_k;H)}$. Therefore, using~\ref{it:H5} and~$\mathrm{K}_3 \lambda(T^*) \leq 1/2$ in~\eqref{eq:Estimate-k} establish
\begin{align*}
   \norm{u^n}_{L^\infty(J_k;H)}^2 &+ \norm{u^n}_{L^\alpha(J_k;V)}^\alpha + \seminorm{u^n}_{B^{1/2}_{2,\infty}(J_k;H )}^2    \\
&\leq \mathrm{K}_1 \norm{u_{(k-1)T^*}^n}_{H}^2 + \mathrm{K}_2 + \mathrm{K}_3+ 1/2 \left\{ \seminorm{u^n}_{B^{1/2}_{2, \infty}(J_k;H)}^2+\norm{u^n}_{L^\infty(J_k;H)}^2 \right\}.
\end{align*}
Reordering and using the induction hypothesis~\eqref{eq:induction-hypothesis} provide
\begin{align*}
     \norm{u^n}_{L^\infty(J_k;H)}^2 &+ \norm{u^n}_{L^\alpha(J_k;V)}^\alpha + \seminorm{u^n}_{B^{1/2}_{2,\infty}(J_k;H )}^2  \\
     &\leq 2  \mathrm{K}_1 \norm{u_{(k-1)T^*}^n}_{H}^2 + 2\mathrm{K}_2 + 2\mathrm{K}_3 \\
     &\leq 2  \mathrm{K}_1 \left( (2 \mathrm{K}_1)^{k-1} \norm{u_0}_H^2 + \left( 2\mathrm{K}_2 + 2\mathrm{K}_3 \right) \sum_{l=0}^{k-1} (2  \mathrm{K}_1)^l \right) + 2\mathrm{K}_2 + 2\mathrm{K}_3 \\
     &= (2 \mathrm{K}_1)^{k} \norm{u_0}_H^2 + \left( 2\mathrm{K}_2 + 2\mathrm{K}_3 \right) \sum_{l=0}^{k-1} (2  \mathrm{K}_1)^l.
\end{align*}
Thus, we have proved that~\eqref{eq:induction-hypothesis} remains valid for $\tilde{k} = k$.

\underline{Step~\ref{it:closing-06} -- Gluing:} The induction verifies a uniform bound for localized norms. It remains to estimate the global norms in terms of local ones. We discuss each term on the left-hand side of~\eqref{eq:uniform-estimate} individually. 

$L^\infty([0,T];H)$: Clearly, using~\eqref{eq:Estimate-k}
\begin{align*}
    \norm{u^n}_{L^\infty(0,T;H)}^2 = \max_{k \in \{1,\ldots, r\} } \norm{u^n}_{L^\infty(J_k;H)}^2 \leq (2 \mathrm{K}_1)^{r} \left( \norm{u_0}_H^2 +2\mathrm{K}_2 + 2\mathrm{K}_3 \right).
\end{align*}

$L^\alpha([0,T];V)$: Similarly,
\begin{align*}
    \norm{u^n}_{L^\alpha([0,T];V)}^\alpha = \sum_{k=1}^r \norm{u^n}_{L^\alpha(J_k; V)}^\alpha &\leq \left( \norm{u_0}_H^2 +2\mathrm{K}_2 + 2\mathrm{K}_3 \right) \sum_{k=1}^r (2 \mathrm{K}_1)^{k} \\
    &\leq (2 \mathrm{K}_1)^{r+1}  \left( \norm{u_0}_H^2 +2\mathrm{K}_2 + 2\mathrm{K}_3 \right).
\end{align*}

$B^{1/2}_{2,\infty}([0,T];H)$: Fractional Besov spaces can't be localised easily, i.e.,
\begin{align*}
    u^n \in \bigcap_{k=1}^r B^{1/2}_{2,\infty}(J_k;H) \quad \not \Rightarrow \quad u^n \in B^{1/2}_{2,\infty}([0,T];H).
\end{align*}
However, we have the additional information $u^n \in L^\infty([0,T]; H)$. This is sufficient for a localisation of the Nikolskii norm; indeed, Lemma~\ref{lem:local-Nikolskii} and~\eqref{eq:Estimate-k} imply
\begin{align*}
    \seminorm{u^n}_{B^{1/2}_{2,\infty}([0,T];H)}^2 &\leq \sum_{k=1}^r  \seminorm{u^n}_{B^{1/2}_{2,\infty}(J_k;H)}^2 + \norm{u^n}_{L^\infty([0,T];H)}^2 2(2r-1) \\
    &\leq  \sum_{k=1}^r \left\{ (2 \mathrm{K}_1)^{k} \norm{u_0}_H^2 + \left( 2\mathrm{K}_2 + 2\mathrm{K}_3 \right) \sum_{l=0}^{k-1} (2  \mathrm{K}_1)^l \right\}  \\
    &\quad + (2 \mathrm{K}_1)^{r} \left( \norm{u_0}_H^2 +2\mathrm{K}_2 + 2\mathrm{K}_3 \right)2(2r-1) \\
    &\leq (2 \mathrm{K}_1)^{r}\left( 2 \mathrm{K}_1 +2(2r-1) \right)   \left( \norm{u_0}_H^2 +2\mathrm{K}_2 + 2\mathrm{K}_3 \right) .
\end{align*}
Observe that
\begin{align*}
    \norm{u}_{L^2(0,T;H)}^2 \leq T \norm{u}_{L^\infty([0,T];H)}^2 \leq T(2 \mathrm{K}_1)^{r} \left( \norm{u_0}_H^2 +2\mathrm{K}_2 + 2\mathrm{K}_3 \right).
\end{align*}
Therefore,
\begin{align*}
     \norm{u^n}_{B^{1/2}_{2,\infty}([0,T];H)}^2 \leq (2 \mathrm{K}_1)^{r}\left( 2 \mathrm{K}_1 +2(2r-1) + T\right)   \left( \norm{u_0}_H^2 +2\mathrm{K}_2 + 2\mathrm{K}_3 \right).
\end{align*}

All together, we find
\begin{align*}
    \norm{u^n}_{L^\infty_t H}^2 &+ \norm{u^n}_{L^\alpha_t V}^\alpha + \norm{u^n}_{B^{1/2}_{2,\infty} H}^2 \\
    &\leq 3(2 \mathrm{K}_1)^{r}\left( 2 \mathrm{K}_1 +2(2r-1) + T\right)   \left( \norm{u_0}_H^2 +2\mathrm{K}_2 + 2\mathrm{K}_3 \right),
\end{align*}
and the proof is complete.
\end{proof}

\subsection{A limit is born}
The Banach--Alaoglu theorem states that
\begin{align*}
    B_{X^*} := \left\{ f \in X^*: \, \norm{f}_{X^*} \leq R \right\}
\end{align*}
is relatively compact with respect to weak-star convergence for arbitrary Banach spaces~$X$ and constants~$R >0$. We aim to use this fact for the construction of limits.
\begin{lemma} \label{lem:existence-of-limits}
There exist $u \in L^\infty_t H \cap B^{1/2}_{2,\infty} H \cap L^\alpha_t V$, $\overline{A} \in L^{\alpha'}_t V^*$, and a sub-sequence (not relabeled) such that 
\begin{subequations} \label{eq:convergence}
\begin{alignat}{2} \label{eq:conv-lalphaV}
    &u^n \rightharpoonup u &&\in L^\alpha_t V, \\ \label{eq:conv-linfH}
    &u^n \overset{*}{\rightharpoonup} u &&\in L^\infty_t H, \\ \label{eq:conv-BH}
    &u^n \overset{*}{\rightharpoonup} u &&\in B^{1/2}_{2,\infty} H, \\ \label{eq:conv-strong}
    &u^n \rightarrow u &&\in L^2_t H, \\ 
    \label{eq:conv-lalpha'V*}
    &A(\cdot, u^n) \rightharpoonup \overline{A} &&\in L^{\alpha'}_t V^*.
\end{alignat}
\end{subequations}
\end{lemma}
\begin{proof}
First, since $V$ is a reflexive Banach space, we identify that $(L^\alpha_t V)^* = L^{\alpha'}_t V^*$, and $( L^{\alpha'}_t V^*)^* = L^\alpha_t V^{**} = L^\alpha_t V$. Moreover, $L^\infty_t H = (L^1_t H)^*$ and $B^{1/2}_{2,\infty} H = (\tilde{B}^{-1/2}_{2,1} H)^*$. The definition of Besov spaces with respect to integrability of difference quotients doesn't immediately transfer to negative derivatives. However, for positive differentiability, $\gamma >0$, there exists an alternative spectral representation, i.e., $B^{\gamma}_{2,\infty} = \tilde{B}^{\gamma}_{2,\infty}$ with equivalent norms. This spectral definition canonically extends to $\gamma \in \mathbb{R}$. Moreover, one has the identification $\tilde{B}^{\gamma}_{2,\infty} = (\tilde{B}^{-\gamma}_{2,1})^*$. More details can be found in~\cite[Appendix~A.1]{Wichmann2024}.

We start by verifying~\eqref{eq:conv-lalphaV}: Choose $X =  L^{\alpha'}_t V^*$ and define 
\begin{align} \label{eq:duality-map}
    f \mapsto \iota_{v}(f):= \int_0^T \langle v(t), f(t) \rangle_{V,V^*} \dd t\in (L^{\alpha'}_t V^*)^*, \quad v\in L^\alpha_t V.
\end{align}
Due to~\eqref{eq:uniform-estimate}, it holds $\{ \iota_{u^n} \}_n \subset B_{(L^{\alpha'}_t V^*)^*}$. The theorem of Banach--Alaoglu establishes the existence of $\overline{\iota} \in (L^{\alpha'}_t V^*)^*$ and a sub-sequence (not relabeled) such that 
\begin{align} \label{eq:V**-weak}
 \forall f \in  L^{\alpha'}_t V^*: \quad  \iota_{u^n}(f) \rightarrow \overline{\iota}(f).
\end{align}
But since $( L^{\alpha'}_t V^*)^* = L^\alpha_t V^{**} = L^\alpha_t V$, there exists $u \in L^\alpha_t V$ such that $\overline{\iota}(f) = \int \langle u(t), f(t) \rangle_{V,V^*} \dd t = \iota_u(f)$. This and~\eqref{eq:V**-weak} imply $u^n \rightharpoonup u \in L^\alpha_t V$. 

The assertions~\eqref{eq:conv-linfH},~\eqref{eq:conv-BH}, and~\eqref{eq:conv-lalpha'V*} follow analogously by choosing $X = L^\alpha_t V$, $X = L^1_t H$, and $X = \tilde{B}^{-1/2}_{2,1} H$, respectively. In each step, we possibly restrict to another sub-sequence.

Lastly, we check~\eqref{eq:conv-strong}: Recall by Assumption~\ref{ass:compact-Gelfand} it holds $V \hookrightarrow H$ compactly. Therefore, an application of Theorem~\ref{nikoslki-lions} with $q = \alpha$, $X_0 = V$, $\gamma = 1/2$, $p = 2$, and $X=H$ shows that $\{u^n\}_{n\in \mathbb{N}}$ is relatively compact in $L^2_t H$.
\end{proof}

\subsection{The limit is a solution}
The identification of the limit is a delicate task. We start by identifying the limiting equation; multiply~\eqref{eq:approximate-model} by a temporal test function $\eta \in C^1_t$ with $\eta(T) = 0$ (we write $\eta' = \partial_t \eta$), and integrate in time
\begin{align*}
\int_0^T \left(u^n_t - u_0, v \right)_H \eta_t' \dd t = \int_0^T \int_0^t \langle A(s,u^n_s), v \rangle_{V^*,V} \dd s \, \eta_t' \dd t + \int_0^T \left( I_t^n(u^n), v \right)_H \eta_t' \dd t.
\end{align*}
Notice that due to~\eqref{eq:conv-linfH},~\eqref{eq:conv-lalpha'V*} respectively~\ref{it:H6}, it holds
\begin{align*}
    \int_0^T \left(u^n_t - u_0, v \right)_H \eta_t' \dd t &\rightarrow \int_0^T \left(u_t - u_0, v \right)_H \eta_t' \dd t, \\
    \int_0^T \int_0^t \langle A(s,u^n_s), v \rangle_{V^*,V} \dd s \, \eta_t' \dd t &\rightarrow \int_0^T \int_0^t \langle \overline{A}_s, v \rangle_{V^*,V} \dd s \, \eta_t' \dd t, \\
    \int_0^T \left( I_t^n(u^n), v \right)_H \eta_t' \dd t & \rightarrow \int_0^T \left( I_t(u), v \right)_H \eta_t' \dd t.
\end{align*}
Thus, $u$ solves for almost all $t \in [0,T]$ and all $v \in V$
\begin{align}
\label{limiting start}
\left(u_t - u_0, v \right)_H = \int_0^t \langle \overline{A}_s, v \rangle_{V^*,V} \dd s + \left( I_t(u), v \right)_H.
\end{align}
Moreover, since the right-hand side of~\eqref{limiting start} is continuous in time for fixed $v \in V$, so must be the left-hand side and~\eqref{limiting start} remains valid for all $t\in [0,T]$.

\begin{lemma}[Identification of monotone operator]
It holds $\overline{A} = A(\cdot, u)$.
\end{lemma}
\begin{proof}
The proof consists of 5 steps:
\begin{enumerate}
    \item \label{it:identification-01} expanding the square;
    \item \label{it:identification-02} preparing the chain rule;
    \item \label{it:identification-03} applying the chain rule;
    \item \label{it:identification-04} identifying the Young integral;
    \item \label{it:identification-05} identifying the monotone operator.
\end{enumerate}

\underline{Step~\ref{it:identification-01} -- expanding the square:}
Notice that $u$ and $u^n$ solve~\eqref{limiting start} and~\eqref{eq:approximate-model}, respectively. An application of Theorem~\ref{chain rule norm} shows for all $t\in[0,T]$
\begin{subequations}
\begin{align} \label{eq:square-u}
    \norm{u_t}^2_H &= \norm{u_0}_H^2 + 2 \int_0^t \langle \overline{A}_s, u_s \rangle_{V^*,V} \dd s + 2\mathscr{S}_t( u,\dd I(u)), \\ \label{eq:square-un}
    \norm{u_t^n}^2_H &= \norm{u_0}_H^2 + 2 \int_0^t \langle A(s,u^n_s), u_s \rangle_{V^*,V} \dd s + 2\mathscr{S}_t( u^n,\dd I^n(u^n)).
\end{align}
\end{subequations}

\underline{Step~\ref{it:identification-02}  -- preparing the chain rule:}
Fix $\phi \in L^\alpha_t V \cap C_t H $ and define
\begin{align*}
    F(t,y) := \exp\left(-\int_0^t h_s + \eta(\phi_s) \dd s \right) y.
\end{align*}
We aim to apply Theorem~\ref{thm:Chain-rule-after}. Therefore, we need to check the assumptions. Clearly, $F$ is differentiable with
\begin{align*}
    \partial_y F(t,y) &=  \exp\left(-\int_0^t h_s + \eta(\phi_s) \dd s \right) ,\\
    \partial_t F(t,y) &= -(h_t + \eta(\phi_t)) F(t,y).
\end{align*}
Let $R >0$ and define $C_R:= (R \vee \sqrt{2}) \left( \int_0^T \left(h_s + \eta(\phi_s) \right) \dd s + 1 \right)$. Then~\eqref{eq:localBounded} holds.

Ad~\eqref{eq:localBounded-l1}: Since $h$ and $\eta$ are non-negative
\begin{align*}
    \int_0^T \sup_{z \in B_R} \abs{\partial_t F(t,z)} \dd t \leq R \int_0^T \left(h_s + \eta(\phi_s) \right) \dd s.
\end{align*}

Ad~\eqref{eq:localBounded-linf} and~\eqref{eq:localBounded-lip-space}: Clearly, 
\begin{align*}
\sup_{t \in [0,T]} \sup_{z \in B_R}  \abs{\partial_y F(t,z)}\leq 1 \quad \text{ and } \quad  \sup_{t \in [0,T]} \sup_{z_1,z_2 \in B_R} \frac{\abs{\partial_y F(t,z_1) - \partial_yF(t,z_2)}}{\abs{z_1 - z_2}} = 0.
\end{align*}

Ad~\eqref{eq:localBounded-lip-time}:
It remains to verify
\begin{align*}
    \sup_{h \in [0,T]} h^{-1/2}\left(  \int_0^{T-h}   \sup_{z \in B_R} \abs{\partial_y F(t+h, z) - \partial_y F(t, z)}^2 \dd t \right)^{1/2}  \leq C_R.
\end{align*}
Notice that $z \mapsto \partial_y F(t,z)$ is constant. The fundamental theorem and and Fubini's theorem imply 
\begin{align*}
    &\int_0^{T-h}   \sup_{z \in B_R} \abs{\partial_y F(t+h, z) - \partial_y F(t, z)}^2 \dd t \\
    &\hspace{2em}\leq 2\sup_{t \in [0,T]} \sup_{z \in B_R}   \abs{\partial_y F(t, z)} \int_0^{T-h}   \abs{\partial_y F(t+h, z) - \partial_y F(t, z)} \dd t \\
    &\hspace{2em}\leq 2 \int_0^{T-h}  \int_t^{t+h} [h_r + \eta(\phi_r)] \exp\left(-\int_0^r h_s +\eta(\phi_s) \dd s \right) \dd r \dd t \\
    &\hspace{2em}\leq 2 h \left( \norm{h}_{L^1_t} + \int_0^T \eta(\phi_s) \dd s \right),
\end{align*}
which is sufficient to conclude~\eqref{eq:localBounded-lip-time}.

\underline{Step~\ref{it:identification-03} -- applying the chain rule:} In the last step we have verified that~$F$ satisfies the assumption of Theorem~\ref{thm:Chain-rule-after}. Applying the theorem with $y_t = \norm{u_t^n}_H^2$, $b_s = 2\langle A(s,u^n_s), u_s^n \rangle_{V^*,V}$ shows
\begin{align*}
     &\exp\left(-\int_0^t h_s + \eta(\phi_s) \dd s \right)\norm{u_t^n}^2_H - \norm{u_0}_H^2 \\
    &= \int_0^t -( h_s + \eta(\phi_s)) \exp\left(-\int_0^s h_r + \eta(\phi_r) \dd r \right)\norm{u_s^n}^2_H   \dd s \\
    &\quad +  \int_0^t  \exp\left(-\int_0^s h_r + \eta(\phi_r) \dd r \right)2 \langle A(s,u^n_s), u^n_s \rangle_{V^*,V} \dd s \\
    &\quad + 2\mathscr{I}_t( \partial_y F(y) u^n, \dd I(u^n)),\\
    &= \int_0^t\exp\left(-\int_0^s h_r + \eta(\phi_r) \dd r \right) \Big\{  -( h_s + \eta(\phi_s)) \left( - \norm{\phi_s}_H^2 + 2\left( u^n_s,\phi_s \right)_H\right) \Big\} \dd s\\
    &\quad + \int_0^t\exp\left(-\int_0^s h_r + \eta(\phi_r) \dd r \right) \Big\{ 2 \langle  A(s,\phi_s), u^n_s - \phi_s \rangle_{V^*,V} + 2\langle A(s,u^n_s) , \phi_s \rangle_{V^*,V} \Big\} \dd s\\
    &\quad + \int_0^t\exp\left(-\int_0^s h_r + \eta(\phi_r) \dd r \right) \\
    &\hspace{3em} \Big\{  -( h_s + \eta(\phi_s)) \norm{u^n_s - \phi_s}_H^2 + 2\langle A(s,u^n_s) - A(s,\phi_s), u^n_s - \phi_s \rangle_{V^*,V}  \Big\} \dd s\\
    &\quad +  2\mathscr{S}_t( \partial_y F u^n, \dd I^n(u^n)),
\end{align*}
where we used that $y \mapsto \partial_y F(t,y)$ is constant in $y$, and
\begin{align*}
    \norm{u^n_s}_H^2 &=  \norm{u^n_s - \phi_s}_H^2 - \norm{\phi_s}_H^2 + 2\left( u^n_s,\phi_s \right)_H, \\
    \langle A(s,u^n_s), u^n_s \rangle_{V^*,V}  &=  \langle A(s,u^n_s) - A(s,\phi_s), u^n_s - \phi_s \rangle_{V^*,V}\\
    &\quad + \langle  A(s,\phi_s), u^n_s - \phi_s \rangle_{V^*,V} + \langle A(s,u^n_s) , \phi_s \rangle_{V^*,V}.
\end{align*}

Let $\overline{\psi} \in W^{1,\infty}_t$ be a non-negative function with $\overline{\psi}(T) = 0$. Set $\psi = \partial_t \overline{\psi}$. Then, using weak lower semi-continuity, Fatou's lemma,~\ref{it:H2} and~\eqref{eq:convergence},
\begin{align} \label{eq:limit-un-equation}
\begin{aligned}
    &\int_0^T \psi_t \bigg[ \exp\left(-\int_0^t h_s + \eta(\phi_s) \dd s \right)\norm{u_t}^2_H - \norm{u_0}_H^2 \bigg] \dd t \\
    &\leq \liminf_{n \to \infty}  \int_0^T \psi_t \bigg[ \exp\left(-\int_0^t h_s + \eta(\phi_s) \dd s \right)\norm{u_t^n}^2_H - \norm{u_0}_H^2 \bigg] \dd t \\
    &\leq \int_0^T \psi_t \bigg[\int_0^t\exp\left(-\int_0^s h_r + \eta(\phi_r) \dd r \right) \Big\{  -( h_s + \eta(\phi_s)) \left( - \norm{\phi_s}_H^2 + 2\left( u_s,\phi_s \right)_H\right) \Big\} \dd s\\
    &\quad + \int_0^t\exp\left(-\int_0^s h_r + \eta(\phi_r) \dd r \right) \Big\{ 2 \langle  A(s,\phi_s), u_s - \phi_s \rangle_{V^*,V} + 2\langle \overline{A}_s , \phi_s \rangle_{V^*,V} \Big\} \dd s \bigg] \dd t\\
    &\quad + \liminf_{n\to \infty} 2 \int_0^T \psi_t  \mathscr{S}_t( \partial_y F u^n, \dd I^n(u^n)) \dd t.
\end{aligned}
\end{align}

Similarly by Theorem~\ref{thm:Chain-rule-after} with $y_t = \norm{u_t}_H^2$, $b_s = 2\langle \overline{A}_s, u_s \rangle_{V^*,V}$,
\begin{align} \label{eq:u-equation}
\begin{aligned}
     &\exp\left(-\int_0^t h_s + \eta(\phi_s) \dd s \right)\norm{u_t}^2_H-  \norm{u_0}_H^2 \\
    &\hspace{2em} = \int_0^t -( h_s + \eta(\phi_s)) \exp\left(-\int_0^s h_r + \eta(\phi_r) \dd r \right)\norm{u_s}^2_H   \dd s \\
    &\hspace{2em} \quad +  \int_0^t  \exp\left(-\int_0^s h_r + \eta(\phi_r) \dd r \right)2 \langle \overline{A}_s, u_s \rangle_{V^*,V} \dd s \\
    &\hspace{2em} \quad+ 2\mathscr{S}_t( \partial_y F u, \dd I(u)).
\end{aligned}
\end{align}

Combining~\eqref{eq:limit-un-equation} and~\eqref{eq:u-equation}
\begin{align} \label{eq:almost-identified}
\begin{aligned}
    &\int_0^T \psi_t \bigg[\int_0^t  \exp\left(-\int_0^s h_r + \eta(\phi_r) \dd r \right)2 \langle \overline{A}_s -  A(s,\phi_s), u_s - \phi_s \rangle_{V^*,V} \dd s \bigg] \dd t \\
    &\quad \leq  \liminf_{n\to \infty} 2 \int_0^T \psi_t \left[ \mathscr{S}_t( \partial_y F u^n, \dd I^n(u^n))  -\mathscr{S}_t( \partial_y F u, \dd I(u)) \right] \dd t\\
    & \quad+   \int_0^T \psi_t \bigg[ \int_0^t ( h_s + \eta(\phi_s)) \exp\left(-\int_0^s h_r + \eta(\phi_r) \dd r \right)\norm{u_s - \phi_s}^2_H   \dd s\bigg] \dd t.
    \end{aligned}
\end{align}

\underline{Step~\ref{it:identification-04} -- identifying the Young integral:}
Next, we will show that
\begin{align} \label{eq:convergence-young-integral}
     \liminf_{n\to \infty} \int_0^T \psi_t \left[ \mathscr{S}_t( \partial_y F u^n, \dd I^n(u^n))  -\mathscr{S}_t( \partial_y F u, \dd I(u)) \right] \dd t = 0.
\end{align}
Recall that Young integration is a bi-linear operation, cf. Theorem~\ref{thm:young-integral}. Thus, we can decompose
\begin{align*}
    &\int_0^T \psi_t \left[ \mathscr{S}_t( \partial_y F u^n, \dd I^n(u^n))  -\mathscr{S}_t( \partial_y F u, \dd I(u)) \right] \dd t \\
    &=\int_0^T \psi_t \left[ \mathscr{S}_t( \partial_y F u^n, \dd [I^n(u^n) - I(u)])  \right] \dd t \\
    &\quad + \int_0^T \psi_t \left[ \mathscr{S}_t( \partial_y F [u^n - u], \dd I(u))  \right] \dd t\\
    &=: \mathrm{R}_1^n + \mathrm{R}_2^n.
\end{align*}
Let~$\overline{\gamma}$ be given by~\ref{it:H6}.

Ad $\mathrm{R}_1^n$: H\"older's inequality, the embedding $B^{\overline{\gamma}}_{1,\infty} \hookrightarrow L^1_t$, using that $\mathscr{S}_0( \partial_y F u^n, \dd [I^n(u^n) - I(u)]) = 0$, and the stability of Young integration~\eqref{thm:young-integral} imply
\begin{align} \label{eq:R-1-n}
\begin{aligned}
    \mathrm{R}_1^n &\leq \norm{\mathscr{S}_\cdot( \partial_y F u^n, \dd [I^n(u^n) - I(u)])}_{L^1_t } \norm{\psi}_{L^\infty_t} \\
    &\lesssim \seminorm{\mathscr{S}_\cdot( \partial_y F u^n, \dd [I^n(u^n) - I(u)])}_{B^{\overline{\gamma}}_{1,\infty}} \\
    &\lesssim  \norm{ \partial_y F u^n  }_{B^{1/2}_{2,\infty} H} \seminorm{I^n(u^n) - I(u)}_{B^{\overline{\gamma}}_{2,\infty} H}.
    \end{aligned}
\end{align}
Let us for the moment assume that
\begin{align} \label{eq:stable-mult}
    \sup_{n \in \mathbb{N}}\norm{ \partial_y F u^n  }_{B^{1/2}_{2,\infty} H} < \infty.
\end{align}
By Lemma~\ref{lem:existence-approximate-model} and~\ref{lem:existence-of-limits} it holds $u, u^n \in L^\infty_t H \cap B^{1/2}_{2,\infty}H$, and $u_n \to u \in L^2_t H$. Thus, using~\ref{it:H6} we find $I^n(u^n) \to I(u) \in B^{\overline{\gamma}}_{2,\infty}H$, which shows $\liminf_{n\to\infty} \mathrm{R}_1^n = 0$. Therefore, it suffices to show~\eqref{eq:stable-mult}.

Clearly,
\begin{align} \label{eq:stable-01}
    \norm{ \partial_y F u^n  }_{L^\infty_t H} =  \sup_{t \in [0,T]} \norm{ \exp\left( - \int_0^t h_s + \eta(\phi_s) \dd s \right) u^n_t  }_{ H} \leq \norm{u^n}_{L^\infty_t H}.
\end{align}
Notice that
\begin{align*}
    \partial_y F_t u^n_t - \partial_y F_s u^n_s &= \left[ \exp\left( - \int_0^t h_r + \eta(\phi_r) \dd r \right) - \exp\left( - \int_0^s h_r + \eta(\phi_r) \dd r \right) \right] u^n_t \\
    &\quad + \exp\left( - \int_0^s h_r + \eta(\phi_r) \dd r \right) [u^n_t - u^n_s].
\end{align*}
Thus,
\begin{align*}
    &\int_0^{T-h} \norm{ \partial_y F_{t+h} u^n_{t+h} - \partial_y F_t u^n_t}_H^2 \dd t \\
    &\lesssim \norm{u^n}_{L^\infty_t H}^2 \int_0^{T-h} \abs{ \exp\left( - \int_0^{t+h} h_r + \eta(\phi_r) \dd r \right) - \exp\left( - \int_0^t h_r + \eta(\phi_r) \dd r \right)}^2 \dd t \\
    &\quad + \int_0^{T-h} \norm{u^n_{t+h} - u^n_t}_H^2 \dd t
\end{align*}
The fundamental theorem implies
\begin{align*}
    &\exp\left( - \int_0^{t+h} h_r + \eta(\phi_r) \dd r \right) - \exp\left( - \int_0^t h_r + \eta(\phi_r) \dd r \right) \\
    &\hspace{4em}= -\int_t^{t+h} \exp\left( - \int_0^{s} h_r + \eta(\phi_r) \dd r \right)  ( h_s + \eta(\phi_s)) \dd s.
\end{align*}
This together with estimating one product of the square by its supremal value, and Fubini's theorem show
\begin{align*}
&\int_0^{T-h} \abs{ \exp\left( - \int_0^{t+h} h_r + \eta(\phi_r) \dd r \right) - \exp\left( - \int_0^t h_r + \eta(\phi_r) \dd r \right)}^2 \dd t \\
&\hspace{4em} \lesssim \int_0^{T-h} \abs{ \exp\left( - \int_0^{t+h} h_r + \eta(\phi_r) \dd r \right) - \exp\left( - \int_0^t h_r + \eta(\phi_r) \dd r \right)} \dd t \\
&\hspace{4em} \lesssim h  \int_0^T \left(h_s + \eta(\phi_s) \right) \dd s,
\end{align*}
which allows us to verify
\begin{align*}
    \int_0^{T-h} \norm{ \partial_y F_{t+h} u^n_{t+h} - \partial_y F_t u^n_t}_H^2 \dd t \lesssim h \left( \norm{u^n}_{L^\infty_t H}^2  \int_0^T \left(h_s + \eta(\phi_s) \right) \dd s + \seminorm{u^n}_{B^{1/2}_{2,\infty}H}^2 \right).
\end{align*}
Finally, we conclude
\begin{align} \label{eq:stable-02}
    \seminorm{\partial_y F u^n}_{B^{1/2}_{2,\infty} H}^2 \lesssim \left( \norm{u^n}_{L^\infty_t H}^2  \int_0^T \left(h_s + \eta(\phi_s) \right) \dd s + \seminorm{u^n}_{B^{1/2}_{2,\infty}H}^2 \right).
\end{align}
The assertion~\eqref{eq:stable-mult} follows from~\eqref{eq:stable-01},~\eqref{eq:stable-02} and the a priori bound~\eqref{eq:uniform-estimate}.

Ad $\mathrm{R}_2^n$: We aim to use the weak-star convergence of $u^n$ to $u$ for the identification. To do so, we first regularize the integrand $I(u)$ so that the Young integral can be identified as a classical Bochner integral. The Bochner integral is handled by weak-star convergence, whilst the correction terms are small as $W^{1,\infty}_t H \hookrightarrow B^{\gamma}_{2,\infty} H$ is dense.

Let $\delta \in (1/2,\overline{\gamma})$ and $\varepsilon > 0$. Recall that $B^{\gamma}_{2,\infty} \hookrightarrow B^{\delta}_{2,2} \hookrightarrow B^{\delta}_{2,\infty}$, cf.~\cite[Theorem~16]{MR1108473}. Thus, $I(u) \in B^{\delta}_{2,2} H$ and there exists $J^\varepsilon \in W^{1,\infty}_t H$ such that 
\begin{align}
    \norm{I(u) - J^\varepsilon}_{B^{\delta}_{2,\infty}} \lesssim \norm{I(u) - J^\varepsilon}_{B^{\delta}_{2,2}} \leq \varepsilon.
\end{align}
Next, we split, using the linearity of $\mathscr{S}$,
\begin{align*}
\mathrm{R}_2^n &= \int_0^T \psi_t \left[ \mathscr{S}_t( \partial_y F [u^n - u], \dd [I(u) - J^\varepsilon])  \right] \dd t \\
&\quad + \int_0^T \psi_t \left[ \mathscr{S}_t( \partial_y F [u^n - u], \dd J^\varepsilon)  \right] \dd t\\
&=: \mathrm{R}^n_{2,a} + \mathrm{R}^n_{2,b}. 
\end{align*}

Ad $\mathrm{R}_{2,a}^n$: Analogously to~\eqref{eq:R-1-n} it holds
\begin{align*}
    \mathrm{R}_{2,a}^n \lesssim  \norm{ \partial_y F [u^n-u] }_{B^{1/2}_{2,\infty} H} \seminorm{I(u) - J^\varepsilon}_{B^{\delta}_{2,\infty} H}.
\end{align*}
Also, similarly to~\eqref{eq:stable-mult}, we verify
\begin{align*}
    \norm{ \partial_y F [u^n-u] }_{B^{1/2}_{2,\infty} H}^2 \lesssim \norm{u^n - u}_{L^\infty_t H}^2\left( 1 +  \int_0^T \left(h_s + \eta(\phi_s) \right) \dd s \right) + \seminorm{u^n-u}_{B^{1/2}_{2,\infty}H}^2,
\end{align*}
which is finite uniformly in $n$ by Lemma~\ref{lem:uniform-estimate} and~\ref{lem:existence-of-limits}. This implies
\begin{align} \label{eq:R-2-a-n}
   \liminf_{n \to \infty} \mathrm{R}_{2,a}^n \leq C \varepsilon.
\end{align}

Ad $\mathrm{R}_{2,b}^n$: Using Lemma~\ref{lem:Identification-Young} (with $u = \partial_y F [u^n - u]$ and $I = J^\varepsilon$) and integration by parts (recall that $\overline{\psi}(T) = 0$), we rewrite
\begin{align*}
    \mathrm{R}_{2,b}^n &= \int_0^T \psi_t \left[ \int_0^t  \exp\left(-\int_0^r h_s + \eta(\phi_s) \dd s \right) \left(  u^n_r - u_r,\partial_t J^\varepsilon_r \right)_H \dd r \right] \dd t \\
    &=  \int_0^T \overline{\psi}_r  \exp\left(-\int_0^r h_s + \eta(\phi_s) \dd s \right) \left(  u^n_r - u_r,\partial_t J^\varepsilon_r \right)_H \dd r \\
    &=  \iota_{u^n - u}(Z),
\end{align*}
where~$\iota$ is defined similarly to~\eqref{eq:duality-map}, and $Z_r = \overline{\psi}_r  \exp\left(-\int_0^r h_s + \eta(\phi_s) \dd s \right) \partial_t J^\varepsilon_r$. Observe that
\begin{align*}
    \norm{Z}_{L^1_t H} \leq T \norm{\overline{\psi}}_{L^\infty_t} \norm{\partial_t J^\varepsilon}_{L^\infty_t H} < \infty.
\end{align*}
Finally, we are ready to use the weak-star convergence of $u^n$ towards $u$, cf.~\eqref{eq:conv-linfH}, to conclude
\begin{align} \label{eq:R-2-b-n}
    \liminf_{n \to \infty}  \mathrm{R}_{2,b}^n =  \liminf_{n \to \infty} \iota_{u^n - u}(Z) = 0.
\end{align}

Combining~\eqref{eq:R-2-a-n} and~\eqref{eq:R-2-b-n} establish
\begin{align*}
    \liminf_{n\to \infty} \mathrm{R}^n_2 \leq \liminf_{n\to \infty} \mathrm{R}^n_{2,a} + \liminf_{n\to \infty} \mathrm{R}^n_{2,b} \leq C \varepsilon.
\end{align*}
Passing with $\varepsilon \to 0$ verifies~\eqref{eq:convergence-young-integral}.

\underline{Step~\ref{it:identification-05} -- identifying the monotone operator:}
Recall that $u \in L^\alpha_t V \cap C_t H \cap B^{1/2}_{2,\infty} H$. Let $\varepsilon >0$. Choosing $\phi = u + \varepsilon \xi$ in~\eqref{eq:almost-identified}, where $\xi  = \zeta v$ for $\zeta \in L^\infty_t$ and $v \in V$. Dividing by $\varepsilon$ and passing with $\varepsilon \to 0$ verify the following:
\begin{align*}
    & \lim_{\varepsilon \to 0} \int_0^T \psi_t \bigg[  \int_0^t  \exp\left(-\int_0^s h_r + \eta(u_r + \varepsilon \xi_r) \dd r \right) \\
    &\hspace{6em} 2 \langle \overline{A}_s -  A(s,u_s + \varepsilon \xi_s), \xi_s \rangle_{V^*,V} \dd s \bigg] \dd t  \leq 0.
\end{align*}
By repeating the argument for $\varepsilon <0$ we derive 
\begin{align*}
    & \lim_{\varepsilon \to 0} \int_0^T \psi_t \bigg[  \int_0^t  \exp\left(-\int_0^s h_r + \eta(u_r + \varepsilon \xi_r) \dd r \right) \\
    &\hspace{6em} 2 \langle \overline{A}_s -  A(s,u_s + \varepsilon \xi_s), \xi_s \rangle_{V^*,V} \dd s \bigg] \dd t = 0.
\end{align*}
Since $\psi$ was arbitrary, we conclude for all $t \in [0,T]$
\begin{align*}
        & \lim_{\varepsilon \to 0} \int_0^t  \exp\left(-\int_0^s h_r + \eta(u_r + \varepsilon \xi_r) \dd r \right) 2 \zeta_s \langle \overline{A}_s -  A(s,u_s + \varepsilon \xi_s), v \rangle_{V^*,V} \dd s = 0.
\end{align*}
Similarly, since $\zeta$ was arbitrary for almost all $s\in [0,T]$,
\begin{align*}
    \lim_{\varepsilon \to 0}   \exp\left(-\int_0^s h_r + \eta(u_r + \varepsilon \xi_r) \dd r \right) \langle \overline{A}_s -  A(s,u_s + \varepsilon \xi_s), v \rangle_{V^*,V} = 0.
\end{align*}

Moreover, using the local boundedness of $\eta$,
\begin{align*}
   \inf_{\abs{\varepsilon}\ll 1} \exp\left(-\int_0^s h_r + \eta(u_r + \varepsilon \xi_r) \dd r \right) \geq \exp\left(- [\norm{h}_{L^1_t} + \sup_{z \in B_V}\eta(z)] \right) > 0,
\end{align*}
where $B_V$ is a ball in $V$ such that $r \mapsto u_r \pm \norm{\zeta}_{L^\infty_t} v \in B_V$. This implies for all $v \in V$ and almost all $s \in [0,T]$
\begin{align} \label{eq:one-limit}
    \lim_{\varepsilon \to 0}  \langle \overline{A}_s -  A(s,u_s + \varepsilon \xi_s), v\rangle_{V^*,V}  = 0.
\end{align}
The hemicontinuity~\ref{it:H1} establishes for all $v \in V$ and almost all $s \in [0,T]$
\begin{align} \label{eq:second-limit}
   \lim_{\varepsilon \to 0} \langle A(s,u_s)  - A(s,u_s + \varepsilon \xi_s ), v\rangle_{V^*,V} = 0.
\end{align}
Combining~\eqref{eq:one-limit} and~\eqref{eq:second-limit} shows for all $v \in V$ and almost all $s \in [0,T]$
\begin{align*}
    &\langle \overline{A}_s -  A(s,u_s) , v\rangle_{V^*,V} \\
    &= \lim_{\varepsilon \to 0} \langle \overline{A}_s - A(s,u_s + \varepsilon \xi_s ), v\rangle_{V^*,V} + \lim_{\varepsilon \to 0} \langle A(s,u_s + \varepsilon \xi_s )-  A(s,u_s) , v\rangle_{V^*,V} =0.
\end{align*}
Since $v \in V$ was arbitrary, the identification is complete.
\end{proof}

\begin{rem}
The estimate~\eqref{eq:main-result-estimate} follows from~\eqref{eq:uniform-estimate} using the weak lower semi-continuity of the norms.
\end{rem}

\section{Proofs: Integral operators and Young regimes} \label{sec:proofs-young-regimes}
In this section, we verify the theorems presented in Section~\ref{sec:main-integral-operator}.

\subsection{Additive Young regime}
\begin{proof}[Proof of Theorem~\ref{thm:AnwerAdditive}]
Let $\gamma > 1/2$ and $Z \in C^\gamma_t H$. Define $I_t(u) \equiv Z_t - Z_0 $.

We omit the verification of Assumption~\ref{ass:integral-operator} for the operator~$I$ since it is trivial. The existence of a weak solution follows by Theorem~\ref{thm:main}. Note that besides the locally monotone operator $A$, no non-linear expression of $u$ appears in the analysis, meaning we do not require the strong convergence \eqref{eq:conv-strong} along subsequences. Thus, we do not require the Aubin-Lions lemma and therefore we may drop Assumption \ref{ass:compact-Gelfand} in this setting. 

Next, we discuss uniqueness: let $u$ and $v$ be solutions to~\eqref{problem def} started in $u_0$ and $v_0$, respectively. Subtracting~\eqref{eq:weak-solution} for $u$ and $v$ yield for all $w \in V$ and $t \in [0,T]$
\begin{align*}
  \left(u_t - v_t - u_0 - v_0, w \right)_H = \int_0^t \langle A(s,u_s) - A(s,v_s), w \rangle_{V^*,V} \dd s.
\end{align*}
Applying Theorem~\ref{chain rule norm} with $X_t = u_t-v_t$, $Y_t = A(t,u_t) - A(t,v_t)$ and $I_t = 0$ shows
\begin{align*}
    \norm{u_t - v_t}^2_H &= \norm{u_0 - v_0}^2_H + \int_0^t 2  \langle A(s,u_s) - A(s,v_s), u_s - v_s \rangle_{V^*,V} \dd s \\
     &\leq \norm{u_0 - v_0}^2_H + \int_0^t (h_s+\eta(u_s))\norm{u_s-v_s}_H^2 \dd s,
\end{align*}
where we used~\ref{it:H2} in the last step. The assertion follows by invoking Gronwall's lemma.
\end{proof}

\subsection{Abstract Young integrals}
\label{abstract young integral}
Let $\sigma: H\to H$ be such that 
\begin{equation}
    \norm{\sigma(u)}_H\leq C(1+\norm{u}_H), \qquad \norm{\sigma(u)-\sigma(v)}_H\leq L \norm{u-v}_H.
    \label{sigma conditions}
\end{equation}
We want to construct the abstract Young integral 
\[
I_t(u):=\int_0^t \sigma(u_s) \dd X_s
\]
for some given path $X$. Note that since locally, $I$ should behave as $\sigma(u_s)(X_t-X_s)$, we need to first ensure that $X$ is of such spatial regularity that $\sigma(u_s)(X_t-X_s)\in H$; in this context, we introduced in Definition~\ref{multiplicative ideal} the concept of multiplier. A multiplier is a sufficient condition for the spatial compatibility of abstract Young integration.

\begin{lemma}
\label{abstract young lemma}
Let $H$ be a Hilbert space and $E$ a multiplier in $H$ in the sense of Definition \ref{multiplicative ideal}. Suppose $\sigma:H\to H$ satisfies \eqref{sigma conditions}. Let $X\in C^\gamma_t E$. 

Then for any $q\in (2, \infty)$ such that $\gamma+1/q>1$, the mapping
\begin{equation}
    \begin{split}
        L^\infty_t H\cap B^{1/2}_{2, \infty}H&\to B^\gamma_{q, \infty}H\\
        u&\to I(u):=\int_0^{(\cdot)}\sigma(u_s) \dd X_s
    \end{split}
\end{equation}
is well-defined and satisfies
\begin{align} \label{eq:estimate-abstract-Young}
\seminorm{I(u)}_{B^\gamma_{q, \infty} H}\lesssim(C+L) \seminorm{X}_{C^\gamma_t E}(1+\norm{u}_{L^\infty_tH}+\seminorm{u}_{B^{1/2}_{2, \infty}H}).
\end{align}
In particular, for any sequence $(X^n)_n\subset C^1_tE$ such that $X^n\to X$ in $C^\gamma_t E$ the approximation $I^n(u):=\int_0^t \sigma(u_s) \dot{X}^n_s \dd s$ satisfies condition \ref{it:H5} in Assumption \ref{ass:integral-operator}.
\label{verify H5}
\end{lemma}
\begin{rem}
    Note that we start with an element $u$ on integrability scale $2$, yet the operator needs to enjoy integrability scale $q$, i.e., we need to gain integrability. This is possible, but at the price of imposing more regularity. Indeed, we use the embedding $B^{1/2}_{2, \infty}\hookrightarrow B^{1/q}_{q, \infty}$, which ensures such a trade in order to gain integrability. The price we pay in regularity becomes apparant in the condition $\gamma+1/q>1$, in that it becomes more restrictive in $\gamma$, the larger we intend to choose $q$. 
\end{rem}

\begin{proof}
Let $u \in L^\infty_t H\cap B^{1/2}_{2, \infty}H$ and $X \in C_t^\gamma E$.

We aim to apply the sewing lemma, Lemma~\ref{sewing}, to the local approximation
\begin{align*}
    G_{s,t}=\langle \sigma(u)\rangle_{s,t}(X_t-X_s),
\end{align*}
where $\mean{u}_{s,t} := \abs{t-s}^{-1} \int_s^t u_r \dd r$ is the average of $u$ on $[s,t]$. Ultimately, we define the abstract Young integral as the sewing of the local approximations, i.e., for all $t \in [0,T]$
\begin{align*}
    I_t(u) := \mathscr{I}_t[G].
\end{align*}
It remains to verify the sew-ability of~$G$; in other words, we need to check the assumptions of the sewing lemma.

Estimating $G$: Let $t \in [0,T]$ and $h \in [0,T-t]$. Note that by \eqref{sigma conditions}, we have 
\[
\norm{G_{t, t+h}}_H\lesssim \norm{\langle\sigma(u)\rangle_{t,t+h}}_H\norm{X_{t+h}-X_t}_E\lesssim C(1+\norm{u}_{L^\infty_t H})\seminorm{X}_{C^\gamma_tE}h^\gamma,
\]
from which we conclude that 
\[
\norm{G}_{\mathbb{B}^\gamma_{\infty, \infty}H}\lesssim C(1+\norm{u}_{L^\infty_t H})\seminorm{X}_{C^\gamma_tE}.
\]

Estimating $\delta G$: Let $s\leq v \leq t$. Moreover, we have 
\begin{equation*}
    \begin{split}
        (\delta G)_{s, v, t}&=\left(\langle \sigma(u)\rangle_{s, t}- \langle \sigma(u)\rangle_{s, v}\right)(X_t-X_s)+\left(\langle \sigma(u)\rangle_{s, v}- \langle \sigma(u)\rangle_{v, t}\right)(X_t-X_v)\\
        &=\left(\langle \sigma(u)\rangle_{s, t}- \sigma(u_s)\right)(X_t-X_s)-\left(\langle \sigma(u)\rangle_{s, v}- \sigma(u_s)\right)(X_t-X_s)\\
        &+\left(\langle \sigma(u)\rangle_{s, v}- \sigma(u_v)\right)(X_t-X_v)-\left(\langle \sigma(u)\rangle_{v, t}- \sigma(u_v)\right)(X_t-X_v).
    \end{split}
\end{equation*}
Let us further bound only the first of the above four expressions (the other ones follow by similar considerations). By Lemma \ref{abstract average minus leftpoint} we have
\begin{equation*}
    \begin{split}
      &\left( \int_0^{T-h}\norm{\left(\langle \sigma(u)\rangle_{s, s+h}- \sigma(u_s)\right)(X_{s+h}-X_s)}_{H}^q \dd s\right)^{1/q}\\
       &\lesssim \seminorm{X}_{C^\gamma_t E}h^\gamma \left(\int_0^{T-h}\norm{\left(\langle \sigma(u)\rangle_{s, s+h}- \sigma(u_s)\right)}_H^{q} \dd s\right)^{1/q}\\
       &\lesssim h^\gamma \seminorm{X}_{C^\gamma_tE}h^{1/q}L(2\norm{u}_{L^\infty_tH})^{1-2/q}\seminorm{u}_{B^{1/2}_{2, \infty}H}^{2/q}.
    \end{split}
\end{equation*}
Bounding the remaining terms in a similar fashion, we obtain 
\[
\bar{\Omega}_q(\delta G, h)\lesssim  h^{\gamma+1/q}\seminorm{X}_{C^\gamma_tE}L\seminorm{u}_{B^{1/2}_{2, \infty}H}.
\]
We are therefore able to apply the Besov Sewing Lemma \ref{sewing} with $\alpha = \gamma$, $p_1 = \infty$, $p_2= q$ and $\gamma = \gamma + 1/q$. Assertion~\eqref{eq:estimate-abstract-Young} follows from~\eqref{est:sewed-seminorm}; indeed, 
\begin{align*}
    \seminorm{I(u)}_{B^{\gamma}_{q,\infty} H} =  \seminorm{\mathscr{I}G }_{B^{\gamma}_{q,\infty} H} &\lesssim \norm{G}_{\mathbb{B}^\alpha_{\infty, \infty}H} + T^{1/q}\norm{\delta G}_{\bar{\mathbb{B}}^{\gamma+1/q}_{q, \infty}H} \\
    &\lesssim  C(1+\norm{u}_{L^\infty_t H})\seminorm{X}_{C^\gamma_tE} + T^{1/q} \seminorm{X}_{C^\gamma_tE}L\seminorm{u}_{B^{1/2}_{2, \infty}H}.
\end{align*}

\end{proof}
\begin{lemma}
\label{verify H6}
    Suppose the setting of Lemma \ref{abstract young lemma} and let $(X^n)_n\subset C^1_tE$ be such that $X^n\to X$ in $C^\gamma_tE$. For $q\geq 2$ and $\gamma>1-1/q$, let $I(u)\in B^{\gamma}_{q, \infty}H$ be the abstract Young integral as defined in Lemma~\ref{abstract young integral} and $I^n(u)=\int_0^{\cdot}\sigma(u_s)\dot{X}^n_s \dd s$ be its approximate. Then the continuity property \ref{it:H6} in Assumption \ref{ass:integral-operator} holds. 
    \end{lemma}
    \begin{proof}
    Let us start by pointing out that in the notation of Assumption \ref{ass:integral-operator}, we have $b^n(t, u)=\sigma(u)\dot{W}^n_t$. By \eqref{sigma conditions}, it follows readily that 
 \begin{align*}
     \norm{b^n(t, v)}_H&\lesssim C\norm{X^n}_{C^1_tE}(1+\norm{v}_H), \\
     \norm{b^n(t, v)-b^n(t, u)}_H &\lesssim L\norm{X^n}_{C^1_tE}\norm{u-v}_H.
 \end{align*}
 We now take $v, v^n\in L^\infty_tH\cap B^{1/2}_{2, \infty}H$ such that $v^n\to v$ in $L^2_tH$. We need to show that $I^n(v_n)\to I(v)$ in $B^{\gamma}_{q, \infty}H$.
    Note that due to the linearity of the integral in the integrator, we have 
    \[
    \seminorm{I^n(v^n)-I(v^n)}_{B^\gamma_{q, \infty}}\lesssim \seminorm{X-X^n}_{C^\gamma_tE}(1+\norm{v^n}_{L^\infty_tH}+\seminorm{v^n}_{B^{1/2}_{2, \infty}H})\to 0.
    \]
   It thus remains to show $I(v^n)\to I(v)$ in $B^\gamma_{q, \infty}$, where we recall that $I$ is constructed as in Lemma \ref{abstract young lemma}. Towards this end, we employ Lemma \ref{sewing convergence}, by which it suffices to show that for $G^n_{s,t}=\langle \sigma(v^n)\rangle_{s,t}(X_t-X_s)$ and $G_{s,t}=\langle \sigma(v)\rangle_{s,t}(X_t-X_s)$ we have $\norm{G^n-G}_{\mathbb{B}^\gamma_{q, \infty}}\to 0$. By Lemma \ref{abstract limit identification lemma}, we have 
   \begin{align*}
      &\left( \int_0^{T-h}\norm{G^n_{r, r+h}-G_{r, r+h}}_H^q \dd r\right)^{1/q}\\
      &\lesssim h^\gamma \seminorm{W}_{C^\gamma_tH}\left(\int_0^{T-h}\norm{\langle \sigma(v^n)\rangle_{r,r+h}- \langle \sigma(v)\rangle_{r,r+h}}_H^q \dd r\right)^{1/q}\\
      &\lesssim  h^\gamma \seminorm{W}_{C^\gamma_tH} (\norm{v^n}_{L^\infty_t H} + \norm{v}_{L^\infty_t H})^{1-2/q} L \norm{v^n-v}_{L^2_t H}^{2/q},
   \end{align*}
  and therefore $\norm{G^n-G}_{\mathbb{B}^\gamma_{q, \infty}}\to 0$. By Lemma \ref{sewing convergence}, we conclude that $I(v^n)\to I(v)$ in $B^\gamma_{q, \infty}H$.
\end{proof}


\begin{proof}[Proof of Theorem~\ref{thm:Existence-abstract-Young}]
Let $I$, $I^n$ be as in Lemma \ref{verify H5}. 
  By Lemma \ref{verify H5}, the Young integral is a well defined element of $B^\gamma_{q, \infty}H$ provided $\gamma>1-1/q$. Note that Assumption \ref{ass:integral-operator} and Theorem \ref{thm:main} on the other hand require $\gamma>1/2+1/q$. Both constraints are satisfied for $q=4$ and $\gamma>3/4$. In this case, Lemma \ref{verify H5} and \ref{verify H6} assure that Assumption \ref{ass:integral-operator} is satisfied. We may therefore conclude by Theorem \ref{thm:main}.
\end{proof}

\subsection{Linear multiplicative Young regime}
\begin{proof}[Proof of Theorem~\ref{thm:Multiplicative}]
For existence, note that $E=\mathbb{R}$ is always a multiplicative ideal in $H$ in the sense of Definition~\ref{multiplicative ideal} and $\sigma(u)=u$ obviously satisfies \eqref{sigma conditions}. Hence existence of weak solutions follows from Theorem~\ref{thm:Existence-abstract-Young}.

Let $u$ and $v$ be two weak solutions to~\eqref{problem def} starting in $u_0$ and $v_0$, respectively. Using Theorem~\ref{chain rule norm} with $X_t = u_t -v_t$, $Y_t = A(t,u_t) - A(t,v_t)$ and $I_t = I_t(u-v)$ (notice that we used the linearity of $I$ at this point); and subtracting the resulting equations for $t > s$ yield
\begin{align*}
    \norm{u_t-v_t}^2_H&=  \norm{u_s-v_s}^2_H+\int_s^t\langle A(r, u_r)-A(r, v_r), u_r-v_r\rangle_{V^*, V} \dd r\\
    &+\mathscr{S}_{t}(u-v,  \dd I(u-v) ) - \mathscr{S}_{s}(u-v,  \dd I(u-v) ).
\end{align*}
Remark that $Z_t:=\norm{u_t-u_t}^2_H$ satisfies
\[
|Z_t-Z_s|\leq 2(\norm{u}_{C_tH}+\norm{v}_{C_tH})(\norm{u_t-u_s}_H+\norm{v_t-v_s}_H)
\]
by which we conclude that $Z\in B^{1/2}_{2, \infty}\mathbb{R}$. Hence, using the sewing Lemma \ref{sewing}, it can be easily verified that 
\begin{align*}
    \mathscr{S}_{t}(u-v,  \dd I(u-v) ) - \mathscr{S}_{s}(u-v,  \dd I(u-v) ) = \int_s^t \norm{u_r-v_r}_H^2 \dd \beta_r.
\end{align*}
Due to~\ref{it:H2-alter} it holds moreover that $\int_s^t\langle A(r, u_r)-A(r, v_r), u_r-v_r\rangle_{V^*, V} \dd r \leq 0$. Thus, the above considerations lead to the inequality: for all $s\leq t$
\begin{equation}
    Z_t\leq Z_s+\int_s^t Z_r \dd \beta_r. 
    \label{z inequality}
\end{equation}
Remark that by definition $Z\geq 0$. Let us further define
\[
B_t:=\int_0^t Z_r \dd \beta_r-Z_t.
\]
By \eqref{z inequality}, $B$ is monotone increasing. Hence, the integral 
\[
\int_0^t \exp{(-\beta_r)} \dd B_r\geq 0
\]
is well defined. Using Lemma \ref{classical rules of calculus}, we obtain
\begin{align*}
    0&\leq \int_0^t \exp{(-\beta_r)}\dd B_r=\int_0^t  \exp{(-\beta_r)} Z_r \dd \beta_r- \int_0^t\exp{(-\beta_r)}\dd Z_r\\
    &=-\int_0^t \dd (Z_r\exp{(-\beta_r)}) = -Z_t\exp{(-\beta_t)}+Z_0\exp{(-\beta_0)},
\end{align*}
which finishes the proof.
\end{proof}

\section{Examples} \label{sec:examples}
In this section we present three equations that fall in the general framework of Theorem~\ref{thm:main}: the $p$-Laplace equation, the porous medium equation and an equation that models shear-thickening fluids. We derive the existence of weak solutions for additive drivers and Young integrals. Moreover, we show that our framework captures regularization by noise effects for the $p$-Laplace equation.

\subsection{The p-Laplace equation}
\label{sec: p-laplace}
Let $\Lambda\subset \mathbb{R}^d$ be a bounded Lipschitz domain and $p\in (\frac{2d}{d+2}, \infty)$.
\subsubsection{The Gelfand triple} We set $V=W^{1, p}_0(\Lambda)$ and $H=L^2(\Lambda)$.

Since $p>\frac{2d}{d+2}$, the Rellich-Kondrachov theorem ensures that the embedding $W^{1, p}_0(\Lambda)\hookrightarrow L^2(\Lambda)$ is compact. Thus, Assumption~\ref{ass:compact-Gelfand} is satisfied.

\subsubsection{The monotone operator}
The $p$-Laplace operator $\Delta_p: V\to V^*$ is defined by 
    \begin{equation}
        \begin{split}
           \langle \Delta_p u, v\rangle_{V^*,V}:=-\int_\Lambda |\nabla u|^{p-2} \nabla u\cdot \nabla v \dd x.
           \label{p-laplace}
        \end{split}
    \end{equation}
    
It is now well-known that $\Delta_p$ satisfies~\ref{it:H1};~\ref{it:H2-alter}; \ref{it:H3} with $c_1=1, c_2=0, f=0$; and~\ref{it:H4} with $c_3=1, g=0$ (see \cite[Example 4.1.9]{Liu2015}). Thus, Assumption~\ref{ass:monotone-operator} holds.

\begin{theorem}\label{thm:p-Laplace}
Let $T > 0$ and $u_0 \in L^2(\Lambda)$. Let $I$ be given by one of the following:
\begin{enumerate}
    \item (Additive driver) \label{it:additive-p-lap} Let $\gamma > 1/2$ and $Z \in C^\gamma(0,T;L^2(\Lambda))$. 
    
    Define $I_t(u) := Z_t - Z_0$;
    \item (Young integral) Let $\gamma > 3/4$, $X\in C^\gamma(0,T; L^\infty(\Lambda))$ and $\sigma:\mathbb{R}\to \mathbb{R}$ satisfy
    \[
    |\sigma(x)|\leq C(1+|x|), \qquad |\sigma(x)-\sigma(y)|\leq L|x-y|.
    \]
    
    Define $I_t(u) := \int_0^t \sigma(u_s) \dd X_s$ as the Young integral given by Lemma~\ref{abstract young lemma};
\end{enumerate}
    Then there exists a (in Case~\ref{it:additive-p-lap} unique) weak solution
    \begin{align*}
        u \in C([0,T]; L^2(\Lambda)) \cap B^{1/2}_{2,\infty}(0,T; L^2(\Lambda)) \cap L^p(0,T; W^{1,p}_{0}(\Lambda))
    \end{align*}
    to
    \begin{align*}
        \dd u - \Delta_p u \dd t = \dd I_t(u), \qquad u(0) = u_0.
    \end{align*}
\end{theorem}

\begin{proof}
    The first part of the statement follows directly from Theorem~\ref{thm:AnwerAdditive}. 
    
    For the second part of the statement, note that $\sigma$ gives rise to a Nemytskii operator in $H=L^2(\Lambda)$ satisfying \eqref{sigma conditions}. Moreover, $E=L^\infty(\Lambda)$ is a multiplicative ideal in $H=L^2(\Lambda)$ in the sense of Definition \ref{multiplicative ideal}. Hence, we may conclude by Theorem~\ref{thm:Existence-abstract-Young}.  
\end{proof}

\subsubsection{Regularization by noise} It is well-known that randomness can lead to regularization; an effect nowadays called regularization by noise. In particular, the translation with respect to a rough process~$w$ is regularizing. One main tool behind the regularization effect is a formula that relates classical integration of differences and the sewing of a convolution with respect to the local time~$L^w$ of $w$ (for brief overview of its basic definition and some properties, refer to the corresponding section in the Appendix \ref{appendix}) evaluated at a time-averaged function value $\mean{u}$ of $u$:
\begin{align} \label{eq:main-formula-reg-by-noise}
    \int_0^t b(u_s - w_s) \dd s = \mathscr{I}_t \big( b * L^w (\mean{u}) \big).
\end{align}
Let us briefly motivate \eqref{eq:main-formula-reg-by-noise}. Since we assume $w$ to be a highly oscillating noise, we would expect that on small time scales, i.e. $|t-s|\ll 1$, we have 
\[
\int_s^tb(u_r-w_r)\dd r\simeq \int_s^tb(\langle u\rangle_{s, t}-w_r) \dd r=(b*L^w_{s,t})(\langle u\rangle_{s,t})
\]
where in the last step, we used the occupation times formula. Sewing together these small scale approximation, one can recover \eqref{eq:main-formula-reg-by-noise}, provided $L^w$ is sufficiently regular with respect to $u$ and $b$.
Note in particular that the right-hand-side can in principle still makes sense for irregular~$b$; in other words, $w$ can regularizes~\eqref{eq:main-formula-reg-by-noise} through the regularity of its local time. 

Our abstract framework captures this regularization effect by choosing
\begin{align*}
    I_t(u) := \mathscr{I}_t \big( b * L^w (\mean{u}) \big).
\end{align*}
It generalizes our results on regularization by noise for the $p$-Laplace equation~\cite{Bechtold2023} to a broader class of singularities $b$.

The next lemma justifies~\eqref{eq:main-formula-reg-by-noise}.
\begin{lemma} \label{verify nonlinear H5} 
Let $q \in [2,\infty]$, $\gamma > 1-1/q$, $b \in \mathscr{S}'(\mathbb{R})$ be a Schwartz distribution, and $w:[0,T] \to \mathbb{R}$ such that its local time $L^w$ satisfies $b * L^w \in C^\gamma(0,T;C^{0,1}(\mathbb{R}))$. 

 Then the generalized integral
   \begin{equation}
       \begin{split}
           L^\infty(0,T; L^2(\Lambda))\cap B^{1/2}_{2, \infty}(0,T;L^2(\Lambda)) &\to B^\gamma_{q, \infty}(0,T;L^2(\Lambda))\\
           u&\to I(u)=  \mathscr{I} \big( b * L^w (\mean{u}) \big),
       \end{split}
   \end{equation}
   where $\mathscr{I}_t \big( b * L^w (\mean{u}) \big)$ is the sewing of the local approximation $G_{s,t} = b* L_{s,t}^w( \mean{u}_{s,t})$ defined in Lemma~\ref{sewing}, is well-defined and satisfies 
   \begin{align} \label{eq:Integral-reg-by-noise}
   \begin{aligned}
       &\seminorm{I(u)}_{B^\gamma_{q, \infty}(0,T;L^2(\Lambda))}\\
       &\hspace{2em} \lesssim \norm{b*L^w}_{C^\gamma(0,T;C^{0,1}(\mathbb{R}))}(1+\norm{u}_{L^\infty(0,T;L^2(\Lambda))}+\seminorm{u}_{B^{1/2}_{2, \infty}(0,T;L^2(\Lambda))}).
        \end{aligned}
   \end{align}

If additionally $b$ is more regular, i.e., it satisfies
\begin{equation}
        \label{b nice}
         |b(x)|\leq C(1+|x|), \qquad |b(x)-b(y)|\leq L |x-y|,
    \end{equation}
then the generalized integral has the alternative representation: for all $t \in [0,T]$
\begin{align*}
    I_t(u) = \int_0^t b(u_s - w_s) \dd s.
\end{align*}
\end{lemma}
\begin{proof}
To prove the first part, we need to show that the germ $G_{s,t}=(b*L^w_{s,t})(\langle u\rangle_{s,t})$ admits a sewing in $L^2(\Lambda)$.

Firstly, for $r \in [0,T]$ and $h \in [0,T-r]$,  
\begin{align*}
\norm{G_{r, r+h}}_{L^2(\Lambda)}&=\norm{(b*L^w_{r,r+h})(\langle u\rangle_{r,r+h})}_{L^2(\Lambda)} \\
&\lesssim \norm{b *L^w_{r,r+h}}_{L^\infty(\mathbb{R})}\leq \norm{b*L^w}_{C^\gamma(0,T; C^{0,1}(\mathbb{R}))} h^\gamma,
\end{align*}
from which we conclude that 
\[
\norm{G}_{\mathbb{B}^\gamma_{\infty, \infty}(0,T;L^2(\Lambda))}\lesssim  \norm{b*L^w}_{C^\gamma(0,T; C^{0,1}(\mathbb{R}))}.
\]

Secondly, we have 
\begin{align*}
    (\delta G)_{s,v,t}&=(b*L^w_{s,v})(\langle u\rangle_{s,t})-(b*L^w_{s,v})(\langle u\rangle_{s,v})+(b*L^w_{v,t})(\langle u\rangle_{s,t})-(b*L^w_{v,t})(\langle u\rangle_{v,t})\\
    &=(b*L^w_{s,v})(\langle u\rangle_{s,t})-(b*L^w_{s,v})(u_s)+(b*L^w_{s,v})(u_s)-(b*L^w_{s,v})(\langle u\rangle_{s,v})\\
    &+(b*L^w_{v,t})(\langle u\rangle_{s,t})-(b*L^w_{v,t})(u_t)+(b*L^w_{v,t})(u_t)-(b*L^w_{v,t})(\langle u\rangle_{v,t})
\end{align*}

For the first difference term, we employ Lemma \ref{abstract average minus leftpoint} for $H=L^2(\Lambda)$ and $\sigma(x)=(b*L^w_{r, r+\theta h})(x)$ with $\theta\in [0, 1]$ to obtain
\begin{align*}
    &\left( \int_0^{T-h} \norm{(b^n*L_{r,r+\theta h})(\langle u\rangle_{r, r+h})-(b^n*L_{r, r+\theta h})(u_r)}^q_{L^2_x}dr\right)^{1/q}\\
    &\leq \norm{(b^n*L_{r, r+\theta h})}_{Lip_x} h^{1/q}(2\norm{u}_{L^\infty_t L^2_x})^{1-2/q}\seminorm{u}^{2/q}_{B^{1/2}_{2, \infty}L^2_x}\\
    &\leq h^{\gamma+1/q}\norm{b^n*L}_{C^\gamma_tLip_x}(2\norm{u}_{L^\infty_t L^2_x})^{1-2/q}\seminorm{u}^{2/q}_{B^{1/2}_{2, \infty}L^2_x}
\end{align*}
Bounding the remaining terms in a similar fashion, we obtain 
\[
\bar{\Omega}_q(\delta G, h)\lesssim h^{\gamma+1/q}\norm{b^n*L}_{C^\gamma_tLip_x}(2\norm{u}_{L^\infty_t L^2_x})^{1-2/q}\seminorm{u}^{2/q}_{B^{1/2}_{2, \infty}L^2_x}
\]
meaning that for $\gamma+1/q>1$, by the Sewing Lemma \ref{sewing} $G$ admits a sewing. By the a priori bound that comes with the Sewing Lemma \ref{sewing}, this concludes the claim. 

It can then be shown that this sewing coincides with the Bochner integral $I^n(u)$ in the statement by means of Lemma \ref{local approx doesnt matter}. 
\end{proof}

\begin{proof}[Proof of Theorem \ref{thm:reg-by-noise}]
We only need to check Assumption~\ref{ass:integral-operator}, since Assumptions~\ref{ass:compact-Gelfand} and~\ref{ass:monotone-operator} have already been verified. The assertion then follows by Theorem~\ref{thm:main}.

\underline{The approximate operator:}
We define the Nemytskii operator $ \sigma^n: [0,T] \times L^2(\Lambda) \to L^2(\Lambda)$ by
\begin{align*}
\sigma^n(t,u) := b^n*L^w_t(u).
\end{align*}
Clearly, $\sigma^n$ is $\mathcal{B}(0,T) \otimes \mathcal{B}(L^2(\Lambda))$-measurable.

Next, we show that $\sigma^n$ satisfies the linear growth and Lipschitz condition. By the occupation times formula~\eqref{occupation times formula} it holds
\begin{align*}
    b^n*L^w_t(u) = \int_{\mathbb{R}} b^n(u -z) L^w_t(z) \dd z = \int_0^t b^n(u -w_s) \dd s.
\end{align*}
Now, H\"older's inequality implies
\begin{align*}
  \norm{\sigma^n(t,u)}_{L^2(\Lambda)}^2 &= \int_{\Lambda} \abs{\int_0^t b^n(u(x) -w_s) \dd s }^2 \dd x \\
  &\leq t \int_{\Lambda} \int_{0}^t \abs{b^n(u(x) -w_s)}^2 \dd s \dd x \\
  &\leq 2t^2 C_n^2 \left( (1 + \norm{w}_{C([0,t];\mathbb{R})})^2 \abs{\Lambda}+ \norm{u}_{L^2(\Lambda)}^2   \right).
\end{align*}
Similarly,
\begin{align*}
     \norm{\sigma^n(t,u) - \sigma^n(t,v)}_{L^2(\Lambda)}^2 &\leq t \int_{\Lambda} \int_0^t \abs{b^n(u(x) -w_s) - b^n(v(x) -w_s) }^2 \dd s \dd x \\
     &\leq t^2 C_n^2 \norm{u -v}_{L^2(\Lambda)}^2.
\end{align*}
Define $I^n_t(u) = \int_0^t \sigma^n(s, u_s) \dd s$.

\underline{Ad~\ref{it:H5}:} 
An application of Lemma~\ref{verify nonlinear H5} with $q=4$ and $\gamma > 3/4$ implies
\begin{align*}
    I^n_t(u) = \int_0^t \sigma^n(s, u_s) \dd s = \int_0^t b^n(u_s - w_s) \dd s.
\end{align*}
Following the derivation of~\eqref{eq:Integral-reg-by-noise} (replacing $(0,T)$ by $(s,t)$ for arbitrary $s<t \in [0,T]$) it holds
  \begin{align*} 
       &\seminorm{I^n(u)}_{B^\gamma_{4, \infty}(s,t;L^2(\Lambda))}\\
       &\hspace{2em} \lesssim \norm{b^n*L^w}_{C^\gamma([s,t];C^{0,1}(\mathbb{R}))}(1+\norm{u}_{L^\infty(s,t;L^2(\Lambda))}+\seminorm{u}_{B^{1/2}_{2, \infty}(s,t;L^2(\Lambda))}).
\end{align*}
Thus,~\ref{it:H5} holds with $\lambda(\abs{t-s}):= \sup_{n \in \mathbb{N}}  \norm{b^n*L^w}_{C^\gamma([s,t];C^{0,1}(\mathbb{R}))}$ and some constant $c_4$ independent of $n$.

\underline{Ad~\ref{it:H6}:} Let us take $(u^n)_{n \in \mathbb{N}}, u \in  L^\infty(0,T; L^2(\Lambda))\cap B^{1/2}_{2, \infty}(0,T;L^2(\Lambda))$ such that
\begin{align*}
    \sup_{n \in \mathbb{N}} \norm{u^n}_{ L^\infty(0,T; L^2(\Lambda))} + \norm{u^n}_{B^{1/2}_{2, \infty}(0,T;L^2(\Lambda))} < \infty,
\end{align*}
and $u^n \to u \in L^2(0,T;L^2(\Lambda))$. We need to show $I^n(u^n)\to I(u)$ in $B^\gamma_{q, \infty} (0,T;L^2(\Lambda))$. 

Note that by construction, we have 
\begin{align*}
&\seminorm{I^n(u^n)-I(u^n)}_{B^\gamma_{q, \infty}(0,T;L^2(\Lambda))}\lesssim \norm{(b^n*L)-(b*L)}_{C^\gamma(0;T;C^{0,1}(\Lambda))}\\
&\hspace{13em}(1+\norm{u^n}_{ L^\infty(0,T; L^2(\Lambda))} + \norm{u^n}_{B^{1/2}_{2, \infty}(0,T;L^2(\Lambda))})\to 0.
\end{align*}
 
An application of Lemma~\ref{nonlinear verify H6} with $q=4$ and $\gamma > 3/4$ shows $I(u^n)\to I(u)$ in $B^\gamma_{q, \infty}(0,T;L^2(\Lambda))$.

At this point, we have verified that Assumption~\ref{ass:integral-operator} holds for $q=4$ and $\gamma > 3/4$. Therefore the claim follows by Theorem~\ref{thm:main}.
\end{proof}

\begin{lemma}
\label{nonlinear verify H6}
Let $q \in [2,\infty)$, $\gamma > 1-1/q$, $b \in \mathscr{S}'(\mathbb{R})$ be a Schwartz distribution, and $w:[0,T] \to \mathbb{R}$ such that its local time $L^w$ satisfies $b * L^w \in C^\gamma(0,T;C^{0,1}(\mathbb{R}))$. 

Moreover, let $(u^n)_{n \in \mathbb{N}}, u \in  L^\infty(0,T; L^2(\Lambda))\cap B^{1/2}_{2, \infty}(0,T;L^2(\Lambda))$ such that
\begin{align*}
    \sup_{n \in \mathbb{N}} \norm{u^n}_{ L^\infty(0,T; L^2(\Lambda))} + \norm{u^n}_{B^{1/2}_{2, \infty}(0,T;L^2(\Lambda))} < \infty,
\end{align*}
and $u^n \to u \in L^2(0,T;L^2(\Lambda))$.

Then $I(u^n) \rightarrow I(u) \in B^{\gamma}_{q,\infty}(0,T;L^2(\Lambda))$.
\end{lemma}
\begin{proof}
The assertion follows from Lemma~\ref{sewing convergence} provided we show that for $G^n_{s,t}=(b*L^w_{s,t})(\langle u^n\rangle_{s,t})$ and $G_{s,t}=(b*L^w_{s,t})(\langle u \rangle_{s,t})$ it holds $\norm{G^n-G}_{\mathbb{B}^\gamma_{q, \infty}(0,T;(L^2(\Lambda))}\to 0$. 

Using Lemma~\ref{abstract limit identification lemma} we have 
\begin{align*}
    &\left( \int_0^{T-h}\norm{G^n_{r, r+h}-G_{r, r+h}}_{L^2(\Lambda)}^q \dd r\right)^{1/q}\\
    &\leq \norm{b*L^w_{r, r+h}}_{C^{0,1}(\mathbb{R})}(\norm{u^n}_{L^\infty(0,T;L^2(\Lambda))}+\norm{u}_{L^\infty(0,T;L^2(\Lambda))})^{1-2/q}\norm{u^n-u}^{2/q}_{L^2(0,T;L^2(\Lambda))}\\
    &\lesssim h^\gamma \norm{b*L^w}_{C^\gamma(0,T;C^{0,1}(\mathbb{R})}\norm{v^n-v}^{2/q}_{L^2_tL^2_x},
\end{align*}
from which we conclude that indeed $\norm{G^n-G}_{\mathbb{B}^\gamma_{q, \infty}(0,T;L^2(\Lambda))}\to 0$.
\end{proof}

\subsection{Porous medium equation}
Let $\Lambda\subset \mathbb{R}^d$ be a bounded Lipschitz domain and $p\in (\frac{2d}{d+2}, \infty)$. 

\subsubsection{The Gelfand triple} We set $V=L^p(\Lambda)$ and $H=W^{-1, 2}(\Lambda)$.

By duality and the Rellich-Kondrachov theorem $p>\frac{2d}{d+2}$ ensures again that $L^p(\Lambda)\hookrightarrow W^{-1, 2}(\Lambda)$ is compact. Thus, Assumption~\ref{ass:compact-Gelfand} is satisfied.

\subsubsection{The monotone operator}
Let $\Psi: \mathbb{R}\to \mathbb{R}$ be a function satisfying:
\begin{enumerate}
   \item[($\Psi$1)] \label{it:psi1}$\Psi$ is continuous;
    \item[($\Psi$2)] $\Psi$ is monotone in the sense that $(\Psi(t)-\Psi(s))\cdot (t-s)\geq 0$ for all $t, s\in \mathbb{R}$;
    \item[($\Psi$3)] $\Psi$ is coercive in the sense that there exist $a\in (0, \infty)$ and $c\in [0, \infty)$ such that for all $s\in \mathbb{R}$
    \[
    s\cdot \Psi(s)\geq a|s|^{p}-c;
    \]
    \item[($\Psi$4)]\label{it:psi4}
    There exist $c_3, c_4$ such that for all $s\in \mathbb{R}$
    \[
    |\Psi(s)|\leq c_4+c_3|s|^{p-1}.
    \]
\end{enumerate}
It can then be shown that for any $\Psi$ satisfying \ref{it:psi1}-\ref{it:psi4}, the operator
\begin{equation}
    \begin{split}
        A: V&\to V^*\\
        u&\to \Delta \Psi(u)
        \label{porous medium operator}
    \end{split}
\end{equation}
is well-defined and satisfies the Conditions~\ref{it:H1}-\ref{it:H4}, refer to \cite[Lemma 4.1.13, p.85-88]{Liu2015}. Thus, Assumption~\ref{ass:monotone-operator} holds.

\begin{theorem}\label{thm:porous-medium}
Let $T > 0$ and $u_0 \in W^{-1,2}(\Lambda)$. Let $I$ be given by one of the following:
\begin{enumerate}
    \item (Additive driver) \label{it:additive-porous} Let $\gamma > 1/2$ and $Z \in C^\gamma(0,T;W^{-1,2}(\Lambda))$. 
    
    Define $I_t(u) := Z_t - Z_0$;
    \item (Young integral) Let $\gamma > 3/4$, $\varepsilon > 0$, $X\in C^\gamma(0,T; C^{1+\varepsilon}(\Lambda))$ and $\sigma: W^{-1, 2}(\Lambda)\to W^{-1, 2}(\Lambda)$ satisfy
    \[
    \norm{\sigma(u)}_{W^{-1, 2}(\Lambda)}\leq C(1+\norm{u}_{W^{-1, 2}(\Lambda)}), \quad \norm{\sigma(u)-\sigma(v)}_{W^{-1, 2}(\Lambda)}\leq L\norm{u-v}_{W^{-1, 2}(\Lambda)}.
    \]
    
    Define $I_t(u) := \int_0^t \sigma(u_s) \dd X_s$ as the Young integral given by Lemma~\ref{abstract young lemma};
\end{enumerate}
    Then there exists a (in Case~\ref{it:additive-porous} unique) weak solution
    \begin{align*}
        u \in C([0,T]; W^{-1, 2}(\Lambda)) \cap B^{1/2}_{2,\infty}(0,T; W^{-1, 2}(\Lambda)) \cap L^p(0,T; L^p(\Lambda))
    \end{align*}
    to
    \begin{align*}
        \dd u - \Delta \Psi(u) \dd t = \dd I_t(u), \qquad u(0) = u_0.
    \end{align*}
\end{theorem}

\begin{proof}
    The first part of the statement follows again directly from Theorem~\ref{thm:AnwerAdditive}. Using Theorem \ref{multiplication theorem}, we note that
    \[
    \norm{u\cdot v}_{W^{-1, 2}(\Lambda)}\leq C\norm{u}_{W^{-1, 2}(\Lambda)}\norm{v}_{C^{1+\epsilon}(\Lambda)}
    \]
   for $\epsilon>0$. Thus, $E=C^{1+\epsilon}(\Lambda)$ is a multiplier in $H=W^{-1, 2}(\Lambda)$ in the sense of Definition~\ref{multiplicative ideal}. Now, the second claim follows from Theorem~\ref{thm:Existence-abstract-Young}. 
\end{proof}

\subsection{Shear-thickening fluids}
Power-law fluids are non-Newtonian fluids, whose rheology -- the interaction of shear-stress and strain-rate -- is described by a power-law behaviour. Here, we assume that this rheology is given by
\begin{align*}
    S(B) := \mu(1 + \abs{B})^{p-2} B \in \mathbb{R}^{d \times d },
\end{align*}
where $p\in (1, \infty)$ and $\mu > 0$. More details can be found in e.g.~\cite{Blechta2020}.

The evolution equation for velocity field~$u$ and pressure~$\pi$ of an unforced incompressible power-law fluid reads
\begin{align} \label{eq:Power-law-fluids}
\begin{cases}
    \partial_t u - \Div \,S(\varepsilon u) + (u\cdot \nabla)u + \nabla \pi &= 0,\\
   \hfill \Div \, u &= 0.
\end{cases}
\end{align}
Here $\varepsilon u \in \mathbb{R}^{d\times d}$ denotes the symmetric gradient. 

We show that if the fluid is shear-thickening for a sufficiently large power-law index, our general framework provides the existence of weak solutions for the forced evolution equation~\eqref{eq:Power-law-fluids}. We only discuss the existence of velocity. The pressure can be reconstructed once the existence of velocity is established. 

From now on, let $d \in \mathbb{N}$, $p\in (d,\infty) \cap [3, \infty)$ and $\Lambda\subset \mathbb{R}^d$ be a bounded Lipschitz domain. 

\subsubsection{The Gelfand triple}
We set 
\begin{align*}
    V &=W^{1, p}_{0,\Div}(\Lambda) = \overline{\{ u\in  C^\infty_c(\Lambda):\, \Div \, u =0 \}}^{\norm{\cdot}_{W^{1,p}(\Lambda)}}, \\
    H &= L_{\Div}^2(\Lambda) =  \overline{\{ u\in  C^\infty_c(\Lambda):\, \Div \, u =0 \}}^{\norm{\cdot}_{L^{2}(\Lambda)}}.
\end{align*}

\subsubsection{The locally monotone operator}
We define the operator $A: V\to V^*$ by
\begin{align} \label{eq:power-law-operator}
  \forall v \in V: \quad   \langle A (u), v\rangle_{V^*,V}:=-\int_\Lambda S(\varepsilon u) : \varepsilon v + (u \cdot \nabla)u \cdot v \dd x.
\end{align}

The following result establishes the existence and conditional uniqueness of the velocity field for shear-thickening fluids.

\begin{theorem} \label{thm:power-law-fluids}
Let $T > 0$, $u_0 \in L^2_{\Div}(\Lambda)$, and $I$ satisfy Assumption~\ref{ass:integral-operator}.

Then there exists a weak solution 
    \begin{align*}
        u \in C([0,T];L^2_{\Div}(\Lambda)) \cap B^{1/2}(0,T;L^2_{\Div}(\Lambda)) \cap L^{p}(0,T; W^{1,p}_{0,\Div}(\Lambda))
    \end{align*}
     such that for all $t \in [0,T]$ and $v \in W^{1,p}_{0,\Div}(\Lambda)$ it holds
\begin{align*} 
     \int_{\Lambda} (u_t - u_0) \cdot v \dd x + \int_0^t\int_\Lambda S(\varepsilon u) : \varepsilon v + (u \cdot \nabla)u \cdot v \dd x \dd s = \int_{\Lambda} I_t(u)\cdot v \dd x.
\end{align*}

In particular, if $I_t(u) =Z_t - Z_0$ for $Z \in C^\gamma(0,T;L^2_{\Div}(\Lambda))$ with $\gamma > 1/2$, then this weak solution is unique.
\end{theorem}
\begin{rem}
Notice that $Z = 0$ is a valid choice in Theorem~\ref{thm:power-law-fluids}, recovering existence and uniqueness for the velocity field of unforced shear-thickening fluids.
\end{rem}

\begin{proof}[Proof of Theorem~\ref{thm:power-law-fluids}]
Existence follows from Theorem~\ref{thm:main}, while uniqueness is derived by Theorem~\ref{thm:AnwerAdditive}. It remains to verify their assumptions. 

\underline{Ad Assumption~\ref{ass:compact-Gelfand}:}~By the Rellich-Kondrachov theorem, it follows that the embedding $W^{1, p}_{0,\Div}(\Lambda)\hookrightarrow L^2_{\Div}(\Lambda)$ is compact.

Closely related to the tensor~$S$ is the tensor
\begin{align*}
     F(B) := \sqrt{\mu} (1 + \abs{B})^{(p-2)/2} B.
\end{align*}
The tensor $F$ quantifies the monotonicity of the tensor $S$ and is a convenient tool for the verification of Assumption~\ref{ass:monotone-operator}. More details on the relation between $F$ and $S$ can be found in e.g.~\cite{Diening2020,MR4286257,2023arXiv230713253L}.

Keep in mind that by Korn's and Poincaré's inequalities the following norms on $W^{1,p}_{0,\Div}(\Lambda)$ are equivalent: 
\begin{align*}
    \norm{\varepsilon u}_{L^p(\Lambda)} \eqsim \norm{\nabla u}_{L^p(\Lambda)} \eqsim \norm{ u}_{W^{1,p}(\Lambda)}.
\end{align*}

Next, we verify Assumption~\ref{ass:monotone-operator}.

\underline{Ad~\ref{it:H1}:} Let $u,v,w \in W^{1,p}_{0,\Div}(\Lambda)$ and $\lambda \in \mathbb{R}$. We split
\begin{align*}
    &\langle A (u + \lambda v) ,w\rangle_{V^*,V}  \\
    &\hspace{2em} = -\int_\Lambda S(\varepsilon u + \lambda \varepsilon v) : \varepsilon w \dd x - \int_{\Lambda} \big( (u + \lambda v) \cdot \nabla\big)(u + \lambda v) \cdot w \dd x =: \mathrm{K}_1 + \mathrm{K}_2.
\end{align*}
We show that~$\mathrm{K}_1$ and~$\mathrm{K}_2$ are continuous functions in $\lambda$. Let $\lambda_1, \lambda_2 \in \mathbb{R}$.

Continuity of $\mathrm{K}_1$: Due to continuity of $S$ it holds 
\begin{align*}
    \big( S(\varepsilon u + \lambda_1 \varepsilon v) -  S(\varepsilon u + \lambda_2 \varepsilon v) \big) : \varepsilon w \rightarrow 0 \quad \text{ for } \quad \lambda_1 \to \lambda_2.
\end{align*}
Moreover, using H\"older's inequality and $\abs{S(B)}^{p'} \lesssim \abs{B}^p + 1$,
\begin{align*}
    &\int_\Lambda \big( S(\varepsilon u + \lambda_1 \varepsilon v) - S(\varepsilon u + \lambda_2 \varepsilon v) \big): \varepsilon w \dd x \\
    &\hspace{2em} \leq \big(\norm{S(\varepsilon u + \lambda_1 \varepsilon v)}_{L^{p'}(\Lambda)} +\norm{S(\varepsilon u + \lambda_2 \varepsilon v)}_{L^{p'}(\Lambda)} \big)  \norm{\varepsilon w}_{L^p(\Lambda)} \\
    &\hspace{2em} \lesssim \big( \norm{\varepsilon u + \lambda_1 \varepsilon v}_{L^p(\Lambda)}^p +\norm{\varepsilon u + \lambda_2 \varepsilon v}_{L^p(\Lambda)}^p + \abs{\Lambda} \big)^{1/p'}  \norm{\varepsilon w}_{L^p(\Lambda)} < \infty.
\end{align*}
Therefore, dominated convergence ensures that
\begin{align*}
    \lambda \mapsto \int_\Lambda S(\varepsilon u + \lambda \varepsilon v) : \varepsilon w \dd x
\end{align*}
is continuous.

Continuity of $\mathrm{K}_2$: H\"older's inequality implies
\begin{align*}
    &\int_{\Lambda} \big[ \big( (u + \lambda_1 v) \cdot \nabla \big)(u + \lambda_1 v) - \big( (u + \lambda_2 v) \cdot \nabla \big)(u + \lambda_2 v) \big] \cdot w \dd x \\
    &\hspace{2em} = (\lambda_1 - \lambda_2) \int_{\Lambda} \big( v \cdot \nabla)(u + \lambda_1 v) + \big( (u + \lambda_2 v) \cdot \nabla) v \cdot w \dd x \\
    &\hspace{2em} \leq  \abs{\lambda_1 - \lambda_2} \big( \norm{\nabla u + \lambda_1 \nabla v}_{L^{p}(\Lambda)} \norm{v \otimes w }_{L^{p'}(\Lambda)} \\
    &\hspace{9em} + \norm{\nabla v}_{L^{p}(\Lambda)} \norm{(u + \lambda_2 v) \otimes w}_{L^{p'}(\Lambda)} \big).
\end{align*}
Notice that, since $p \geq 3d/(d+2)$ it holds $W^{1,p}_{0}(\Lambda) \hookrightarrow L^{2p'}(\Lambda)$,
\begin{align*}
    \norm{v \otimes w }_{L^{p'}(\Lambda)} \leq \norm{v}_{L^{2p'}(\Lambda)} \norm{w}_{L^{2p'} (\Lambda)} \lesssim \norm{\nabla v}_{L^p(\Lambda)}  \norm{\nabla w}_{L^p(\Lambda)}.
\end{align*}
Therefore,
\begin{align*}
    &\norm{\nabla u + \lambda_1 \nabla v}_{L^{p}(\Lambda)} \norm{v \otimes w }_{L^{p'}(\Lambda)}  + \norm{\nabla v}_{L^{p}(\Lambda)} \norm{(u + \lambda_2 v) \otimes w}_{L^{p'}(\Lambda)} \\
    &\hspace{2em}  \lesssim \norm{\nabla u + \lambda_1 \nabla v}_{L^{p}(\Lambda)} \norm{\nabla v}_{L^p(\Lambda)}  \norm{\nabla w}_{L^p(\Lambda)} \\
    &\hspace{4em} + \norm{\nabla v}_{L^{p}(\Lambda)} \norm{\nabla u + \lambda_2 \nabla v}_{L^p(\Lambda)}  \norm{\nabla w}_{L^p(\Lambda)}  < \infty.
\end{align*}
This implies the continuity of 
\begin{align*}
    \lambda \to \int_{\Lambda} \big( (u + \lambda v) \cdot \nabla\big)(u + \lambda v) \cdot w \dd x.
\end{align*}
Thus, Condition~\ref{it:H1} holds.

\underline{Ad~\ref{it:H2}:} Let $u,v \in W^{1,p}_{0,\Div}(\Lambda)$. We split
\begin{align*}
    &2\langle A (u) - A(v),u- v\rangle_{V^*,V} = -2\int_\Lambda \big(S(\varepsilon u) - S(\varepsilon v)\big) : \big(\varepsilon u- \varepsilon v \big) \dd x \\
    &\hspace{4em} -2 \int_{\Lambda} \big( (u \cdot \nabla)u - (v \cdot \nabla)v \big) \cdot (u-v) \dd x =: \mathrm{J}_1 + \mathrm{J}_2,
\end{align*}
into strongly monotone and lower-order parts. 

Using~\cite[Lemma~40]{Diening2020}, we find
\begin{align} \label{eq:strong-dissipation}
\begin{aligned}
    2\int_\Lambda \big(S(\varepsilon u) - S(\varepsilon v)\big) : \big(\varepsilon u- \varepsilon v \big) \dd x  &\eqsim \norm{F( \varepsilon u) - F(\varepsilon v)}_{L^2(\Lambda)}^2 \\ 
    &\geq c \mu \norm{\varepsilon (u - v)}_{L^2(\Lambda)}^2.
    \end{aligned}
\end{align}

Controlling $\mathrm{J}_2$: We write
\begin{align*}
& \int_{\Lambda} \big( (u \cdot \nabla)u - (v \cdot \nabla)v \big) \cdot (u-v) \dd x \\
&\hspace{2em} =  \int_{\Lambda} \big( (u \cdot \nabla)(u-v) + ((u-v) \cdot \nabla)v \big) \cdot (u-v) \dd x.
\end{align*}
Notice that, using integration by parts and the divergence-free condition,
\begin{align} \label{eq:Vanish}
    \int_{\Lambda} (u \cdot \nabla)(u-v)  \cdot (u-v) \dd x = - \frac{1}{2} \int_{\Lambda} \Div\, u \abs{u-v}^2 \dd x = 0.
\end{align}
Similarly, together with H\"older's inquality and $W^{1,p}_0(\Lambda) \hookrightarrow L^\infty(\Lambda)$, 
\begin{align*}
     \int_{\Lambda} ((u-v) \cdot \nabla)v  \cdot (u-v) \dd x &= -   \int_{\Lambda} ((u-v) \cdot \nabla)(u-v)  \cdot v \dd x \\
     &\leq \norm{\nabla (u-v)}_{L^2(\Lambda)} \norm{ u-v}_{L^{2}(\Lambda)} \norm{v}_{L^\infty(\Lambda)} \\
     &\lesssim \norm{\nabla (u-v)}_{L^2(\Lambda)} \norm{ u-v}_{L^{2}(\Lambda)} \norm{\nabla v}_{L^p(\Lambda)}.
\end{align*}

Korn's inequality and weighted Young's inequality imply
\begin{align} \label{eq:bound-lower-order}
\begin{aligned}
     \mathrm{J}_2 &\leq \delta  \norm{\varepsilon (u-v)}_{L^2(\Lambda)}^2 + c_\delta \norm{ u-v}_{L^2(\Lambda)}^2  \norm{\nabla v}_{L^{p}(\Lambda)}^{2}.
     \end{aligned}
\end{align}
Combining~\eqref{eq:strong-dissipation} and~\eqref{eq:bound-lower-order}, and choosing $\delta >0$ sufficiently small, yield
\begin{align*}
    2\langle A (u) - A(v),u- v\rangle_{V^*,V} &\leq (\delta-c \mu)  \norm{\varepsilon (u-v)}_{L^2(\Lambda)}^2  +c_\delta \norm{ u-v}_{L^2(\Lambda)}^2  \norm{\nabla v}_{L^{p}(\Lambda)}^{2} \\
    &\leq c_\delta \norm{ u-v}_{L^2(\Lambda)}^2  \norm{\nabla v}_{L^{p}(\Lambda)}^{2}.
\end{align*}
Thus, Condition~\ref{it:H2} holds with $h = 0$ and $\eta(v) = c_\delta  \norm{\nabla v}_{L^{p}(\Lambda)}^{2}$.

\underline{Ad~\ref{it:H3}:} Using $\abs{F(B)}^2 = \mu(1+ \abs{B})^{p-2} \abs{B}^2 \geq \mu \abs{B}^p $, and an argument similar to~\eqref{eq:Vanish},
\begin{align*}
     \langle A (u), u\rangle_{V^*,V} &=-\int_\Lambda S(\varepsilon u) : \varepsilon u + (u\cdot \nabla)u \cdot u \dd x \\
     &= - \norm{F(\varepsilon u)}_{L^2(\Lambda)}^2 \leq - \mu \norm{\varepsilon u}_{L^p(\Lambda)}^p.
\end{align*}
Thus, Condition~\ref{it:H3} holds with $\alpha = p$, $c_1 = \mu$, $c_2 = 0$ and $f(t) = 0$.

\underline{Ad~\ref{it:H4}:} Using $W^{1,p}_0(\Lambda) \hookrightarrow L^\infty(\Lambda)$, $\abs{S(B)}^{p'} \leq \mu^{p'}(1+\abs{B})^p$, H\"older's inequality and Korn's inequality,
\begin{align*}
    \norm{A(u)}_{V^*} &\leq \norm{S(\varepsilon u)}_{L^{p'}(\Lambda)} + C\norm{(u\cdot \nabla)u}_{L^1(\Lambda)} \\
    &\leq \mu \norm{1+\varepsilon u}_{L^p(\Lambda)}^{p-1} + C \abs{\Lambda}^{1/p'}\norm{u}_{L^{\infty}(\Lambda)} \norm{\nabla u}_{L^{p}(\Lambda)} \\
    &\leq \mu(\abs{\Lambda}+ \norm{\varepsilon u}_{L^p(\Lambda)})^{p-1} + C^2 C_{\mathrm{Korn}}\abs{\Lambda}^{1/p'} \norm{\varepsilon u}_{L^{p}(\Lambda)}^2
\end{align*}
Since $p \geq 3$, 
\begin{align*}
    \norm{\varepsilon u}_{L^{p}(\Lambda)}^2 \leq (1+\norm{\varepsilon u}_{L^{p}(\Lambda)})^2 \leq  (1+\norm{\varepsilon u}_{L^{p}(\Lambda)})^{p-1} \leq 2^{p-2}(1 +\norm{\varepsilon u}_{L^{p}(\Lambda)}^{p-1}).
\end{align*}
Overall,
\begin{align*}
    \norm{A(u)}_{V^*} &\leq 2^{p-2} (\mu \abs{\Lambda}^{p-1} +  C^2 C_{\mathrm{Korn}}\abs{\Lambda}^{1/p'}) +  \mu  2^{p-2}(\mu + C^2 C_{\mathrm{Korn}}\abs{\Lambda}^{1/p'}) \norm{\varepsilon u}_{L^p(\Lambda)}^{p-1}.
\end{align*}
Thus, Condition~\ref{it:H4} holds with $g(t) = 2^{p-2} (\mu \abs{\Lambda}^{p-1} +  C^2 C_{\mathrm{Korn}}\abs{\Lambda}^{1/p'}) $ and $c_3 = \mu  2^{p-2}(\mu + C^2 C_{\mathrm{Korn}}\abs{\Lambda}^{1/p'})$.

By assumption, Assumption~\ref{ass:integral-operator} is satisfied.

We have verified all assumptions of Theorems~\ref{thm:main} and~\ref{thm:AnwerAdditive}, and the proof is complete. 
\end{proof}

\newpage

\section{Appendix} \label{appendix}
\subsection{Local time and occupation times formula}
We recall for the reader the basic concepts of occupation measures, local times and the occupation times formula. A comprehensive review paper on these topics is \cite{horowitz}. 
\begin{definition}
Let $w:[0,T]\to \R^N$ be a measurable path. Then the occupation measure at time $t\in [0,T]$, written $\mu^w_t$ is the Borel measure on $\R^d$ defined by 
\[
\mu^w_t(A):=\lambda(\{ s\in [0,t]:\ w_s\in A\}), \quad A\in \mathcal{B}(\R^N),
\]
where $\lambda$ denotes the standard Lebesgue measure. 
\end{definition}
The occupation measure thus measures how much time the process $w$ spends in certain Borel sets. Provided for any $t\in [0,T]$, the measure is absolutely continuous with respect to the Lebesgue measure on $\R^N$, we call the corresponding Radon-Nikodym derivative local time of the process $w$:
\begin{definition}
Let $w:[0,T]\to \R^N$ be a measurable path. Assume that there exists a measurable function $L^w:[0,T]\times \R^N\to \R_+$ such that 
\[
\mu^w_t(A)=\int_A L^w_t(z) \dd z, 
\]
for any $A\in \mathcal{B}(\R^N)$ and  $t\in [0,T]$. Then we call $L^w$ local time of $w$. 
\end{definition}
Note that by the definition of the occupation  measure, we have for any bounded measurable function $f:\R^N\to \R$ that 
\begin{equation}
    \int_0^tf(w_s)ds=\int_{\R^N} f(z)\mu^w_t(\dd z).
    \label{occupation times formula}
\end{equation}
The above equation \eqref{occupation times formula} is called occupation times formula. Remark that in particular, provided $w$ admits a local time, we also have for any $u\in \R^N$
\begin{equation}
    \int_0^tf(u-w_s)\dd s=\int_{\R^N} f(u-z)\mu^w_t(\dd z)=\int_{\R^N}f(u-z)L^w_t(z) \dd z=(f*L^w_t)(u).
\end{equation}
\subsection{Some helpful Lemmata}

\begin{lemma}
\label{abstract average minus leftpoint}
    Suppose $\sigma:H\to H$ is Lipschitz-continuous, i.e., there exists a constant $L>0$ such that for all $x,y \in H$ it holds
\begin{align} \label{eq:Lipschitz-ass}
    \norm{\sigma(x)-\sigma(y)}_H\leq L\norm{x-y}_H.
\end{align}
Moreover, let $u \in B^{1/2}_{2,\infty} H\cap L^\infty_tH$. Then we have for $q\geq 2$
    \[
   \left( \int_0^{T-h} \norm{\langle \sigma(u)\rangle_{r, r+h}-\sigma(u_r)}_H^qdr\right)^{1/q}\leq h^{1/q} (2\norm{u}_{L^\infty_t H})^{1-2/q} L \seminorm{u}_{B^{1/2}_{2, \infty}H}^{2/q}.
    \]
    and 
    \[
       \left( \int_0^{T-h} \norm{\sigma(\langle u\rangle_{r, r+h})-\sigma(u_r)}_H^qdr\right)^{1/q}\leq h^{1/q} (2\norm{u}_{L^\infty_t H})^{1-2/q} L \seminorm{u}_{B^{1/2}_{2, \infty}H}^{2/q}.
    \]
\end{lemma}
\begin{proof}
Using Jensen's inequality,~\eqref{eq:Lipschitz-ass}, and H\"older's inequalities, 
\begin{align*}
    &\int_0^{T-h}\norm{ \langle \sigma(u) \rangle_{r,r+h}-\sigma(u_r)}_{H}^q \dd r \\
    &\hspace{2em} \leq\frac{1}{h^q}\int_0^{T-h}\left(\int_{r}^{r+h}\norm{\sigma(u_z)-\sigma(u_r)}_H \dd z\right)^q \dd r \\
    &\hspace{2em} \leq \frac{L^q}{h^q} (2\norm{u}_{L^\infty_t H})^{q-2} \int_0^{T-h}\left(\int_{r}^{r+h}\norm{u_z-u_r}^{2/q}_H \dd z \right)^q \dd r \\
    &\hspace{2em} \leq \frac{L^q}{h} (2\norm{u}_{L^\infty_t H})^{q-2} \int_0^{T-h} \int_{r}^{r+h}\norm{u_z-u_r}^{2}_H \dd z \dd r.
\end{align*}
A substitution $z = r+s$ and Fubini's theorem show
\begin{align*}
    \int_0^{T-h} \int_{r}^{r+h}\norm{u_z-u_r}^{2}_H \dd z \dd r &= \int_0^h \int_0^{T-h}\norm{u_{r+s}-u_r}_H^2 \dd r \dd s \\
    &\leq \int_0^h s \dd s \seminorm{u}_{B^{1/2}_{2,\infty}H}^2 \leq h^2 \seminorm{u}_{B^{1/2}_{2,\infty}H}^2. 
\end{align*}
The first assertion follows by combining the estimates. The second assertion is proven analogously.
\end{proof}

\begin{lemma}
    \label{abstract limit identification lemma}
    Let $\sigma:H\to H$ be Lipschitz-continuous, i.e. 
    \[
    \norm{\sigma(x)-\sigma(y)}_H\leq L\norm{u-v}.
    \]
    Let $u,v \in L^\infty_t H$. Then we have for $q\geq 2$ 
    \begin{align}
    \begin{aligned}
          &\left(\int_0^{T-h}\norm{\langle \sigma(u)\rangle_{r, r+h}-\langle \sigma(v)\rangle_{r, r+h}}_H^q \dd r \right)^{1/q} \\
          &\hspace{2em} \leq (\norm{u}_{L^\infty_t H} + \norm{v}_{L^\infty_t H})^{1-2/q} L \norm{u-v}_{L^2_t H}^{2/q}.
    \end{aligned} 
    \end{align}
    as well as
     \begin{align}
    \begin{aligned}
          &\left(\int_0^{T-h}\norm{ \sigma(\langle u\rangle_{r, r+h})- \sigma(\langle v\rangle_{r, r+h})}_H^q \dd r \right)^{1/q} \\
          &\hspace{2em} \leq (\norm{u}_{L^\infty_t H} + \norm{v}_{L^\infty_t H})^{1-2/q} L \norm{u-v}_{L^2_t H}^{2/q}.
    \end{aligned} 
    \end{align}
\end{lemma}
\begin{proof}
The proof proceeds analogously to the one of Lemma~\ref{abstract average minus leftpoint}.
\end{proof}

\begin{theorem}\cite[Corollary 2.1.35]{Martin2018Refinements}
\label{multiplication theorem}
    Let $\alpha, \beta\in \mathbb{R}\backslash \{0\}$ with $\alpha<\beta$ and $\alpha+\beta>0$. Then there exists a $C>0$ such that for all $p_1, p_2, p, q_1, q_2, r\in [1, \infty]$ satisfying
    \[
    \frac{1}{p}=\frac{1}{p_1}+\frac{1}{p_2},\qquad r\leq q_1
    \]
    we have
    \[
    \norm{u\cdot v}_{B^\alpha_{p, r}}\leq C \norm{u}_{B^\alpha_{p_1, q_1}}\norm{v}_{B^\beta_{p_2, q_2}}
    \]
\end{theorem}

\begin{lemma} \label{lem:Shift-bound}
Let $q \in (1,\infty]$, $\gamma \in (1/q,1]$ and $E$ be a Banach space. Then for all finite intervals $J$ and $u \in B^{\gamma}_{q,\infty}(J;E)$ it holds 
\begin{align} \label{eq:Shift-bound}
  \sup_{s\in J}  \left(\int_J \norm{u_t - u_s}_E^q \dd t \right)^{1/q} \leq  C_{\gamma,q} \abs{J}^{\gamma} \seminorm{u}_{B^{\gamma}_{q,\infty}(I;E)},
\end{align}
where $C_{\gamma,q} = \frac{3^{2-(\gamma - 1/q)}}{\gamma - 1/q}$.
\end{lemma}
\begin{proof}
Let $\nu \in J$. Using~\cite[Theorem~10]{MR1108473} with $s = \gamma$, $r = \gamma - 1/q$, $p = q$ and $q = \infty$ shows
\begin{align*}
    \left(\int_J \norm{u_t - u_\nu}_E^q \dd t \right)^{1/q} &\leq \seminorm{u}_{C^{\gamma - 1/q}(J;E)} \left(\int_J \abs{t-\nu}^{\gamma q-1} \dd t \right)^{1/q} \\
    &\leq C_{\gamma,q} \seminorm{u}_{B^{\gamma}_{q,\infty}(J;E)}\abs{J}^\gamma.
\end{align*}
Since $\nu \in J$ was arbitrary,~\eqref{eq:Shift-bound} follows by taking the supremum.
\end{proof}

\begin{lemma}[Localisation of Nikolskii norm] \label{lem:local-Nikolskii}
Let $J = [a,b]$, $a<b$ be an interval, $\mathcal{P} = \{ [t_{k-1}, t_k]: \, t_0 = a < t_1  < \ldots < t_N = b\}$ be a finite partition of $J$, and $E$ be a Banach space. Then
\begin{align*}
    u \in  \bigcap_{k=1}^N B^{1/2}_{2,\infty}([t_{k-1},t_k];E) \cap L^\infty(J;E) \quad \Rightarrow \quad u \in B^{1/2}_{2,\infty}(J;E).
\end{align*}
Moreover,
\begin{align} \label{eq:Local-Nikolskii}
       \seminorm{u}_{B^{1/2}_{2,\infty}(J;E)}^2 \leq \sum_{k=1}^N  \seminorm{u}_{B^{1/2}_{2,\infty}([t_{k-1},t_k];E)}^2 + \norm{u}_{L^\infty(J;E)}^2 2(N-1+h_{\mathrm{min}}^{-1}\abs{J} ),
\end{align}
where $h_{\mathrm{min}} = \min_{k=1,\ldots,N}\abs{t_k - t_{k-1}}$.
\end{lemma}
\begin{proof}
We distinguish two cases: differences for small and large times. Here small and large times should be understood with respect to the minimal partition mesh-size~$h_{\mathrm{min}}$. 

\underline{$h \leq h_{\mathrm{min}}$:} For short times we can use the control provided by the local Nikolskii norm and an $L^\infty$-estimate. Notice that each $[t_{k-1},t_k]$ splits into $[t_{k-1}, t_{k} - h] \cup [t_{k} - h,t_{k}]$. Therefore,
\begin{align*}
    \int_a^{b-h} \norm{u_{t+h} - u_t}_E^2 \dd t &= \sum_{k=1}^N \int_{t_{k-1}}^{t_k-h} \norm{u_{t+h} - u_t}_E^2 \dd t +  \sum_{k=1}^{N-1} \int_{t_{k}-h}^{t_k} \norm{u_{t+h} - u_t}_E^2 \dd t \\
    &\leq h \left( \sum_{k=1}^N  \seminorm{u}_{B^{1/2}_{2,\infty}([t_{k-1},t_k];E)}^2 + \norm{u}_{L^\infty(J;E)}^2 2(N-1) \right) . 
\end{align*}
It follows 
\begin{align} \label{eq:small-h}
    \sup_{h \leq h_{\mathrm{min}}} h^{-1} \int_a^{b-h} \norm{u_{t+h} - u_t}_E^2 \dd t \leq  \sum_{k=1}^N  \seminorm{u}_{B^{1/2}_{2,\infty}([t_{k-1},t_k];E)}^2 + \norm{u}_{L^\infty(J;E)}^2 2(N-1).
\end{align}

\underline{$h > h_{\mathrm{min}}$:} For large times we no longer need to guarantee a decay of increments. Thus, it is sufficient to estimate
\begin{align*}
    h^{-1} \int_a^{b-h} \norm{u_{t+h} - u_t}_E^2 \dd t \leq  2h_{\mathrm{min}}^{-1} \norm{u}_{L^\infty(J;E)}^2 \abs{J},
\end{align*}
which implies
\begin{align} \label{eq:large-h}
    \sup_{h > h_{\mathrm{min}}} h^{-1} \int_a^{b-h} \norm{u_{t+h} - u_t}_E^2 \dd t \leq  2h_{\mathrm{min}}^{-1} \norm{u}_{L^\infty(J;E)}^2 \abs{J}.
\end{align}

Combining~\eqref{eq:small-h} and~\eqref{eq:large-h} yields~\eqref{eq:Local-Nikolskii}.
\end{proof}

\begin{lemma}
    \label{classical rules of calculus}
Let $Z\in B^{1/2}_{2, \infty}\mathbb{R}$ and $\beta\in C^\gamma_{t}\mathbb{R}$ with $\gamma>3/4$. Then 
\[
B_t:=\int_0^t Z_r \dd \beta_r-Z_t
\]
is a well-defined object in $B^{1/2}_{2, \infty}$ and moreover, we have 
\[
\int_0^t\exp{(-\beta_r)}\dd B_r=-Z_t\exp{(-\beta_t)}+Z_0\exp{(-\beta_0)}.
\]
\end{lemma}
\begin{proof}
    The first part of the statement is an immediate consequence of the Sewing Lemma \ref{sewing}. Note that if $\beta^n, Z^n$ are smooth, we have for the corresponding $B^n$
\begin{align*}
   & \int_0^t \exp{(-\beta^n_r)} \dd B^n_r=\int_0^t  \exp{(-\beta^n_r)}Z^n_r \dd \beta^n_r- \int_0^t\exp{(-\beta^n_r)}\dd Z^n_r\\
    &=-\int_0^t \dd (Z^n_r\exp{(-\beta^n_r)}) = -Z^n_t\exp{(-\beta^n_t)}+Z^n_0\exp{(-\beta^n_0)}.
\end{align*}
Using the stability of sewings under mollification as expressed in Lemma \ref{sewing convergence}, we are able to pass to the limit in the above, yielding the claim.
\end{proof}

\subsection{An Aubin Lions type Lemma}
In an abstract sense, upon establishing suitable a priori bounds on approximate problems, one is confronted with the problem of the passage to the limit. However, as only weak or weak-$*$ convergences are available from a priori bounds, the passage to the limit in non-linear terms typically poses an issue as this requires some form of strongly convergent subsequences. Such strong convergence is classically provided by the Rellich-Kondrachov theorem in combination with the Aubin-Lions Lemma. Remark however that due to the irregularity of $I(u)$, no $L^q([0, T], X)$ estimates on $\partial_t u^n$ will be available. We therefore establish an Aubin-Lions type Lemma suited for our context.
\begin{theorem}[\cite{aubinlions} Theorem A.1]
\label{aubingut}
    Let $p\in [1, \infty)$ and let $T\in (0, \infty)$. Let $X$ be a Banach space. A set $F\subset L^p([0, T], X)$ is relatively compact in $L^p([0, T], X)$ if and only if 
    \begin{enumerate}
        \item $\sup_{f\in F} \norm{f}_{L^p([0, T], X)}<\infty$
        \item $\lim_{h\to 0}\int_0^{T-h} \norm{f_{r+h}-f_r}_X^pdr=0$ uniformly in $f\in F$
        \item For every $\epsilon>0$, there exists a compact set $Q_\epsilon\subset X$ such that for every $f\in F$ there exists a set $A_{f, \epsilon}\subset [0, T]$ with $\lambda([0, T]\backslash A_{f, \epsilon})\leq \epsilon$ and 
        \[
        f_t\in Q_\epsilon, \quad \forall t\in A_{f, \epsilon}
        \]
    \end{enumerate}
\end{theorem}
Given the above result, we may conclude the following. 
\begin{theorem}
\label{nikoslki-lions}
    Let $X_0, X$ be Banach spaces, such that the embedding $X_0\hookrightarrow \hookrightarrow X$ is compact. Let $p, q\in [1, \infty)$ and $\gamma>0$. Then the space 
    \[
    \mathcal{Y}:=L^q_tX_0\cap B^{\gamma}_{p, \infty}X
    \]
    embeds compactly into $L^p_tX$. 
\end{theorem}
\begin{proof}
    We apply Theorem \ref{aubingut}. Points $(1)$ and $(2)$ thereof follow immediately as 
    \[
\int_0^{T-h} \norm{f_{r+h}-f_r}_X^pdr\leq h^{\gamma p} \norm{f}_{B^\gamma_{p, \infty}X}\to 0. 
    \]
    Let $(f^n)_n$ be a bounded sequence in $\mathcal{Y}$. Given $\epsilon>0$, we choose $R(\epsilon)$ sufficiently large such that 
    \[
    \frac{1}{R(\epsilon)^q}\sup_n \norm{f^n}_{L^q_tX_0}^q<\epsilon
    \]
    Further set $A_{n, \epsilon}:=\{ t\in [0, T]\ |\ \norm{f_n}_{X_0}\geq R(\epsilon)\}$ and for the canonical imbedding $i:X_0\to X$, set $Q_\epsilon:=\overline{i(B_{X_0}(0, R(\epsilon)))}$, which is a compact set in $X$ such that $f^n_t\in Q_\epsilon$ for all $t\in A_{n, \epsilon}$. Finally, by Markov's inequality we have 
    \[
    \lambda([0, T]\backslash A_{n, \epsilon})\leq \frac{1}{R(\epsilon)^q}\int_0^T \norm{f^n_r}_{X_0}^qdr<\epsilon.
    \]
    This establishes $(3)$, thus allowing us to apply Theorem \ref{aubingut} concluding the proof. 
\end{proof}

\section*{Acknowledgement} The first authors warmly thanks the organizers and participants of the "GEN-Y research workshop in stochastic analysis" at the University of Münster for feedback on a preliminary version of this work.  The first author acknowledges support from the European Research Council (ERC) under the European Union’s Horizon 2020 research
and innovation programme (grant agreement No. 949981). The second author acknowledges support from the Australian Government through the Australian Research Council’s Discovery Projects funding scheme (grant number DP220100937).

\textbf{Declarations: }
All the authors declare that they have no conflicts of interest. Data availability statement is not applicable in the context to the present article since all the results here are theoretical in nature and do not involve any data.

\bibliographystyle{alpha}
\bibliography{main}

\end{document}